\documentclass[draft,reqno]{amsart}

\usepackage{ upgreek }

\usepackage{amssymb}
\usepackage[leqno]{amsmath}
\usepackage{amsthm}
\usepackage{epic}
\usepackage{accents}
\usepackage{setspace}

\makeatletter
\newcommand{\leqnomode}{\tagsleft@true\let\veqno\@@leqno}
\newcommand{\reqnomode}{\tagsleft@false\let\veqno\@@eqno}
\makeatother

\newtheorem{theorem}{Theorem}[section]
\newtheorem{lemma}[theorem]{Lemma}

\newtheorem{proposition}[theorem]{Proposition}

\newtheorem{corollary}[theorem]{Corollary}

\input{xy}
\xyoption{matrix}
\xyoption{arrow}

\def \C{\mathbb{C}}
\def \F{\mathcal{F}}

\def \Z{\mathbb{Z}}

\def \N{\mathbb{N}}

\def \R{\mathbb{R}}
\def \O{\mathcal{O}}
\def \L{\mathcal{L}}
\def \cM{\mathcal{M}}
\def \Lrc{\mathcal{L}^{\mathrm{c}}_{\mathrm{r}}}
\def \Lv{\mathcal{L}_{\mathrm{v}}}
\def \Lr{\mathcal{L}_{\mathrm{r}}}
\def \Lvg{\mathcal{L}^{\mathrm{g}}_{\mathrm{v}}}
\def \Acg{A\mathrm{cg}}
\def \co{\operatorname{co}}

\def \res{\operatorname{res}}
\def \k {\boldsymbol{k}}
\def \<{\langle}
\def \>{\rangle}
\def \a{\operatorname{a}}

\def \fm{\mathfrak{m}}
\def \Frac{\operatorname{Frac}}
\def \ch{\operatorname{char}}
\def \supp{\operatorname{supp}}
\def \alg{\operatorname{alg}}
\def \qftp{\operatorname{qftp}}
\def \tp{\operatorname{tp}}

\def \K{\mathcal{K}}
\def \E{\mathcal{E}}
\def \A{\mathcal{A}}

\DeclareFontFamily{OMS}{smallo}{}
\DeclareFontShape{OMS}{smallo}{m}{n}{<->s*[.65]cmsy10}{}
\DeclareSymbolFont{smallo@m}{OMS}{smallo}{m}{n}
\DeclareMathSymbol{\smallo}{\mathord}{smallo@m}{79}

\copyrightinfo{}{}
\issueinfo{}
  {}
  {}
  {}

\PII{}
\def\@serieslogo{\@empty}

\title[Analytic Ax-Kochen-Ersov theory]{Analytic Ax-Kochen-Ersov theory with lifts of the residue field and value group}

\author[Bhardwaj]{Neer Bhardwaj}
\address{Weizmann Institute of Science, Rehovot, Israel}
\email{nbhardwaj@msri.org}

\author[van den Dries]{Lou van den Dries$^1$}
\thanks{$^1$The authors thank the Fields Institute and the Simons Laufer Mathematical Institute for hospitality during the preparation of the paper.}
\address{Department of Mathematics\\
University of Illinois at Urbana-Cham\-paign\\
Urbana, IL 61801\\
U.S.A.}
\email{vddries@illinois.edu}
\begin{document}

\begin{abstract} We develop an extension theory for {\em analytic} valuation rings in order  to establish Ax-Kochen-Ersov type results for these structures. New is that we can add in salient cases lifts of the residue field and the value group and show that the induced structure on the lifted residue field is just its field structure, and on the lifted value group is just its ordered abelian group structure. This restores an analogy with the non-analytic AKE-setting that was missing in earlier treatments of analytic AKE-theory. 
\end{abstract}


\maketitle

{
    \setcounter{tocdepth}{1}
    \tableofcontents

}

\section{Introduction}

\noindent
In the 1960s Ax and Kochen \cite{AxK1,AxK2,AxK3}, and Ersov \cite{Er1, Er2, Er3, Er4} independently, developed a model theory for henselian valuation rings with significant applications to $p$-adic number theory. Since then there have been many  generalizations and refinements, and AKE-theory remains a very active area of research. For example, in the 1980s Denef and van den Dries \cite{DevdD, vdD} treated the ring of $p$-adic integers with {\em analytic\/} structure given by (restricted) power series. This led to the solution of a problem posed by Serre \cite{S}, and to a theory of $p$-adic subanalytic sets. Using ``separated" power series this was  upgraded to a theory of rigid subanalytic sets over henselian valuation rings equipped with a richer analytic structure,  by L.~Lipshitz, Z.~Robinson, R. Cluckers, see \cite{L,LR, CLR, CL, LC2}. 

An interesting part of the original AKE-theory has so far not been extended to this analytic setting: in the equal characteristic 0 case one can add a predicate for a {\em coefficient field} (a lift of the residue field to the ambient field), and then the structure induced on this coefficient field can be shown to be just its pure field structure; likewise for a {\em monomial group}, that is, a lift of the value group. 

In the analytic setting, there is only a partial result in this direction by Binyamini, Cluckers and Novikov \cite[Proposition 2]{BCN}, and the usual approaches to analytic AKE-theory---based on direct reductions to ordinary AKE-theory by Weierstrass division ``with parameters”---cannot be adapted to cover fully the induced structure aspect, as far as we know. Their partial result inspired us to try another approach. 

We do indeed obtain the expected induced structure results in an analytic setting by developing a theory of analytic valuation rings in closer analogy with ordinary valuation theory. Weierstrass division is still key, as in \cite{DevdD, vdD},  but now in a different way.  In an earlier version of the present paper, now in \cite{Neer},  this was done by elaborating, generalizing, and cleaning up substantial parts of \cite{DHM}. The cleaning up was necessary because we noticed problems with  \cite[Lemma 3.1]{DHM}, and to remedy it we had to pass to finite extensions, with additional complications. We found subsequently that a somewhat different and more general approach to similar issues
was already available in \cite{CLR, CL}. So in our Section 9 we quote and slightly adapt instead  relevant results from those sources, saving some 15 pages compared to this earlier version. 
Much of our analytic valuation theory is characteristic-free, but for the analytic AKE-results in Section~\ref{A-AKE} we require that the valued field be of equicharacteristic $0$. 

\subsection*{More on induced structure} Here we state in detail a typical case of our result on induced structure. First we say what it is in the classical (non-analytic) setting.  Let $C$ be a (coefficient) field. 
This yields the valuation ring $C[[t]]$ of formal power series in one variable $t$ over $C$. We now expand the ring $C[[t]]$
to the structure $\big(C[[t]],  C\big)$: a ring with a distinguished subset.  Then a classical ``induced structure" result is that if
$\operatorname{char} C =0$, any set $X\subseteq C^n$ which is definable in  $\big(C[[t]], C\big)$  is even definable in the field $C$. (This can be proved along familiar lines, so we consider it as folklore knowledge, though we do not know an explicit reference.
It seems this is still open for $\operatorname{char} C >0$.) Here and below, $n$ ranges over $\N=\{0,1,2,\dots\}$ and ``definable" means ``definable with parameters from the ambient structure''.

We now equip $C[[t]]$ with analytic structure as follows: for each $n$ we have the (Tate) ring $A\<Y_1,\dots, Y_n\>$ of restricted power series in the distinct indeterminates $Y_1,\dots, Y_n$ over $A=C[[t]]$: it consists of the formal power series
$$f\ =\ f(Y_1,\dots, Y_n)\ =\ \sum_{\nu} a_\nu Y_1^{\nu_1}\cdots Y^{\nu_n}, \qquad \nu=(\nu_1,\dots, \nu_n) \text{ ranging over }\ \N^n,$$ 
with all $a_{\nu}\in A$ such that
$a_{\nu}\to 0$, $t$-adically, as $|\nu|=\nu_1+\cdots + \nu_n\to \infty$. Each such $f$ gives rise to an $n$-ary operation on $C[[t]]$, namely
$$y\ =\ (y_1,\dots, y_n)\mapsto f(y_1,\dots, y_n)\ :\  C[[t]]^n\to C[[t]].$$ 
We expand the ring $C[[t]]$ to $C[[t]]_{\text{an}}$ by taking each such $f$ as a new $n$-ary function symbol that names the above $n$-ary operation on $C[[t]]$. Further expansion yields the structure $\big(C[[t]]_{\text{an}}, C\big)$, and now our new induced structure result says that any set $X\subseteq C^n$ which is definable in  $\big(C[[t]]_{\text{an}}, C\big)$  is even definable in the field $C$.  (For example, any subset of $\C$ definable in $\big(\C[[t]]_{\text{an}}, \C\big)$
is finite or its complement in $\C$ is finite.) In fact, our induced structure result, Corollary~\ref{indstr}, is stronger and more general in several ways, for example in also allowing $t^{\N}$ as a distinguished subset of $C[[t]]$. For various reasons it is more convenient to take the fraction field $C(\!(t)\!)$ of $C[[t]]$ as the ambient ring, equipped with its natural valuation to recover $C[[t]]$. For $C=\C$
we obtain \cite[Proposition 2]{BCN} as a special case, as explained in Section~\ref{A-AKE}. 

\subsection*{Contents of Sections 2--11} We begin with a brief section on henselianity. Next a section on ultranormed
rings and restricted power series over them, including the Weierstrass theorems. At the beginning of Section~\ref{ras} we define for any complete ultranormed ring $A$ subject to mild conditions the notion of {\em $A$-analytic ring}: each $n$-variable restricted power series over $A$ yields an $n$-ary operation on any $A$-analytic ring.  
Starting in Section~\ref{noethA} we specialize to the case that $A$ is noetherian with an ideal $\smallo(A)\ne A$, such that $\bigcap_{n}\smallo(A)^n=\{0\}$ and $A$ is $\smallo(A)$-adically complete.  

In Section~\ref{vrna} we define $A$-analytic {\em valuation\/} rings and establish basic facts about them. In Section~\ref{immAE} 
we treat immediate extensions and prove an analytic version of Kaplansky's embedding theorem. In Section~\ref{trunc} (not needed for  ``induced structure" results, but included for its independent interest) we apply this to show that various extension procedures preserve truncation closedness.  
Section~\ref{division} uses  \cite{CL,CLR} to describe the function given by a univariate term in the language of analytic valued fields as analytic on the annuli of a suitable finite covering of the valuation ring. This allows us to complete the full array of extension results.  Then we can prove in Section~\ref{A-AKE} an analytic AKE-type equivalence theorem, with an induced structure result for ``coefficient field + monomial group'' as a 
consequence. Together with our work on immediate extensions in Section~\ref{immAE}, this yields {\em NIP transfer} in our analytic context. Section~\ref{req} proves a ``separation of variables'' result.

\subsection*{Notational and terminological conventions} 
Throughout  $d,m,n$ range over
$ \N=\{0,1,2,\ldots\}$;  {\em ring\/} means {\em commutative ring with $1$}. From Section~\ref{vrna} onwards we consider valued fields. Let $K$ be a {\em valued\/}  field; it is specified by a valuation ring $R$
of the field $K$.  Let  $v: K^\times \to \Gamma$ be a valuation
on $K$ with $R=\{a\in K: v a\geqslant 0\}$. Here $\Gamma=v(K^\times)$ is the (ordered) value group, and we extend $v$ to a
function
 $v: K \to \Gamma_{\infty}=\Gamma\cup\{\infty\}$ by setting $v(0):=\infty$ and we extend the total ordering of $\Gamma$ to a total ordering on $\Gamma_{\infty}$ by $\Gamma < \infty$. 
It will be convenient to let $\preccurlyeq$, $\asymp$, $\prec$, $\succcurlyeq$, $\succ$, and $\sim$
denote the binary relations on $K$ given for $x,y\in K$ by
\begin{align*}  x\preccurlyeq y\ &: \Leftrightarrow\ vx\geqslant vy\ \Leftrightarrow\ x=yz \text{ for some $z\in R$},\\
 x\asymp y\ &:\Leftrightarrow\ x\preccurlyeq y \text{ and } y\preccurlyeq x, \qquad x\prec y\ :\Leftrightarrow\ x\preccurlyeq y \text{ and }x\not\asymp y,\\
x\succcurlyeq y\ &:\Leftrightarrow\ y\preccurlyeq x, \quad x\succ y\ :\Leftrightarrow\ y\prec x,\quad x\sim y\ \Leftrightarrow\ x-y\prec x.
\end{align*}  
We let $\smallo(R)$ be the maximal ideal of $R$, and let $\res K := R/\smallo(R)$ be the residue field. For $a\in R$ we let $\res a$ be the residue class of $a$ in $\res K$. If we need to indicate dependence on $K$ we write $R_K$, $v_K$, $\Gamma_K$ instead of $R, v, \Gamma$. The reason we use the letter $R$ here instead of the more common $\O$ is that
in Section~\ref{division} we follow \cite{FvdP} in denoting the algebra of analytic functions on a suitable set $F$ by $\O(F)$;
see the start of Section~\ref{division} for context and definitions of these notions.  

Model theoretic arguments become important in Sections~\ref{division}  and ~\ref{A-AKE}, although in earlier sections we already construe various mathematical structures as $L$-structures for various first-order languages $L$. We deal only with one-sorted structures, and 
``$\mathcal{M}\subseteq \mathcal{N}$''  indicates that $\mathcal{M}$ is a substructure of $\mathcal{N}$, for $L$-structures $\mathcal{M}$ and $\mathcal{N}$. (One exception: we refer to a $3$-sorted structure from~\cite{BCN} in 
Section~\ref{req}.)

\medskip
\noindent
 We cite many results of classical AKE-theory from the exposition \cite{Lou}. We do so for convenience and do not suggest that the cited facts originate with \cite{Lou}.\footnote{We take the opportunity to mention that the Chinese Remainder Theorem on \cite[p. 86]{Lou} is stated there in too great a generality. But it holds for ideals in a ring, which is how it gets applied.}

\section{Henselianity} 

\noindent
There are a few places where we need ``henselianity'' outside the usual pattern of a henselian local ring. Accordingly, this section proves basic facts about  henselian pairs (which generalize henselian local rings).  These facts are well-known, but our treatment is more elementary than we have seen in the literature. 

\medskip\noindent
 Given a ring $R$ we let $R^\times$ denote the multiplicative group of units of $R$. The {\em Jacobson radical\/} of a ring $R$ is the
intersection of the maximal ideals of $R$. For the Jacobson radical $J$ of $R$, if $a\in R$ and
$a+J\in (R/J)^\times$, then $a\in R^\times$.
\noindent In this section $X$ and $Y$ are distinct indeterminates and $I$ is an ideal of the ring $R$.

\begin{lemma} Let
$I$ be contained in the Jacobson radical of $R$ and let $P(X)\in R[X]$ and $a\in R$ be such that $P'(a)\in R^\times$. Then $P(b)=0$ for at most one $b\in a+I$.
\end{lemma}
\begin{proof} Let $b\in a+I$ and $P(b)=0$. Then for $\epsilon\in I$ we have $r\in R$ such that
 $$P(b+\epsilon)\ =\ P(b) + P'(b)\epsilon + r\epsilon^2\ =\ P'(b)\epsilon + r\epsilon^2\ =\ 
 P'(b)\epsilon\big(1+rP'(b)^{-1}\epsilon\big)\ =\ 0,$$ 
 and $P'(b), 1+rP'(b)^{-1}\epsilon\in  R^\times$, so $\epsilon=0$.
 \end{proof}

\medskip\noindent
 {\em The pair $(R,I)$ is henselian\/} means: \begin{itemize}
\item $I$ is contained in the Jacobson radical of $R$, equivalently, $1+I\subseteq R^\times$;
\item for all polynomials $P(X)\in R[X]$ and $a\in R$ with $P(a)\in I$ and $P'(a)\in R^\times$ there exists $b\in R$ such that $P(b)=0$ and $a-b\in I$.
\end{itemize}

\noindent
Thus given a maximal ideal $\fm$ of the ring $R$, the pair $(R,\fm)$ is henselian iff
$R$ is a henselian local ring in the usual sense. 

\begin{lemma}\label{he1} Assume $1+I\subseteq R^\times$. Then the following conditions are equivalent: \begin{enumerate}
\item[(i)] $(R,I)$ is henselian;
\item[(ii)] each polynomial $1+X + ea_2X^2 + \cdots + ea_nX^n$ with $n\geqslant 2$, $e\in I$, and $a_2,\dots, a_n\in R$ has a zero in $R$ $($obviously, such a zero lies in $-1+I)$;
\item[(iii)] each polynomial $Y^n + Y^{n-1} + ea_2Y^{n-2} + \cdots + ea_n$ with $n\geqslant 2$, $e\in I$, and
$a_2,\dots, a_n\in R$ has a zero in $R^\times$;
\item[(iv)] given any polynomial $P(X)\in R[X]$ and $a\in R$, $e\in I$ such that $P(a)=eP'(a)^2$ 
there exists $b\in R$ such that $P(b)=0$ and $b-a\in eP'(a)R$. 
\end{enumerate}
\end{lemma}
\begin{proof} (i)$\Rightarrow$(ii) is clear. For (ii)$\Leftrightarrow$ (iii): use that for $x\in R^\times$ and $y:= x^{-1}$,
$x$ is a zero in (ii) iff  $y$ is a zero in (iii). Now assume (ii) and let $P, a, e$ be as in the hypothesis of (iv).  
Let $x\in R$ and consider the expansion:
\begin{align*}
P(a+x)\ 	&=\ P(a)+P'(a)x+\sum_{i\geqslant 2} P_{(i)}(a)x^i \\
			&=\ eP'(a)^2 + P'(a)x+\sum_{i\geqslant 2} P_{(i)}(a)x^i.
\end{align*} 
Set $x=eP'(a)y$ where $y\in R$. Then
$$P(a+x)\ =\ eP'(a)^2 \left(1+y+\sum_{i\geqslant 2} ea_i y^i\right)$$
where the $a_i\in R$ do not depend on $y$. From (ii) we obtain $y\in R$ such that
$$1+y+\sum_{i\geqslant 2} ea_iy^i\ =\ 0.$$
This yields an element $b=a+x=a+eP'(a)y$ as required. This shows (ii)~$\Rightarrow$~(iv), and (iv)~$\Rightarrow$~(i) is clear.
\end{proof}

\begin{lemma}\label{he2} Suppose every element of $I$ is nilpotent. Then $(R,I)$ is henselian.
\end{lemma}
\begin{proof} Consider a polynomial $P(X)=a+X+\sum_{i=2}^nea_iX^i$ where $n\geqslant 2$ and
$$a,e, a_2,\dots, a_n\in R, \quad e^m=0,\ m\geqslant 1.$$ 
By induction on $m$ we show that $P(X)$ has a zero in $R$. The case $m=1$ being trivial, let $m\geqslant 2$. Then 
$$P(-a+eY)\ =\ a+ (-a+eY) + \sum_{i=2}^nea_i(-a+eY)^i\ =\ 
e\big(Y+\sum_{i=2}^na_i(-a+eY)^i\big).$$
An easy computation gives $f, b, b_2,\dots, b_n\in R$ such that
$$Y+\sum_{i=2}^na_i(-a+eY)^i\ =\ b+ Y(1+ef) + \sum_{i=2}^ne^2b_iY^i.$$ 
Now use that $1+ef\in R^\times$ and $(e^2)^{m-1}=0$. 
\end{proof}

\begin{lemma}\label{he3} Let $J$ be an ideal of $R$ with $I\subseteq J$. Then the following are equivalent:
\begin{enumerate}
\item[(i)] $(R,I)$ and $(R/I, J/I)$ are henselian;
\item[(ii)] $(R,J)$ is henselian.
\end{enumerate}
\end{lemma}
\begin{proof} The condition $1+J\subseteq R^\times$ is easily seen to be equivalent to the conjunction of
$1+I\subseteq R^\times$ and $1+(J/I)\subseteq (R/I)^\times$. This gives (ii)$\Rightarrow$(i).
Now assume (i), and let $P(X)\in R[X]$ and $a\in R$ with $P(a)\in J,\ P'(a)\in R^\times$.
Working modulo $I$ this gives $b\in R$ such that $P(b)\in I$ and $a-b\in J$. Hence
$P'(b)-P'(a)\in J$, and thus $P'(b)\in R^\times$, giving $c\in R$ with
$P(c)=0$ and $b-c\in I$. Hence $a-c\in J$. 
\end{proof}

\begin{corollary} Suppose $(R,I)$ is henselian and $J$ is an ideal of $R$ contained in the nilradical $\sqrt{I}$ of $I$. Then $(R,J)$ is henselian. 
\end{corollary}
\begin{proof} Every element of $\sqrt{I}/I$ is nilpotent in $R/I$, so by Lemmas~\ref{he2} and ~\ref{he3} the pair $(R,\sqrt{I})$ is henselian, and so is $(R,J)$. \end{proof}

\noindent
Recall also that a local ring $R$ is said to be {\em henselian\/} if the pair $(R,\fm)$ is henselian, where $\fm$ is the maximal ideal of $R$.

\section{Complete Ultranormed Rings and Restricted Power Series}\label{ASR}

\noindent
We introduce here the restricted power series that will define operations on the valuation rings considered in later sections,
where we develop an AKE-theory for these valuation rings with these extra operations. The coefficients of these
restricted power series will be from a fixed coefficient ring $A$  which is complete with respect to an ultranorm.
We begin with defining ultranorms.

\subsection*{Ultranormed abelian groups} Let $A$ be an additively written abelian group. An {\em ultranorm\/} on $A$ is a function $a\mapsto |a|: A \to \R^{\geqslant}$ such that for all $a,b\in A$, \begin{itemize}
\item $|a|=0\Leftrightarrow a=0$;
\item $|-a|=|a|$;
\item $|a+b|\leqslant \max(|a|,|b|)$.
\end{itemize}
Let $A$ be equipped with the ultranorm $|\cdot|$ on $A$. We make $A$ a metric space
with metric $(a,b)\mapsto |a-b|$. Then $A$ is a topological group with respect
to the topology on $A$ induced by this metric. The ultranorm $|\cdot|: A \to \R$ and the group operations $-: A \to A$ and $+: A\times A \to A$ are uniformly continuous. 

{\em In the rest of this subsection $A$ is complete with respect to its ultranorm, that is, complete with respect to the metric above.} We now discuss convergence of series with terms in $A$.  Let $(a_i)=(a_i)_{i\in I}$ be a family in $A$ (that is, all $a_i\in A$). We say $(a_i)$ is {\em summable\/} if for every $\epsilon$ we have $|a_i|<\epsilon$ for all but finitely many $i\in I$.
In that case the set of $i\in I$ with $a_i\ne 0$ is countable, and there is a unique $a\in A$ such that for every $\epsilon\in \R^{>}$ there is a finite
$I(\epsilon)\subseteq I$ with $|a-\sum_{i\in J} a_i|<\epsilon$ for all finite
$J\subseteq I$ with $I(\epsilon)\subseteq J$; this $a$ is then denoted by
$\sum_{i\in I} a_i$ (or $\sum_i a_i$ if $I$ is understood from the context).
Instead of saying that $(a_i)$ is summable we also say that
$\sum_i a_i$ exists, or that $\sum_i a_i$ converges. Of course, if $I$ is finite, then $\sum_i a_i$ exists and is the usual sum.  Here are simple rules, used throughout,
for dealing with such (possibly infinite) sums, where $(a_i)_{i\in I}$ is
a summable family in $A$: \begin{itemize}
\item if $c\in \R^{>}$ and $|a_i|\leqslant c$ for all $i$, then $|\sum_i a_i|\leqslant c$;
\item $(-a_i)$ is summable with $\sum_i -a_i=-\sum_i a_i$; 
\item if $(b_i)_{i\in I}$ is also a summable family in $A$, then so is 
$(a_i+b_i)$ with $$\sum_i a_i+b_i\ =\ \sum_i a_i + \sum_i b_i;$$
\item if $i\mapsto \lambda(i): I\to \Lambda$ is a bijection and $(b_{\lambda})_{\lambda\in \Lambda}$ is a family in $A$ with $a_i=b_{\lambda(i)}$ for all $i\in I$, then $\sum_{\lambda} b_{\lambda}$ exists and equals $\sum_i a_i$;
\item if the family $(a_j)_{j\in J}$ in $A$ is also summable with $I\cap J=\emptyset$, then $(a_k)_{k\in I\cup J}$ is summable with $\sum_k a_k=\sum_i a_i + \sum_j a_j$; 
\item if $I=\dot{\bigcup}_{\lambda\in \Lambda} I_{\lambda}$ (disjoint union), then $\sum_{i\in I_{\lambda}}a_i$ exists for all $\lambda\in \Lambda$, and $\sum_{\lambda}\big(\sum_{i\in I_{\lambda}}a_i\big)$ exists and equals $\sum_{i\in I} a_i$. 
\end{itemize} 
Suppose $E$ is a closed subgroup of $A$. Then
$$|a+E|\ :=\ \inf_{e\in E}|a+e|\qquad(a\in A)$$
yields an ultranorm on the quotient group $A/E$ with respect to which $A/E$ is complete; 
we call it the {\em quotient norm of $A/E$}. If the family $(a_i)$ in $A$ is summable, then so is the family
$(a_i+E)$ in $A/E$ with its quotient norm, and $$\big(\sum_i a_i\big)+E\ =\ \sum_i (a_i+E).$$

\subsection*{Ultranormed rings}   
Let $A$ be a ring. An {\em ultranorm on $A$} is a function $$a\mapsto |a|\ :\ A \to \R^{\geqslant}$$ such that for all $a,b\in A$, \begin{itemize}
\item $|a|=0\Leftrightarrow a=0$, $|1|=|-1|=1$;
\item $|a+b|\leqslant \max(|a|,|b|)$;
\item $|ab|\leqslant |a|\cdot |b|$.
\end{itemize}
Let $A$ be equipped with the ultranorm $|\cdot|$ on $A$. Then $|-a|=|a|$ for all
$a\in A$, so $|\cdot|$ is an ultranorm on the underlying additive group of $A$.
The function $\cdot: A\times A \to A$ is continuous.
If $A$ is complete with respect to its ultranorm and $(a_i)_{i\in I}$ and $(b_j)_{j\in J}$ are summable families in $A$, then $(a_ib_j)_{(i,j)\in I\times J}$ is summable, with $(\sum_i a_i)(\sum_j b_j)=\sum_{(i,j)}a_ib_j$.  

\medskip\noindent
{\bf From now on in this paper $A$ is a ring with $1\ne 0$, equipped with an ultranorm 
$|\cdot|$ such that $|a|\leqslant 1$ for all $a\in A$, and $A$ is complete with respect to its ultranorm}.

\medskip\noindent
It follows that if $a\in A$ and $|a|<1$, then $\sum_n a^n$ exists, with $$(1-a)\sum_n a^n\ =\ 1.  $$
We have the ideal $\smallo(A):=\{a\in A:\ |a|<1\}$, and set $\overline{A}:= A/\smallo(A)$, with the canonical ring morphism
$a\mapsto \overline{a}=a+\smallo(A): A\to \overline{A}$. 
We saw that $1+\smallo(A)$ consists entirely of units of $A$. Thus
$a\in A$ is a unit of $A$ iff $\overline{a}$ is a unit of $\overline{A}$.
In particular, $\smallo(A)$ is contained in the Jacobson radical of $A$. 
The completeness assumption now yields Hensel's Lemma as stated 
in \cite[Section 2.2]{Lou}: the pair $(A, \smallo(A))$ is henselian. It follows that $(A, \sqrt{\smallo(A)})$ is also henselian.

\subsection*{Passing to $A/I$}
Suppose the proper ideal $I$ of $A$ is closed. Then the quotient norm of
the quotient group $A/I$ is an ultranorm on the ring $A/I$.  Equipping $A/I$ with the quotient norm, the canonical map $A\to A/I$ is norm decreasing, and $\smallo(A/I)$ is the image of $\smallo(A)$ under this canonical map. 

If $A$ is noetherian, then by \cite[Theorem 8.14]{Mat} every ideal of $A$ is closed.

\subsection*{Restricted power series over an ultranormed ring}
For distinct indeterminates $Y_1,\dots, Y_n$ we let $A\<Y\>=A\<Y_1,\dots, Y_n\>$ be the subalgebra of the $A$-algebra $A[[Y_1,\dots, Y_n]]$ consisting of the series $\sum_{\nu}a_{\nu}Y^\nu$ with
$a_{\nu}\to 0$ as $|\nu|\to \infty$. Here and below, when using an expression like
$\sum_{\nu}a_{\nu}Y^{\nu}$ for a series in $A\<Y\>$ it is assumed that $a_{\nu}\to 0$ as $|\nu|\to \infty$.
We extend $|\cdot|$ on $A$ to an ultranorm on the ring $A\<Y\>$ by
$$ |\sum_{\nu}a_{\nu}Y^{\nu}|\ :=\ \max_{\nu}|a_{\nu}|,$$
so with respect to this ultranorm, $A\<Y\>$ is complete and $A[Y]$ is dense in it. Note that for
$a_1,\dots, a_n\in \smallo(A)$ we have $|a_1Y_1+\cdots + a_nY_n|<1$, so $1+a_1Y_1+ \cdots + a_nY_n$ is a unit of the ring $A\<Y\>$.

For $f=\sum_{\nu} a_\nu Y^\nu\in A\<Y\>$, the family $(a_{\nu}Y^{\nu})$ in $A\<Y\>$ is in fact summable with sum $f$.  
If $|ab|=|a|\cdot |b|$ for all $a,b\in A$, then $|fg|=|f|\cdot |g|$ for all $f,g\in A\<Y\>$. For any $y=(y_1,\dots, y_n)\in A^n$ we have the evaluation map 
$$f=\sum_{\nu}a_{\nu}Y^\nu\mapsto f(y):= \sum_{\nu}a_{\nu}y^{\nu}\ :\  A\<Y\>\to A,$$ which is an $A$-algebra morphism with $|f(y)|\leqslant |f|$ for all $y\in A^n$. 
If $(f_i)_{i\in I}$ is a summable family in $A\<Y\>$ and $y\in A^n$, then
$\sum_i f_i(y)$ exists in $A$ and equals $(\sum_i f_i)(y)$.
  The obvious inclusion of $A[[Y_1,\dots, Y_m]]$ in
$A[[Y_1,\dots, Y_n]]$ for $m\leqslant n$ restricts to an inclusion of
$A\<Y_1,\dots, Y_m\>$ in $A\<Y_1,\dots, Y_n\>$.  
For  $f=f(Y)\in A\<Y\>$ we have unique $f_j\in A\<Y_1,\dots, Y_j\>$ for $j=0,\dots, n$ such that
$$f(Y)\ =\ f_0+ Y_1f_1+ \cdots + Y_nf_n.$$

\subsection*{Substitution}\label{substitution} Besides $Y=(Y_1,\dots, Y_n)$, let $X=(X_1,\dots, X_m)$ also be a tuple of distinct indeterminates. Let $f=\sum_{\mu}a_{\mu}X^\mu\in A\<X\>$
with $\mu=(\mu_1,\dots, \mu_m)$ ranging over $\N^m$, and $g_1,\dots, g_m\in A\<Y\>$. Then 
$|a_{\mu}g_1^{\mu_1}\cdots g_m^{\mu_m}|\leqslant |a_{\mu}|\to 0 \text{ as }|\mu|\to \infty$, so
$$f(g_1,\dots, g_m)\ :=\ \sum_{\mu} a_{\mu}g_1^{\mu_1}\cdots g_m^{\mu_m}\in A\<Y\>,$$
and for fixed $g=(g_1,\dots, g_m)\in A\<Y\>^m$ the map $f\mapsto f(g): A\<X\>\to A\<Y\>$ is an $A$-algebra morphism with $|f(g)|\leqslant |f|$ and 
$f(g)(y)=f(g(y))$ for $y\in A^n$. Moreover, if $(f_i)$ is a summable family in $A\<X\>$ and $g\in A\<Y\>^m$, then $\sum_i f_i(g)$ exists in $A\<Y\>$ and equals $(\sum_i f_i)(g)$. It follows that the above kind of composition is associative in the following sense: let $Z=(Z_1,\dots, Z_p)$ be a third tuple
of distinct indeterminates, $p\in \N$, and $h=(h_1,\dots, h_n)\in A\<Z\>^n$.
Then $$\big(f(g)\big)(h)\ =\ f\big(g_1(h),\dots,g_m(h)\big) \text{ in }A\<Z\>.$$ 

\medskip\noindent
From now on $X_1, X_2, X_3,\dots, Y_1,Y_2,Y_3,\dots$ (two infinite sequences) are distinct indeterminates, and unless specified otherwise, 
$$X:=(X_1,\dots, X_m), \qquad Y:= (Y_1,\dots, Y_n).$$
 The natural $A[[X]]$-algebra isomorphism
$A[[X]][[Y]]\to A[[X,Y]]$ restricts to
the norm preserving $A\<X\>$-algebra isomorphism $A\<X\>\<Y\>\to A\<X,Y\>$ given by
 $$\sum_{\nu} f_{\nu}Y^\nu \mapsto \sum_{\nu}f_{\nu}Y^\nu$$
where $f_{\nu}\in A\<X\>$ for all $\nu$ and $f_{\nu}\to 0$ as $|\nu|\to \infty$, with righthand and lefthand side interpreted naturally in $A\<X\>\<Y\>$ and $A\<X,Y\>$ respectively. We identify $A\<X\>\<Y\>$ and $A\<X,Y\>$ via this isomorphism. 

\subsection*{Polynomials as restricted power series}
Let $p=p(T)\in A[T]$ be a monic polynomial of degree $d\geqslant 1$ over $A$, so $|p|=1$ as an element of  $A\<T\>$. 

\begin{lemma}\label{pff} For all $f\in A\<T\>$ we have $|pf|=|f|$. Moreover,  $pA\<T\>$ is a proper ideal of $A\<T\>$ and is closed in $A\<T\>$.
\end{lemma}
\begin{proof} For $f=\sum_n a_nT^n\in A[T]^{\ne}$, take $n$ maximal with $|a_n|=|f|$, and note that then
the coefficient of $T^{d+n}$ in $pf$ is $a_n+ b$ with $|b|< |a_n|$, so $|a_n+b|=|a_n|=|f|$. 
The rest follows easily. 
\end{proof}

\begin{lemma} \label{fqpr} Let $f\in A\<T\>$. Then there are unique $q\in A\<T\>$ and $r\in A[T]$ with $\deg r < d$ such that $f=qp+r$; moreover, $|f|=\max(|q|, |r|)$ for these $q,r$.
\end{lemma}
\begin{proof} For each $n$ we have $T^n=q_np+r_n$ with $q_n, r_n\in A[T]$ and $\deg r_n < d$. Thus for $f=\sum_n a_nT^n\in A\<T\>$ we have $f=qp+r$ with $q=\sum_n a_nq_n\in A\<T\>$ and $r=\sum_n a_nr_n\in A[T]$
with $\deg r < d$, and $|f|=\max(|q|, |r|)$ for these $q,r$. 

Uniqueness holds because for $g\in A\<T\>$ with $gp\in A[T]$, $\deg gp < d$, we have $g=0$ by the proof of 
Lemma~\ref{pff}. 
\end{proof}

\begin{corollary}
The composition $A[T]\to A\<T\>\to A\<T\>/pA\<T\>$, with inclusion on the left and the canonical map on the right, is surjective and has kernel $pA[T]$, so induces an $A$-algebra isomorphism $A[T]/pA[T]\to  A\<T\>/pA\<T\>$. 
\end{corollary}
\begin{proof} Lemma~\ref{fqpr} gives surjectivity. The uniqueness in that lemma and  division with remainder in $A[T]$ (by $p$)
yields kernel $pA[T]$. 
\end{proof}

\subsection*{Division with Remainder} Let $n\geqslant 1$, set $Y':=(Y_1,\dots, Y_{n-1})$. The inclusion  $A\<Y'\>[Y_n]\subseteq A\<Y'\>\<Y_n\>=A\<Y\>$ makes $A\<Y'\>[Y_n]$ a subring of $A\<Y\>$. 

\begin{lemma}\label{d1} Let $f\in A\<Y'\>[Y_n]$ be monic of degree $d\ge 1$ and $g\in A\<Y\>$. Then there are unique $q\in A\<Y\>$ and $r\in A\<Y'\>[Y_n]$ with $\deg_{Y_n} r < d$ such that $g=qf+r$. Moreover, $|g|=\max(|q|,|r|)$ for these $q,r$.
\end{lemma}
\begin{proof} This is Lemma~\ref{fqpr} applied to $A\<Y'\>$ in the role of $A$.
\end{proof}

\medskip\noindent
Consider
$A\<X,Y_1,\dots, Y_{j-1}\>[Y_j]$ likewise as a subring of $A\<X,Y\>$ for $j=1,\dots,n$. By a straightforward induction on $n$ the previous lemma gives:

\begin{lemma}\label{d2} Let $f_j\in A\<X,Y_1,\dots,Y_{j-1}\>[Y_j]$ be monic of degree $d_j$ in $Y_j$ for $j=1,\dots,n$. Then 
$$A\<X,Y\>\ =\ (f_1,\dots,f_n)A\<X,Y\> + \bigoplus_{(j_1,\dots, j_n)} A\<X\>Y_1^{j_1}\cdots Y_n^{j_n}$$
where $(j_1,\dots, j_n)$ ranges over the elements of $\N^n$ with
$j_1< d_1,\dots, j_n< d_n$.
\end{lemma}

\begin{corollary}\label{cord2} Let $m=n$ and $f(X)\in A\<X\>$. Then
$$f(X)-f(Y)\in (X_1-Y_1,\dots, X_n-Y_n)A\<X,Y\>.$$
\end{corollary} 
\begin{proof} By Lemma~\ref{d2} we have $f(X)-f(Y) =\sum_{j=1}^n (X_j-Y_j)q_j + r$
with all $q_j$ in $A\<X,Y\>$ and $r\in A\<X\>$. Substituting $X_j$ for $Y_j$ gives
$0=r$.
\end{proof}

\noindent
We extend $a\mapsto \overline{a}: A\to \overline{A}$ to the ring morphism
$$f=\sum_{\nu} a_{\nu}Y^\nu\ \mapsto \overline{f}:= \sum_{\nu}\overline{a_{\nu}}Y^\nu\ :\ A\<Y\>\to \overline{A}[Y],$$
whose kernel is $\smallo(A\<Y\>):=\{f\in A\<Y\>:\ |f|<1\}$. Moreover,
$$\overline{f(g_1,\dots, g_m)}=\overline{f}\big(\overline{g_1},\dots, \overline{g_m}\big), \qquad (f\in A\<X\>,\ g_1,\dots, g_m\in A\<Y\>).$$

\medskip\noindent
For $d\in \N$, call $f\in A\<Y\>$
{\em regular in $Y_n$ of degree $d$} if 
$\overline{f}= f_0 + f_1Y_n + \cdots + f_dY_n^d$ with 
$f_0,\dots, f_{d}\in \overline{A}[Y']$ and $f_d$ a unit in $\overline{A}[Y']$.
We now extend Lemma~\ref{d1}:

\begin{proposition}[Weierstrass Division]\label{wd1} Suppose $f\in A\<Y\>$ is regular in $Y_n$ of degree $d$ and $g\in A\<Y\>$. Then there are $q\in A\<Y\>$ and $r\in A\<Y'\>[Y_n]$ with
$$g\ =\ qf + r, \qquad \deg_{Y_n}r\ <\ d, \qquad |g|=\max(|q|,|r|).$$
\end{proposition}
\begin{proof} Multiplying $f$ by a unit of $A\<Y'\>$ we arrange that $\overline{f}\in \overline{A}[Y]$ is monic in $Y_n$ of degree $d$. Hence $f = f_0 +E$
where $f_0\in A\<Y'\>[Y_n]$ is monic of degree $d$ in $Y_n$
and $E\in A\<Y\>,\ |E|<1$. Now $g=q_0f_0+r_0$ with $q_0\in A\<Y\>$ and
$r_0\in A\<Y'\>[Y_n]$, $\deg_{Y_n}r_0 < d$ and $|g|=\max(|q_0|, |r_0|)$, so
$g=q_0f+r_0 + g_1$ with $g_1=-Eq_0$, and thus $|g_1|\leqslant |E||g|$. With $g_1$ in the role of $g$ and iterating:
\begin{align*}g\ &=\ q_0f+ r_0 + g_1,\quad \qquad g_1\ =\ -Eq_0, \qquad |g_1|\leqslant |E||g|,\\ 
  g_1\ &=\ q_1f+ r_1 + g_2,\quad \qquad g_2\ =\ -Eq_1, \qquad |g_2|\leqslant |E|^2|g|,\\
  \dots\ &=\ \dots\\
  \dots\ &=\ \dots\\
  g_k\ &=\ q_kf+ r_k + g_{k+1},\quad g_{k+1}\ =\ -Eq_k,\quad |g_{k+1}|\leqslant |E|^{k+1}|g|,\\
  \dots\ &=\ \dots
\end{align*}
where $q_k\in A\<Y\>,\ r_k\in A\<Y'\>[Y_n],\ \deg_{Y_n}r_k < d$ and 
$|g_k|=\max(|q_k|,|r_k|)$.  It follows that $g_k, q_k, r_k\to 0$ as $k\to \infty$. Thus we can add
the right and left-hand sides in the equalities above to obtain $g=qf+r$ where
$q:=\sum_k q_k\in A\<Y\>$ and $r:=\sum_k r_k\in A\<Y'\>[Y_n]$, $\deg_{Y_n}r < d$,
so $|g|=\max(|q|,|r|)$. 
\end{proof}

\begin{corollary}[Weierstrass Preparation]\label{wd2} Suppose $f\in A\<Y\>$ is regular in $Y_n$ of degree $d$. Then for some unit $u$ of $A\<Y\>$ we have:
$uf\in A\<Y'\>[Y_n]$, and $uf$ is monic of degree $d$ in $Y_n$.  
\end{corollary}
\begin{proof} We have $Y_n^d=qf+r$ with $q\in A\<Y\>$ and $r\in A\<Y'\>[Y_n]$, $\deg_{Y_n}r < d$. Hence $Y_n^d-\overline{r}=\overline{q}\overline{f}$ in $\overline{A}[Y]$, so $\overline{q}$ is a unit of $\overline{A}[Y']$, hence $q$ is a unit of $A\<Y\>$, and thus $u:=q$ has the desired property.  
\end{proof}

\noindent
A somewhat twisted argument also gives uniqueness in the last two results: 

\begin{corollary}\label{wd3} Let $f\in A\<Y\>$ be regular in $Y_n$ of degree $d$. Then there is only one pair $(q,r)$ with $q\in A\<Y\>$ and $r\in A\<Y'\>[Y_n]$ with $g\ =\ qf + r$ and $\deg_{Y_n}r< d$. There is also only one unit $u$ of $A\<Y\>$
such that $uf\in A\<Y'\>[Y_n]$, and $uf$ is monic of degree $d$ in $Y_n$.  
\end{corollary}
\begin{proof} By Corollary~\ref{wd2} (just the existence of $u$), the uniqueness of $(q,r)$ follows from the uniqueness in Lemma~\ref{d1}. Next, the uniqueness of
$u$ follows from the proof of Corollary~\ref{wd2} and the uniqueness in Proposition~\ref{wd1}.
\end{proof}

\noindent
Besides $n\geqslant 1$ we now also assume $d\geqslant 1$. Under an extra assumption on $\overline{A}$ (see Lemma~\ref{wprep}) we can apply automorphisms to arrange regularity in $Y_n$. Set
$$T_d(Y)\ :=\ \big(Y_1+Y_n^{d^{n-1}},\dots, Y_{n-1}+Y_n^d, Y_n\big),$$
which gives a norm preserving automorphism $f(Y)\mapsto f\big(T_d(Y)\big)$ of the
$A$-algebra $A\<Y\>$ with inverse $g(Y)\mapsto g\big(T_d^{-1}(Y)\big)$, where 
$$T_d^{-1}(Y)\ :=\ \big(Y_1-Y_n^{d^{n-1}},\dots, Y_{n-1}-Y_n^d, Y_n\big).$$

\begin{lemma}\label{wprep} Assume $\overline{A}$ is a field. Let $f\in A\<Y\>$ be such that
$\overline{f}\ne 0$ in $\overline{A}[Y]$, and $d>\deg \overline{f}$. Then $f\big(T_d(Y)\big)$ is regular in $Y_n$ of some degree.  
\end{lemma} 
\begin{proof} With $f=\sum_{\nu} a_{\nu}Y^\nu$,  let $(\mu_1,\dots, \mu_n)$ be lexicographically largest among the $\nu\in \N^n$ for which $\overline{a_{\nu}}\ne 0$. A straightforward computation shows that then for $\ell:=\mu_1d^{n-1} +\cdots + \mu_{n-1}d + \mu_n$ we have
$$\overline{f}\big(T_d(Y)\big)\ =\ \overline{a_{\mu}}Y_n^\ell + \text{terms in }\overline{A}[Y] \text{ of degree $<\ell$ in }Y_n.$$
Thus $f\big(T_d(Y)\big)$ is regular in $Y_n$ of degree $\ell$.
\end{proof}

\section{Rings with $A$-analytic Structure} \label{ras}

\noindent
Given a ring $R$ and a set
$E$ we have the ring $R^E$ of $R$-valued functions on $E$, where the ring operations are given pointwise. A {\em ring with $A$-analytic structure} is a ring $R$ together with a ring morphism 
$$\iota_n\ :\  A\<Y_1,\dots, Y_n\> \to\ \text{ring of $R$-valued functions on $R^n$}$$
for every $n$, such that the following conditions are satisfied:\begin{enumerate}
\item[(A1)] $\iota_n(Y_k)(y_1,\dots, y_n)=y_k$, for $k=1,\dots,n$ and $y=(y_1,\dots, y_n)\in R^n$;
\item[(A2)] for $f\in A\<Y_1,\dots,Y_n\>\subseteq A\<Y_1,\dots, Y_n, Y_{n+1}\>$ and $( y_1,\dots, y_n, y_{n+1})\in R^{n+1}$ we have $\iota_n(f)(y_1,\dots, y_n)=\iota_{n+1}(f)(y_1,\dots, y_n, y_{n+1})$;
\item[(A3)] for $n\geqslant 1$, $f\in A\<Y_1,\dots, Y_n\>$, $g:=f(Y_{n+1},\dots, Y_{2n})\in A\<Y_1,\dots, Y_{2n}\>$, and
 $( y_1,\dots, y_{2n})\in R^{2n}$ we have: $\iota_{n}(f)(y_{n+1},\dots, y_{2n})=\iota_{2n}(g)(y_1,\dots, y_{2n})$.
\end{enumerate}
Let $R$ be a ring with $A$-analytic structure as above. We set $h(y):=\iota_n(h)(y)$ for $h\in A\<Y\>$ and
$y\in R^n$, a notational convention that will be in force from now on.
In other words, each $h\in A\<Y\>$ defines a function $R^n\to R$ that we also denote by $h$. For $n=0$ the above gives the ring morphism $\iota_0: A \to R$ upon identifying
a function $R^0\to R$ with its only value, and so $R$ is an $A$-algebra with structural morphism $\iota_0$. Accordingly we denote for $a\in A$ the element $\iota_0(a)$ of  $R$ also by $a$ when no confusion is likely.
Simple example of a ring with $A$-analytic structure:
$A$ itself with $\iota_n(f)(y):=f(y)$ for $f\in A\<Y\>$ and
$y\in A^n$ and below we consider $A$ to be equipped with this $A$-analytic structure.

\begin{lemma}\label{A3} Let $f, g_1,\dots, g_n\in A\<Y\>$ and
$y\in R^n$. Then $$f(g_1,\dots, g_n)(y)\ =\ f\big(g_1(y),\dots, g_n(y)\big).$$
\end{lemma}
\begin{proof} The case $n=0$ is trivial. Let $n\geqslant 1$ and set $B:=A\<Y_1,\dots, Y_{2n}\>$. In $A\<Y\>$ we have
$f(g_1,\dots, g_n)(Y)=f\big(g_1(Y),\dots, g_n(Y)\big)$, trivially. Also
$A\<Y\>\subseteq  B$, $f(Y_{n+1},\dots, Y_{2n})\in B$, and
by Corollary~\ref{cord2}, $$f(g_1,\dots,g_n)(Y)- f(Y_{n+1},\dots, Y_{2n})\in  
\big(g_1(Y)-Y_{n+1},\dots, g_n(Y)-Y_{2n}\big)B.$$
Now apply $\iota_{2n}$ to this, and use (A1), (A2), (A3) to evaluate at the point
\begin{equation*}
\qquad\qquad\qquad\quad \big(y_1,\dots, y_n, g_1(y),\dots, g_n(y)\big)\in R^{2n}.\qquad\qquad\qquad\qquad \qquad\quad  \qedhere
\end{equation*}
\end{proof}

\noindent
We abbreviate the expression {\em ring with $A$-analytic structure\/} to {\em $A$-analytic ring}, or just {\em $A$-ring}. 
A good feature of the above is that the $A$-rings
naturally form an {\em equational class} (which is not the case for the narrower notion of rings with analytic $A$-structure defined in \cite{vdD}.) 
To back this up,
we introduce the language $\L^A$ of $A$-rings: 
it is the language $\{0,1,-,+,\cdot\}$ of rings augmented by an $n$-ary function symbol for each
$f\in A\<Y\>=A\<Y_1,\dots, Y_n\>$, to be denoted also by $f$. We construe any $A$-ring $R$ in the obvious way as an $\L^A$-structure, with $f$ as above naming the function $y \mapsto f(y): R^n\to R$, so the $A$-rings are exactly the models of an equational $\L^A$-theory, and for any $\L^A$-term $t(Z_1,\dots, Z_n)$ there is an $f\in A\<Y\>$ such that $t(z)=f(z)$ for every $A$-ring $R$ and $z\in R^n$.  

The $A$-ring $A$ is initial in this equational class: 

\begin{lemma}\label{initial} Given any $A$-ring $R$ there is a unique morphism $A\to R$ of $A$-rings, namely $\iota_0: A \to R$. 
\end{lemma}
\begin{proof} If $j: A \to R$ is an $A$-ring morphism, then clearly $j(a)=\iota_0(a)$ for $a\in A$. It remains to check that for $n\geqslant 1$, $f\in A\<Y\>$, $Y=(Y_1,\dots, Y_n)$, and $a_1,\dots, a_n\in A$, 
$$\iota_0\big(f(a_1,\dots, a_n)\big)\ =\ f\big(\iota_0(a_1),\dots, \iota_0(a_n)\big).$$
Take $a_1,\dots, a_n$ as elements of $A\<Y\>$ and  $f(a_1,\dots, a_n)\in A$ accordingly also as an element of $A\<Y\>$. Fixing any $y\in R^n$ we obtain from (A2) that $$\iota_0\big(f(a_1,\dots, a_n)\big)\ =\ \iota_n \big(f(a_1,\dots,a_n)\big)(y),$$ which by Lemma~\ref{A3} equals $f\big(\iota_n(a_1)(y),\dots, \iota_n(a_n)(y)\big)$, and by (A2) again this equals
$f\big(\iota_0(a_1),\dots, \iota_0(a_n)\big)$, as promised.
\end{proof} 

\medskip\noindent
{\bf Example}. Let $A_0$ be a ring with $1\ne 0$ and $A:=A_0[[t]]$, the power series ring in one variable $t$
over $A_0$, with the (complete) ultranorm given by $|f|=2^{-n}$ for $f\in t^nA\setminus t^{n+1}A$.  Let $\iota: A_0 \to \k$ be a ring morphism into a field $\k$, let $\Gamma$ be an ordered abelian group with a distinguished element $1>0$. We identify $\Z$ with its image in $\Gamma$  via $k\mapsto k\cdot1$,
which makes $\Z$ an ordered subgroup
of $\Gamma$. (We do not assume here that $1$ is the least positive element of $\Gamma$.)  This yields the Hahn field
$K=\k(\!(t^\Gamma)\!)$ with its valuation ring $\k[[t^{\Gamma^{\geqslant}}]]\supseteq \k[[t]]$. Now 
$\iota$ extends to the ring morphism $ A \to \k[[t^{\Gamma^{\geqslant}}]]$,
 $$ \sum_n c_nt^n\mapsto \sum_n \iota(c_n)t^n\in \k[[t]] \ \text{ (with all $c_n\in A_0$)}. $$  
We have a natural $A$-analytic structure $(\iota_n)$ on $\k[[t^{\Gamma^{\geqslant}}]]$, where $\iota_0$ is the above ring morphism $A \to \k[[t^{\Gamma^{\geqslant}}]]$, and more generally, for
$f=\sum a_{\nu}Y^\nu$ in $A\<Y\>$ and $y\in (\k[[t^{\Gamma^{\geqslant}}]])^n$, 
$$\iota_n(f)(y)\ :=\ \sum_{\nu} \iota_0(a_{\nu})y^{\nu}\in \\k[[t^{\Gamma^{\geqslant}}]].$$  
Note that $\iota_0(A)$ is the subring $\iota(A_0)[[t]]$ of $\k[[t]]$.

\medskip\noindent
Returning to the general setting, let $R$ be an $A$-ring. Among its units are clearly the elements
$1+a_1y_1+ \cdots + a_ny_n$ for  $a_1,\dots,a_n\in \smallo(A)$ and $y_1,\dots, y_n\in R$.

Any ideal $I$ of $R$ yields a {\em congruence relation\/} for the $A$-analytic structure of $R$. This means: for
any $f\in A\<Y\>$ and any $x,y\in R^n$ with $x\equiv y \mod I$ (that is, $x_1-y_1,\dots, x_n-y_n\in I$), we have
$f(x)\equiv f(y)\mod I$, an immediate consequence of Corollary~\ref{cord2}. Thus $R/I$ is an $A$-ring, given by 
$$f(y_1+I,\dots,y_n+I)\ :=\ f(y_1,\dots, y_n)+I \qquad(f\in A\<Y\>,\ (y_1,\dots,y_n)\in R^n).$$
This construal of $R/I$ as an $A$-ring is part of our goal of developing some algebra for $A$-rings analogous to ordinary facts about rings. But we need some extra notational flexibility in dealing with indeterminates, as we already tacitly used in this argument about $R/I$: we do not want to be tied down to the particular sequence of indeterminates $Y_1, Y_2,\dots $ used in the definition of {\em $A$-analytic structure}. Namely, for any
tuple $Z=(Z_1,\dots, Z_n)$ of distinct indeterminates, not necessarily among the $X_1, X_2,\dots, Y_1, Y_2,\dots$, any $f=f(Z)=\sum_{\nu} a_{\nu}Z^\nu\in A\<Z\>$ and $z\in R^n$ we 
set $f(z):= \big(\iota_n f(Y)\big)(z)$, where $f(Y):=\sum_{\nu}a_{\nu}Y^{\nu}\in A\<Y\>$. 

This is in harmony with other notational conventions: Let $V, Z_1,\dots, Z_n$ be distinct variables. Identifying $A\<Z\>=A\<Z_1,\dots, Z_n\>$ as usual with a subring of $A\<V,Z\>$, this harmony means that for $f\in A\<Z\>$ and 
$(v,z_1,\dots, z_n)$ in $R^{n+1}$
we have 
$f(z_1,\dots, z_n)= f(v,z_1,\dots, z_n)$, where the last $f$ refers to the image of the series $f\in A\<Z\>$ in $A\<V, Z\>$. Thus we can add dummy variables on the left. We can also add them at other places: identifying $f\in A\<Z\>$ with its image in $A\<Z_1,\dots,Z_i, V, Z_{i+1},\dots, Z_n\>$ as usual, where $1\leqslant i \leqslant n$, we have likewise
$$f(z_1,\dots, z_n)\ =\ f(z_1,\dots, z_i, v, z_{i+1},\dots, z_n)$$
for $(z_1,\dots, z_i,v,z_{i+1},\dots, z_n)\in R^{n+1}$. We shall tacitly use these facts.

\subsection*{Henselianity Again}   
Let $R$ be an $A$-ring. Note that $\smallo(A)R$, that is, the ideal of $R$ generated by
$\iota_0\big(\smallo(A)\big)$, is contained in the Jacobson radical of $R$, because for
$a_1,\dots, a_n\in \smallo(A)$ the series $1+ a_1Y_1+\cdots + a_nY_n$ is a unit of $A\<Y\>$.
In later sections we consider the case that $R$ is a valuation ring whose maximal ideal
is $\sqrt{\smallo(A)R}$, and then the following is relevant:

\begin{lemma}\label{he5} The pair $(R, \smallo(A)R)$ is henselian, hence so is $(R,\sqrt{\smallo(A)R})$.
\end{lemma}
\begin{proof} Let $n=1$, so $Y=Y_1$. We show that any polynomial $$f(Y)\ =\ 1+Y + z_2Y^2 + \cdots + z_NY^N\in R[Y], \qquad(N\in \N^{\geqslant 2})$$
with $z_2,\dots, z_N\in \smallo(A)R$ has a zero in $R$. Take $m$ and $x\in R^m$ such that
$$z_2\ =\ g_2(x),\dots, z_N\ =\ g_N(x), \quad g_2,\dots, g_N\in \smallo(A)A[X]\subseteq \smallo(A)A\<X\>.$$
Then $F(X,Y):= 1+Y + g_2(X)Y^2+\cdots + g_N(X)Y^N\in A[X,Y]=A\<X,Y\>$ is regular in $Y$ of degree $1$,
so $F(X,Y)=E\cdot(Y-c)$ for a unit $E$ of $A\<X,Y\>$ and $c\in A\<X\>$. Thus
$f(Y)$ has a zero $c(x)$ in $R$.
\end{proof}

\subsection*{Extensions of $A$-rings} Let $R$ be an $A$-ring.
When referring to an $A$-ring $R^*$ as {\em extending $R$\/} this means of course that $R$ is a subring of $R^*$, but also includes
the requirement that the $A$-analytic structure of $R^*$ extends that of $R$.  

A set $S\subseteq R$ is said to be $A$-closed (in $R$) if for all $m$, $f\in A\<X\>$ and $x_1,\dots, x_m$ in $S$ we have $f(x_1,\dots, x_m)\in S$.
Then $S$ is a subring of $R$ and the $A$-analytic structure of $R$ restricts to an $A$-analytic structure on $S$. We view such $S$ as an $A$-ring so as to make the $A$-ring $R$ extend $S$.
For $S\subseteq R$, the {\em $A$-closure of $S$ in $R$\/} is the smallest (with respect to inclusion)  $A$-closed subset of $R$ that contains $S$.   

\begin{lemma}\label{gen} Let $R^*$ be an $A$-ring extending $R$, and $y=(y_1,\dots, y_n)\in (R^*)^n$. Let $R\<y\>$ be the $A$-closure of $R\cup\{y_1,\dots, y_n\}$ in $R^*$. Then  
$$R\<y\>\ =\ \bigcup_m\{g(x,y):\ x\in R^m,\ g\in A\<X,Y\>\}.$$ 
\end{lemma}

\medskip\noindent
Here is a consequence of Lemma~\ref{d2}:

\begin{lemma}\label{Aint1} Suppose $R^*$ is an $A$-ring that extends $R$.  Let $f\in A\<X,Y_1,\dots, Y_n\>$ and $x\in R^m$, and assume $y_1,\dots, y_n\in R^*$ are integral over $R$.
Then $$f(x,y_1,\dots, y_n)\ \in\ R[y_1,\dots, y_n].$$ 
\end{lemma}
\begin{proof} By increasing $m$ and accordingly extending $x$ with extra coordinates we arrange that for $j=1,\dots,n$ we have a polynomial $f_j(X, Y_j)\in A[X,Y_j]$, monic in $Y_j$, with $f_j(x,y_j)=0$. Now apply Lemma~\ref{d2}.
\end{proof}

\begin{lemma}\label{Aint2}
Let $R^*$ be a ring extension of $R$ with $z\in R^*$ integral over $R$. Then at most one $A$-analytic structure on $R[z]$ makes $R[z]$ an $A$-ring extending $R$.    
\end{lemma}
\begin{proof} We can assume $R^*=R[z]$. Take a monic polynomial 
$\phi\in R[Z]$, say of degree $d\geqslant 1$, with $\phi(z)=0$. Let $R^*$ be equipped with
an $A$-analytic structure extending that of $R$, and let $g\in A\<Y_1,\dots, Y_n\>$, $n\geqslant 1$, and let
$y_1,\dots, y_n\in R^*$; we have to show that then the element 
$g(y_1,\dots, y_n)\in R^*$ does not depend on the given $A$-analytic 
structure on $R^*$. We have $\phi(Z)=x_{00} + x_{01}Z +\cdots + x_{0,d-1}Z^{d-1} + Z^d$
with $x_{00},\dots,x_{0,d-1}\in R$ and
 $y_j=x_{j0} + x_{j1}z +\cdots + x_{j,d-1}z^{d-1}$, $x_{j0},\dots, x_{j,d-1}\in R$, for $j=1,\dots,n$. We now set $m:=(1+n)d$ and
\begin{align*}  x\ &:=\ (x_{00},\dots, x_{0,d-1}, x_{10},\dots, x_{1,d-1},\dots, x_{n0},\dots, x_{n,d-1})\in R^m,\\
X\ &=\ (X_1,\dots, X_m)\ :=\ \big(X_{00},\dots, X_{0,d-1},\dots, X_{n0},\dots, X_{n,d-1}\big),
\end{align*}
so $\phi(Z)=F(x,Z)$,  $F(X,Z):=X_{00}+ X_{01}Z + \cdots + X_{0,d-1}Z^{d-1} + Z^d\in A[X,Z]$. 
Let $G(X, Z)\in A\<X,Z\>$ be the following substitution instance of $g$: 
$$  g\big(X_{10}+X_{11}Z + \cdots + X_{1,d-1}Z^{d-1},\dots,X_{n0}+X_{n1}Z + \dots + X_{n,d-1}Z^{d-1}\big).$$  
Lemma~\ref{d1} gives $G(X,Z)=Q(X,Z)F(X,Z)+R_0 +R_1Z+ \cdots + R_{d-1}Z^{d-1}$ with
$R_0,\dots, R_{d-1}\in A\<X\>$, and so
$g(y)=G(x,z)= R_0(x) +R_1(x)z+ \cdots + R_{d-1}(x)z^{d-1}$,
which uses only the $A$-analytic structure on $R$, not that on $R^*$. 
\end{proof}

\begin{proposition}\label{Aint3} 
Let $R^*$ be a ring extension of $R$ and integral over $R$. Then some $A$-analytic structure on $R^*$ makes $R^*$ an $A$-ring extending $R$.  
\end{proposition}
\begin{proof} In view of Lemmas~\ref{Aint1} and ~\ref{Aint2} this reduces to the case $R^*=R[z]$ where $z\in R^*$ is integral over $R$. Let $\phi(Z)\in R[Z]$ be as in the proof of Lemma~\ref{Aint2}, in particular monic of degree $d\geqslant 1$ in $Z$.
If the ring extension $R[Z]/\phi(Z)R[Z]$ of $R$ can be given an $A$-analytic structure extending that of $R$, then this is also the case for its image $R[z]$  under the $R$-algebra morphism $R[Z]\to R[z]$ sending $Z$ to $z$. Thus replacing $R[z]$ by $R[Z]/\phi(Z)R[Z]$ if necessary we arrange that $R[z]$ is free as an $R$-module with basis $1,z,\dots, z^{d-1}$.    
We now adopt other notation from the proof above, where $n\geqslant 1$ and where we introduced a tuple 
$$X\ =\ (X_1,\dots, X_m)\  =\  (X_{00},\dots, X_{n,d-1})$$ of $m=(n+1)d$ distinct variables,
the polynomial $F(X,Z)\in A[X,Z]$, and for any $g\in A\<Y\>=A\<Y_1,\dots, Y_n\>$ the series $G=G(X,Z)\in A\<X,Z\>$, and the series $R_0,\dots, R_{d-1}\in A\<X\>$.
To indicate their dependence on $g$ we set 
\begin{align*} G_g\ :=\ G,\quad  R_{g,0}\  :=\ R_{0},\quad \dots\ ,\ R_{g,d-1}\  :=\ R_{d-1},\\
R_g\ :=\  R_{g,0} + R_{g,1}Z +\cdots + R_{g,d-1}Z^{d-1}\ \in\ A\<X\>[Z].
\end{align*} 
We claim that setting $g(y) := R_{g}(x,z)$
for any $n\geqslant 1$ and $g\in A\<Y\>$ yields an $A$-analytic structure
on $R[z]$ extending that on $R$. We just verify two
items that are part of this claim: let $f,g,h, g_1,\dots, g_n\in A\<Y\>$ and $y\in R^n$; then \begin{enumerate}
\item  $gh(y)\ =\ g(y)h(y)$;
\item  $f(g_1,\dots, g_n)(y)\ =\ f\big(g_1(y),\dots, g_n(y)\big)$.
\end{enumerate}
As to (1), we have $G_{gh}=G_{g}G_{h}$, so $R_{gh}\equiv R_{g}R_{h}\mod F$ in $A\<X,Z\>$. We also have $R\in A\<X\>[Z]$ of
degree $< d$ in $Z$ such that $R_{g}R_{h}\equiv R \mod F$ in $A\<X\>[Z]$.
Hence $R_{gh}=R$, and thus $gh(y)=R(x,z)=R_g(x,z)R_h(x,z)=g(y)h(y)$. 
As to (2), by Corollary~\ref{cord2} we have in $A\<X,Z\>$,
$$G_{f(g_1,\dots,g_n)}\ =\ f\big(G_{g_1},\dots, G_{g_n}\big)\ \equiv\ f(R_{g_1},\dots, R_{g_n})\mod F. $$
Note that $f(R_{g_1},\dots, R_{g_n})$ is obtained by substituting $R_{{g_j},i}$
for $X_{ji}$ in $G_f$, for $j=1,\dots,n$ and $i=0,\dots, d-1$ (and $Z$ for $Z$), that is, 
$$ f(R_{g_1},\dots, R_{g_n})\ =\  G_f\big(R_{g_1,0},\dots,R_{g_1,d-1},\dots, R_{g_n,0}, \dots, R_{g_n,d-1}, Z\big). $$
Making the same substitution in the congruence $G_f\equiv R_f \mod F$, using that
the variables $X_{ji}$ with $j=1,\dots,n$ and $i=0,\dots,d-1$ do not occur in $F$, we obtain 
$$G_f\big(R_{g_1,0},\dots,R_{g_1,d-1},\dots, R_{g_n,0}, \dots, R_{g_n,d-1},Z\big)$$
is congruent in $A\<X,Z\>$ modulo $F$ to
$$R_f(X_{00},\dots, X_{0,d-1},R_{g_1,0},\dots,R_{g_1,d-1},\dots, R_{g_n,0}, \dots, R_{g_n,d-1},Z\big), $$
which is in $A\<X\>[Z]$ of degree $<d$ in $Z$, and thus equals $R_{f(g_1,\dots,g_n)}$.  Since
$$  f\big(g_1(y),\dots, g_n(y)\big)\ =\ R_f\big(x_{00},\dots, x_{0,d-1},R_{g_1,0}(x) ,\dots, R_{g_n,d-1}(x),z\big),$$
this yields  $f\big(g_1(y),\dots, g_n(y)\big)=f(g_1,\dots,g_n)(y)$, as required.
\end{proof}

\begin{corollary}\label{Aint4} If $R^*$ is a ring extension of $R$ and integral over $R$, then there is a unique $A$-analytic structure on $R^*$ that makes $R^*$ an $A$-ring extending $R$.
\end{corollary} 

\begin{corollary}\label{Aint5} Let $R_1$ and $R_2$ be $A$-rings extending $R$ and let $\phi: R_1\to R_2$ be an $R$-algebra morphism such that $\phi(R_1)$ is integral over $R$. Then $\phi$ is a morphism of $A$-rings 
$($that is, a homomorphism in the sense of $\L^A$-structures$)$.
\end{corollary}
\begin{proof} The kernel of $\phi$ is a congruence relation on $R$ as $A$-ring, so $\phi(R_1)$ has an $A$-analytic structure making $\phi: R_1\to \phi(R_1)$
a morphism of $A$-rings. Since $\phi(R_1)$ is $A$-closed as a subset of $R_2$ it follows from Corollary~\ref{Aint4} that this $A$-analytic structure on $\phi(R_1)$ coincides with the one that makes the inclusion
$\phi(R_1)\to R_2$ a morphism of $A$-rings. Thus $\phi$ is a morphism of $A$-rings.
\end{proof}

\begin{corollary}\label{Aint6} Suppose the $A$-ring $R^*$ extends $R$, and $z_i\in R^*$ for $i\in I$ is integral over $R$. Then $R[z_i: i\in I]$ is $A$-closed in $R^*$. 
\end{corollary}
\begin{proof} Let $f(Y)\in A\<Y\>$ and suppose $y_1,\dots, y_n\in R^*$ are integral over $R$; it suffices to show that then $f(y)\in R[y]$ where $y=(y_1,\dots, y_n)$. Take $x\in R^m$ and monic $f_j\in A\<X\>[Y_j]$  such that $f_j(x,y_j)=0$ for $j=0,\dots,n$, and apply Lemma~\ref{d2}. 
\end{proof}

\subsection*{Defining $R\<Y\>$} Let $R$ be an $A$-ring. To define a ring $R\<Y\>$ analogous to the polynomial ring $R[Y]$, observe that
polynomials over $R$  arise from polynomials over $\Z$ by specializing: for $f(X,Y)\in \Z[X,Y]$ and $x\in R^m$
we have $f(x,Y)\in R[Y]$.  We take this as a hint and with $A$  instead of $\Z$, we define for $f(X,Y)\in A\<X,Y\>$ and $x\in R^m$ the power series
$f(x,Y)\in R[[Y]]$ as follows: $f(X,Y)=\sum_{\nu}f_{\nu}(X)Y^{\nu}$ with all
$f_{\nu}\in A\<X\>$, and then $$f(x,Y)\ :=\ \sum_{\nu}f_{\nu}(x)Y^\nu.$$ 
Thus for fixed $x\in R^m$ the map 
$f(X,Y)\mapsto f(x,Y) :  A\<X,Y\>\to R[[Y]]$ is an $A$-algebra morphism. 
We define $$R\<Y\>\ :=\ \bigcup_m\{f(x,Y):\ f=f(X,Y)\in \\A\<X,Y\>,\ x\in R^m\}\ \subseteq\ R[[Y]].$$ 
An easy consequence is that inside the ambient ring $R[[Y]]$ we have for $i\leqslant n$:
$$ R\<Y_1,\dots, Y_i\>\ =\ R\<Y\>\cap R[[Y_1,\dots, Y_i]].$$

\begin{lemma} Given any $g_1,\dots, g_k\in R\<Y\>$, $k\in \N$, there exists
$m$, $x\in R^m$, and $f_1,\dots, f_k\in A\<X,Y\>$, such that
$g_1=f_1(x,Y),\dots, g_k=f_k(x,Y)$.
\end{lemma}
\begin{proof} Let $m_1,\dots, m_k\in \N$ and $x^1\in R^{m_1},\dots, x^k\in R^{m_k}$ be such that 
\begin{align*} g_1\ &=\ f^1(x^1,Y),\dots, g_k\ =\ f^k(x^k,Y),\quad  f^1\in A\<X^1, Y\>,\dots, f^k\in A\<X^k, Y\>,\\
 x^1\ &=\ (x_{11},\dots, x_{1m_1}),\ \dots,\  x^k\ =\ (x_{k1},\dots, x_{km_k}),\\
X^1  :&=\ (X_{11},\dots, X_{1m_1}),\ \dots,\ X^k\ :=\ (X_{k1},\dots, X_{km_k}).
\end{align*}
We can also arrange for $m:=m_1+\cdots + m_k$ that
$$ X\ =\ (X_1,\dots, X_m)\ =\ (X^1,\dots, X^k).$$
For $f_1(X,Y):=f^1(X^1,Y)\in A\<X,Y\>,\dots,f_k(X,Y):=f^k(X^k,Y)\in A\<X,Y\>$
we then have $f_1(x,Y)=g_1,\dots, f_k(x,Y)=g_k$ for
$x=(x^1,\dots, x^k)\in R^m$.   
\end{proof} 

\begin{corollary}\label{ry} $R\<Y\>$ is a subring of $R[[Y]]$ with $R[Y]\subseteq R\<Y\>$.  If $R$ is a domain, then so is $R\<Y\>$; if $R$ has no nilpotents other than $0$, then neither does
$R\<Y\>$. For an $A$-ring $R^*$ extending $R$ the inclusion $R[[Y]]\to R^*[[Y]]$ maps $R\<Y\>$ into $R^*\<Y\>$, so $R\<Y\>$ is a subring of $R^*\<Y\>$.
\end{corollary}
\begin{proof} The claim about domains holds because it holds with $R[[Y]]$ in place of $R\<Y\>$. Suppose $R$ has no nilpotents. With $\mathfrak{p}$ ranging over the prime ideals of $R$ this yields an injective ``diagonal'' ring morphism $R[[Y]]\to \prod_{\mathfrak{p}} (R/\mathfrak{p})[[Y]]$ into a ring with no nilpotents other than $0$, so $R[[Y]]$ has no such nilpotents either.
\end{proof}

\noindent
By the remark following the definition of $R\<Y\>$ we have for $i\leqslant n$
the subring $R\<Y_1,\dots, Y_i\>$ of $R\<Y\>$. The ring $A\<Y\>$
as defined in Section~\ref{ASR} is the same as the ring $A\<Y\>$ as
defined just now for $R=A$ viewed as an $A$-ring.

\begin{corollary}\label{corintgeny} Suppose the $A$-ring $R^*$ extends $R$ and is integral over $R$. Then 
$R^*\<Y\>$ is generated as a ring over its subring $R\<Y\>$ by $R^*$.
\end{corollary}
\begin{proof} Using Corollary~\ref{Aint6} it suffices to consider the case $R^*=R[z]$ where $z\in R^*$ is integral
over $R$ and to show $R^*\<Y\>=R\<Y\>[z]$. Let $f(X,Y)\in A\<X,Y\>$ and $x\in (R^*)^m$.  Towards proving $f(x,Y)\in R\<Y\>[z]$, let  $\phi(z)=0$ where $$\phi(Z)\ =\ Z^d + u_{0,d-1}Z^{d-1} + \cdots + u_{0,0},\quad  d\geqslant 1,\quad
u_{0,0},\dots, u_{0,d-1}\in R.$$ Then for $i=1,\dots,m$ we have $x_i=u_{i0} + u_{i1}z+\cdots + u_{i,d-1}z^{d-1}$ with all $u_{ij}\in R$.
Let $U=(U_{ij})_{0\leqslant i\leqslant m, j<d}$ be a tuple of distinct variables different also from the $Y_j$ and $Z$ and
set $u:=(u_{ij})$. Then $f(x,Y)=g(u,z,Y)$ where
$$g(U,Z,Y)\ =\ f\big(\sum_{j<d}U_{1j}Z^i,\dots, \sum_{j<d}U_{mj}Z^i, Y\big)\ \in\ A\<U,Z,Y\>.$$
In the ring $A\<U,Z,Y\>$,  $g(U,Z,Y)$ is congruent modulo $Z^d+ U_{0,d-1}Z^{d-1}+ \cdots + U_{0,0}$ to $g_0(U,Y) +g_1(U,Y)Z + \cdots +g_{d-1}(U,Y)Z^{d-1}$ for suitable $g_0,\dots, g_{d-1}\in A\<U,Y\>$, by Lemma~\ref{d1}, and for such
$g_j$ we have 
\begin{equation*}
g(u,z,Y)\ =\ g_0(u,Y)+g_1(u,Y)z+\cdots + g_{d-1}(u,Y)z^{d-1}\in R\<Y\>[z]. \qedhere
\end{equation*}
\end{proof}

\medskip\noindent
Let  $x\in R^m$. Then we equip $R$ with the $A\<X\>$-analytic structure  $(\iota_{n,x})$ given for $f(X,Y)\in A\<X\>\<Y\>=A\<X,Y\>$ by
$$\iota_{n,x}f\ :\  R^n\to R, \qquad y\mapsto f(x,y).$$
We refer to  $R$ with this $A\<X\>$-analytic structure as {\em  the $(A,x)$-ring $R$}.  
Construing $R$  as an $(A,x)$-ring gives the same subring
$R\<Y\>$ of $R[[Y]]$ as when considering $R$ as an $A$-ring.

\subsection*{$R\<Z\>$ as an $A$-ring} 
Let $Z_1,\dots, Z_N$ with $N\in \N$ be distinct variables different from $X_1, X_2,\dots$, and set $Z:=(Z_1,\dots, Z_N)$. We define $R\<Z\>=R\<Z_1,\dots, Z_N\>$ in the same way
as $R\<Y_1,\dots, Y_N\>$, with $Z_1,\dots, Z_N$ in the role of $Y_1,\dots, Y_N$. We make $R\<Z\>$ an $A$-ring extending $R$ as follows. 
Let $f\in A\<Y\>$ and 
$u_1(x,Z),\dots, u_n(x,Z)$ in $R\<Z\>$, where $u_1(X,Z),\dots, u_n(X,Z)\in
A\<X,Z\>$ and $x\in R^m$. Set $$g(X,Z)\ :=\  f\big(u_1(X,Z),\dots, u_n(X,Z)\big)\in A\<X,Z\>.$$ Our aim is to define $f(u_1(x,Z),\dots, u_n(x,Z)):= g(x,Z)\in R\<Z\>$.
In order for this to make sense as a definition we first show:

\begin{lemma}\label{gh} Suppose $v_j(X,Z)\in A\<X,Z\>$ and
$u_j(x,Z)=v_j(x,Z)$ for $j=1,\dots,n$. Set $h(X,Z):= f\big(v_1(X,Z),\dots, v_n(X,Z)\big)\in A\<X,Z\>$. Then  $$g(x,Z)\ =\ h(x,Z).$$
\end{lemma}
\begin{proof} By Corollary~\ref{cord2}  we have for distinct variables $U_1,\dots, U_n, V_1,\dots, V_n$,
$$f(U_1,\dots, U_n)-f(V_1,\dots, V_n)\in (U_1-V_1,\dots, U_n-V_n)A\<U,V\>.$$
Substituting
$u_j(X,Z)$ and $v_j(X,Z)$ for $U_j$ and $V_j$ gives 
$$g(X,Z)-h(X,Z)\in \big(u_1(X,Z)-v_1(X,Z),\dots, u_n(X,Z)-v_n(X,Z)\big)A\<X,Z\>$$ 
from which we obtain the desired result by substituting $x$ for $X$. 
\end{proof}

\noindent
Now the next lemma shows the above does define $f\big(u_1(x,Z),\dots, u_n(x,Z)\big)$:

\begin{lemma}\label{ind}
Let $m_1, m_2\in \N$, $m:=m_1+m_2$, and
\begin{align*}X^1\ &:=\ (X_1,\dots, X_{m_1}),\quad X^2\ :=\ (X_{m_1+1},\dots, X_m),\\  x^1\ &=\ (x_1,\dots, x_{m_1})\in R^{m_1},\quad x^2\ =\ (x_{m_1+1},\dots, x_m)\in R^{m_2}.
\end{align*} Suppose that the series
$$u^1(X^1,Z),\dots, u^n(X^1,Z)\in A\<X^1,Z\>,\quad v^1(X^2,Z),\dots, v^n(X^2,Z)\in A\<X^2,Z\>$$
are such that $u^1(x^1,Z)=v^1(x^2,Z),\dots, u^n(x^1,Z)=v^n(x^2,Z)$. Then for
\begin{align*}g^1(X^1,Z)\ &:=\ f\big(u^1(X^1,Z),\dots, u^n(X^1,Z)\big)\in A\<X^1,Z\>,\\
 g^2(X^2,Z)\ &:=\ f\big(v^1(X^2,Z),\dots, v^n(X^2,Z)\big)\in A\<X^2,Z\>
\end{align*}
we have $g^1(x^1,Z)=g^2(x^2,Z)$ in $R\<Z\>$. 
\end{lemma} 
\begin{proof} Set $X:=(X^1, X^2)=(X_1,\dots, X_m)$, and for $j=1,\dots,n$,
\begin{align*} u_j(X,Z)\ &:=\ u^j(X^1,Z)\in A\<X,Z\>,\quad v_j(X,Z)\ :=\ v^j(X^2,Z) \in A\<X,Z\>,\\ 
g(X,Z)\ &:=\ f\big(u_1(X,Z),\dots, u_n(X,Z)\big)\ =\ g^1(X^1,Z)\in A\<X,Z\>,\\
 h(X,Z)\ &:=\ f\big(v_1(X,Z),\dots, v_n(X,Z)\big)\ =\ g^2(X^2,Z)\in A\<X,Z\>.
\end{align*}
Then for $x=(x_1,\dots,x_m)$ 
we have $u_j(x,Z)=v_j(x,Z)$ for $j=1,\dots,n$, so $g(x,Z)=h(x,Z)$ by Lemma~\ref{gh}, and thus $g^1(x^1,Z)=g^2(x^2,Z)$.
\end{proof}

\noindent
We have now defined for $f\in A\<Y\>$ a corresponding operation
$$(u_1,\dots, u_n)\mapsto f(u_1,\dots, u_n)\ :\  R\<Z\>^n\to R\<Z\>.$$
This makes $R\<Z\>$ an $A$-ring extending $R$. 
For $f\in A\<X,Z\>$ and $x\in R^m$ we can interpret $f(x,Z)$ on the one hand as
the element of $R[[Z]]$ defined in the beginning of this subsection (with
$Y$ instead of $Z$), but also as the element of $R\<Z\>$ obtained by evaluating
$f$ at the point $(x,Z)\in R\<Z\>^{m+N}$ according to the $A$-analytic structure
we gave $R\<Z\>$; one checks easily that these two interpretations
give the same element of $R\<Z\>$, so there is no conflict of notation. 
This also shows that $R\<Z\>$ is generated as an $A$-ring by its subset
$R\cup\{Z_1,\dots, Z_N\}$.

\medskip\noindent
For $R=A$ as an $A$-ring the above yields the $A$-ring $A\<Z\>$
extending $A$. Let $f=f(Y)=\sum_\nu a_{\nu}Y^\nu\in A\<Y\>$. One checks easily that for $(g_1,\dots,g_n)\in A\<Z\>^n$ the convergent sum 
$f(g_1,\dots, g_n)=\sum_{\nu}a_{\nu}g_1^{\nu_1}\cdots g_n^{\nu_n}\in A\<Z\>$ equals $f(g_1,\dots, g_n)$ as defined above for $R=A$, so this causes no conflict of notation. It is routine to check that for any $A$-ring $R$ and $z\in R^N$ 
the evaluation map $g\mapsto g(z): A\<Z\>\to R$ is a morphism of $A$-rings.
For $N=0$ this is just $\iota_0: A \to R$.

\begin{corollary}\label{he6} Let $J:=\sqrt{\smallo(A)R}$. Then $\big(R\<Z\>,JR\<Z\>\big)$ is henselian.
\end{corollary}
\begin{proof} Applying Lemma~\ref{he5} to the $A$-ring $R\<Z\>$, the pair
$\big(R\<Z\>, \sqrt{\smallo(A)R\<Z\>}\big)$ is henselian. Now use that $JR\<Z\>\subseteq \sqrt{\smallo(A)R\<Z\>}$.
\end{proof}

\medskip\noindent
Our next goal is to define for $z\in R^N$ an evaluation map $g\mapsto g(z): R\<Z\>\to R$. We do this in the next section under a further noetherian assumption on $A$.

\section{The case of noetherian $A$} \label{noethA}

\noindent
{\em Let $A$ be a noetherian ring with an ideal $\smallo(A)\ne A$ such that $\bigcap_e\smallo(A)^e=\{0\}$ $($with $e$ ranging here and below over $\N)$ and
$A$ is $\smallo(A)$-adically complete}. Taking $0<\delta<1$ and defining
$|a|:=\delta^{n}$ if $a\in \smallo(A)^n\setminus \smallo(A)^{n+1}$ for $a\in A^{\ne}$ and $|0|:= 0$ gives an ultranorm on $A$ with respect to which $A$ is complete, with $\smallo(A)=\{a\in A:\ |a|<1\}$. Then the $\smallo(A)$-adic topology is the norm-topology. Take $t_1,\dots, t_r\in A$, $r\in \N$, such that $\smallo(A)=(t_1,\dots, t_r)$. Below, $n\geqslant 1$ and $Y=(Y_1,\dots, Y_n)$ as before, and $\lambda,\mu,\nu$ range over $\N^n$. 

\begin{lemma}\label{n1} Let $f=\sum_{\nu} a_{\nu}Y^{\nu}\in A\<Y\>$. Then there is $d\in \N^{\geqslant 1}$ such that for all $\nu$ with $|\nu|\geqslant d$ we have
$a_{\nu}=\sum_{|\mu|<d}a_{\mu}b_{\mu\nu}$ where the $b_{\mu\nu}\in \smallo(A)$ can be chosen such that $b_{\mu\nu}\to 0$ as $|\nu|\to \infty$ for each fixed
$\mu$ with $|\mu|<d$.
\end{lemma} 
\begin{proof} Since $a_{\nu}\to 0$ as $|\nu|\to \infty$, we have
$a_{\nu}\in \smallo(A)^{e(\nu)}$ with $e(\nu)\in \N$, $e(\nu)\to \infty$ as 
$|\nu|\to \infty$. So
 $a_{\nu}=P_\nu(t_1,\dots, t_r)$ with $P_{\nu}\in A[T_1,\dots, T_r]$ homogeneous of degree $e(\nu)$. Take $d_0\in \N$ such that the ideal of
  $A[T_1,\dots, T_r]$ generated by the $P_{\nu}$ is already generated by the
  $P_{\mu}$ with $|\mu|< d_0$. Next take $d\geqslant d_0$ in $\N^{\geqslant 1}$ so large that $e(\nu)> e(\mu)$
  for all $\mu,\nu$ with $|\mu|<d_0$ and $|\nu|\geqslant d$. Let $|\nu|\geqslant d$. Then $P_{\nu}=\sum_{|\mu|<d_0} P_{\mu} Q_{\mu\nu}$ with each $Q_{\mu\nu}\in A[T_1,\dots, T_r]$ homogeneous of degree $e(\nu)-e(\mu)$. Hence
  $$a_{\nu}\ =\ \sum_{|\mu|< d_0} a_{\mu}b_{\mu\nu}, \qquad b_{\mu\nu}\ :=\ Q_{\mu\nu}(t_1,\dots, t_r)\in \smallo(A),$$ which yields the desired result.    
\end{proof}

\noindent
Let $f$, $d$, and the $b_{\mu\nu}$ be as in the lemma. For $\mu$ with $|\mu|<d$ we set
$$f_{\mu}\ :=\ Y^\mu + \sum_{|\nu|\geqslant d}b_{\mu\nu}Y^\nu\in A\<Y\>,\ \text{ so }\
      f\ =\ \sum_{|\mu|<d} a_{\mu}f_{\mu}.$$
Therefore $\smallo(A\<Y\>)=(t_1,\dots,t_r)A\<Y\>$ and the ultranorm on $A\<Y\>$ induced by the above ultranorm
on $A$ has the property that for all $f\in A\<Y\>$ and $n$,
 $$|f|\leqslant \delta^n\ \Longleftrightarrow\ f\in \smallo(A\<Y\>)^n,$$
so the norm-topology of $A\<Y\>$ is the same as its $\smallo(A\<Y\>)$-adic topology. Moreover,
$$ A[Y]\cap \smallo(A\<Y\>)^n\ =  \smallo(A)^n A[Y], \text{ for all }n,$$
so $A\<Y\>$ with the $\smallo(A\<Y\>)$-adic topology is a completion of its noetherian subring $A[Y]$ with the $\smallo(A)A[Y]$-adic topology. Then
 $A\<Y\>$ is noetherian by \cite[Theorem 8.12]{Mat}, so 
 $A\<Y\>$ inherits the conditions we imposed on $A$ at the beginning of this section.  

\subsection*{Passing to $A/I$}  Let
 $I$ be a proper ideal of $A$. Since $A$ is noetherian, $I$ is closed in $A$ by \cite[Theorem 8.14]{Mat}. We equip $A/I$ with its
quotient norm, and observe that the ring morphism 
$$A\<Y\>\to (A/I)\<Y\>,\  \quad f:=\sum_{\nu} a_{\nu}Y^\nu \mapsto f/I:=\sum_{\nu}(a_{\nu}+I)Y^\nu$$
is surjective and that its kernel contains $IA\<Y\>$. Moreover:

\begin{lemma} \label{IaA} The ring morphism $A\<Y\>\to (A/I)\<Y\>$ has the following properties:
\begin{enumerate}
\item[(i)] its kernel is $IA\<Y\>$, a closed proper ideal of $A\<Y\>$;
\item[(ii)]  the induced ring isomorphism 
$$A\<Y\>/IA\<Y\>\to (A/I)\<Y\>$$ is norm preserving, with the quotient norm on $A\<Y\>/IA\<Y\>$; 
\item[(iii)] for $f, g_1,\dots, g_n\in A\<Y\>$ we have $(f/I)(g_1/I,\dots, g_n/I)=f(g_1,\dots, g_n)/I$.
\end{enumerate}
\end{lemma} 
\begin{proof} Suppose $f(Y)=\sum_{\nu} a_{\nu}Y^\nu\in A\<Y\>$ is in the kernel. Then all $a_{\nu}\in I$. 
For some $d\geqslant 1$ we have an equality $f(Y)=\sum_{|\mu|<d} a_{\mu}f_{\mu}(Y)$
with $\mu$ ranging over $\N^n$ and all $f_{\mu}(Y)\in A\<Y\>$, hence $f(Y)\in IA\<Y\>$. Verifying (ii) is routine using the definitions of the norms involved. Item (iii) is an easy consequence of (ii). 
\end{proof} 

\subsection*{The effect on $A$-rings} Our noetherian assumption on $A$ has consequences for $A$-rings. 
{\em In the rest of this section $R$ is an $A$-ring, and $Z=(Z_1,\dots, Z_N)$ as before}. With $(\iota_n)$ 
the $A$-analytic structure of $R$, here is a corollary of Lemma~\ref{IaA}(i):

\begin{corollary} \label{Iar} Suppose the proper ideal $I$ of $A$ is contained in the kernel of $\iota_0$. Then
we have an $(A/I)$-analytic structure $(\iota_n/I)_n$ on $R$ given by
  $$(\iota_n/I)(f/I)\ :=\  \iota_n(f)\ \text{  for }f\in A\<Y\>.$$ 
\end{corollary}

\begin{lemma}\label{n2} Let $f=\sum_{\nu}a_{\nu}(X)Y^\nu\in A\<X\>\<Y\>= A\<X,Y\>$. Suppose
$x\in R^m$ and $f(x,Y)=0$, that is, $a_{\nu}(x)=0$ for all $\nu$.
Then $f(x,y)=0$ for all $y\in R^n$.
\end{lemma}
\begin{proof} With $A\<X\>$ in the role of $A$, the above gives 
a finite sum decomposition $f\ =\ \sum_{|\mu|<d} a_{\mu}f_{\mu}$
with the $f_{\mu}\in A\<X,Y\>$, which yields the desired conclusion.
\end{proof}

\noindent
We can now prove the following key universal property of the $A$-ring $R\<Z\>$:

\begin{theorem}\label{n3} Let $\phi: R \to R^*$ be an $A$-ring morphism and $z=(z_1,\dots, z_N)\in (R^*)^N$. Then $\phi$ extends uniquely to an $A$-ring
morphism $R\<Z\>\to R^*$ sending $Z_1,\dots, Z_N$ to $z_1,\dots, z_N$, respectively.
\end{theorem}
\begin{proof} Let $g(Z)\in R\<Z\>$. Take $f(X,Z)\in A\<X,Z\>$ and $x\in R^m$ such that $g(Z)=f(x,Z)$, and set $g(z):= f(\phi(x),z)\in R^*$. By Lemma~\ref{n2}
(with $Z$ instead of $Y$) and the usual arguments with dummy variables, this element of $R^*$ depends only on $g(Z)$ and $z$, not on the choice of $m,f,x$. 
Moreover, the map $g(Z)\mapsto g(z): R\<Z\>\to R^*$ is a ring morphism that extends $\phi$ and sends $Z_j$ to $z_j$ for $j=1,\dots,N$. One also verifies easily that for $F\in A\<Y\>$ and $g_1,\dots, g_n\in R\<Z\>$ we have
$$ F(g_1,\dots, g_n)(z)\ =\ F\big(g_1(z),\dots, g_n(z)\big),$$
so this map $R\<Z\> \to R^*$ is an $A$-ring morphism.  
\end{proof} 

\noindent
We retain the notation $g(z)$ introduced in the proof above. In Theorem~\ref{n3}, 
$g\in R\<Z_1,\dots, Z_i\>$ with $i\leqslant N$ gives
$g(z_1,\dots, z_i)=g(z_1,\dots, z_N)$ where on the right we take $g$ as an element of $R\<Z\>$.  For $R=R^*$ and $\phi$ the identity on $R$ this theorem gives the evaluation map
$g\mapsto g(z): R\<Z\>\to R$ promised earlier as a morphism of $A$-rings. It is also a morphism of $R$-algebras.

\begin{lemma}\label{kerev} Let $z\in R^N$. Then the kernel of the morphism $g\mapsto g(z): R\<Z\>\to R$ of $R$-algebras is the ideal $(Z_1-z_1,\dots, Z_N-z_N)R\<Z\>$ of $R\<Z\>$.
\end{lemma}
\begin{proof} For $N\geqslant 1$, Lemma~\ref{d1} gives $R\<Z\>=(Z_N-z_N)R\<Z\> + R\<Z_1,\dots, Z_{N-1}\>$. Proceeding inductively we obtain $R\<Z\> = (Z_1-z_1,\dots, Z_N-z_N)R\<Z\>+R$, which gives the desired result.
\end{proof}

\noindent
Note also that the
map $$R\<Z\>\to \text{ring of $R$-valued functions on $R^N$}$$ assigning to each
$g\in R\<Z\>$ the function $z\mapsto g(z)$ is an $R$-algebra morphism.

\bigskip\noindent
Another special case of Theorem~\ref{n3}: let $R^*$ be an $A$-ring extending $R$, let $\phi$ be the resulting inclusion $R\to R^*\<Z\>$, and $z_j:=Z_j\in R^*\<Z\>$ for $j=1,\dots,N$. Then the corresponding $A$-ring morphism
$R\<Z\> \to R^*\<Z\>$ is a restriction of the inclusion $R[[Z]]\to R^*[[Z]]$ and sends 
$f(x,Z)\in R\<Z\>$ for $f\in A\<X,Z\>$ and $x\in R^m$ to $f(x,Z)\in R^*\<Z\>$. We identify $R\<Z\>$ with an 
$A$-subring of $R^*\<Z\>$ via this morphism.  Thus in the situation of Lemma~\ref{gen} we have 
$$R\<y\>\ =\ \{g(y):\ g\in R\<Y\>\}.$$

\noindent
Let $I$ be an ideal of $R$. Then the canonical map $R\to R/I$ extends to the morphism
$R\<Z\>\to (R/I)\<Z\>$ of $A$-rings sending $Z_j$ to $Z_j$ for $j=1,\dots,N$, and we have:

\begin{lemma}\label{kermodI} The kernel of the above morphism $R\<Z\>\to (R/I)\<Z\>$ is $IR\<Z\>$.
\end{lemma}
\begin{proof} With $\beta$ ranging over $\N^N$, this morphism is a restriction of the ring morphism
$$\sum_{\beta} c_{\beta}Z^\beta\mapsto \sum_{\alpha} (c_{\beta}+I)Z^\beta\ :\ R[[Z]]\to (R/I)[[Z]] \qquad(\text{ all }c_{\beta}\in R),$$
so $IR\<Z\>$ is contained in the kernel. Suppose $f(x,Z)$ is in the kernel where $x\in R^m$ and $f(X,Z)=\sum_{\beta} a_{\beta}(X)Z^\beta\in A\<X,Z\>$. Then all $a_{\beta}(x)\in I$, and
since for some $d\geqslant 1$ we have an equality $f(X,Z)=\sum_{|\alpha|<d} a_{\alpha}(X)f_{\alpha}(X,Z)$
with $\alpha$ ranging over $\N^N$ and all $f_{\alpha}(X,Z)\in A\<X,Z\>$, substitution of $x$ for $X$ gives $f(x,Z)\in IR\<Z\>$.   
\end{proof}

\begin{corollary} Let $J:= \sqrt{\smallo(A)R}$. Then $JR\<Z\>=\sqrt{\smallo(A)R\<Z\>}$. 
\end{corollary}
\begin{proof} We have  $\smallo(A)R\<Z\>\subseteq JR\<Z\>\subseteq \sqrt{\smallo(A)R\<Z\>}$. It remains to note that $JR\<Z\>$ is a radical ideal of $R\<Z\>$, by Lemma~\ref{kermodI}
and a part of Corollary~\ref{ry}.
\end{proof}

 \noindent
 Let $x\in R^m$, construe $R$ as an $(A,x)$-ring, so $R$ is equipped with a certain $A\<X\>$-analytic structure, and let
 $\phi: R \to R^*$ be an $A$-ring morphism. Then $\phi: R \to R^*$ is also an $A\<X\>$-ring morphism where we construe $R^*$ as an $(A,\phi(x))$-ring, with $\phi(x):=\big(\phi(x_1),\dots, \phi(x_m)\big)$. For $z\in (R^*)^N$ the unique extension of $\phi$ to an $A$-ring
 morphism $R\<Z\>\to R^*$ sending $Z_1,\dots, Z_N$ to $z_1,\dots, z_N$ is also an $A\<X\>$-ring morphism.  In other words, for $g\in R\<Z\>$ and $z\in (R^*)^N$
 the two ways of interpreting $g(z)$ give the same  element of $R^*$, and so this raises no conflict of notation.

\subsection*{Substituting elements of $R\<Z\>$ in elements of $R\<Y\>$}
Here is another case of Theorem~\ref{n3}:
let $g_1,\dots, g_n\in R\<Z\>$ and $\phi: R \to R\<Z\>$ the inclusion map.
Then $\phi$ extends uniquely to the $A$-ring morphism 
$$f\mapsto f(g_1,\dots, g_n)\ :\  R\<Y\>\to R\<Z\>$$ that sends $Y_1,\dots, Y_n$ to $g_1,\dots, g_n$. For $z\in R^N$ we have
$$ f(g_1,\dots,g_n)(z)\ =\ f\big(g_1(z),\dots, g_n(z)\big). $$
This follows for example from the uniqueness in Theorem~\ref{n3}. 

Let now $n,d\geqslant 1$. For $N=n$ and $Y=Z$ this yields the automorphism 
$$f(Y)\mapsto f\big(T_d(Y)\big)$$ of the
$A$-ring $R\<Y\>$ and the $R$-algebra $R\<Y\>$, with inverse $g(Y)\mapsto g\big(T_d^{-1}(Y)\big)$. 

\medskip\noindent
Let $f\in A\<X,Z\>$, $g_1(X,Z),\dots, g_N(X,Z)\in A\<X,Z\>$, and set
$$h(X,Z)\ :=\ f\big(X, g_1(X,Z),\dots, g_N(X,Z)\big)\in A\<X, Z\>.$$
Then for $x\in R^m$ we can interpret $f\big(x,g_1(x,Z),\dots, g_N(x,Z)\big)$ on the one hand as $h(x,Z)\in R\<Z\>$, and on the other hand
as the element of $R\<Z\>$ obtained by evaluating
$f$ at the point $\big(x,g_1(x,Z),\dots, g_N(x,Z)\big)\in R\<Z\>^{m+N}$ according to the $A$-analytic structure
we gave $R\<Z\>$. By the uniqueness in Theorem~\ref{n3} these two elements of $R\<Z\>$ are equal, so this raises no conflict of notation.

\subsection*{Introducing $K\<Y\>$}  Let the $A$-ring $R$
 be a domain with fraction field $K$. Set
$$K\<Y\>\ :=\ \{c^{-1}g(Y):\ c\in R^{\ne},\ g(Y)\in R\<Y\>\subseteq K[[Y]]\}.$$
Thus $K\<Y\>$ is a subring of $K[[Y]]$ and contains $R\<Y\>$ as a subring.
For $n,d\geqslant 1$ the automorphism $g(Y)\mapsto g\big(T_d(Y)\big)$ of the
$A$-ring $R\<Y\>$ extends (uniquely) to an automorphism of the
$K$-algebra $K\<Y\>$, also to be indicated by $g\mapsto g\big(T_d(Y)\big)$.

\begin{lemma}\label{intKR} $K\<Y\>\cap R[[Y]] = R\<Y\>$, inside the ambient ring $K[[Y]]$.
\end{lemma}  
\begin{proof} The inclusion $\supseteq$ is clear. For the reverse inclusion, let
$g\in K\<Y\>\cap R[[Y]]$. Now $g=c^{-1}\sum_{\nu}a_{\nu}(x)Y^{\nu}$ with
$c\in R^{\ne}$ and $\sum_{\nu}a_{\nu}(X)Y^\nu\in A\<X,Y\>$, $x\in R^m$. 
By Lemma~\ref{n1} applied to $A\<X\>$ instead of $A$ we have
$d\in \N^{\geqslant 1}$ such that for all $\nu$ with $|\nu|\geqslant d$ we have
$a_{\nu}(X)=\sum_{|\mu|<d}a_{\mu}(X)b_{\mu\nu}(X)$ where the $b_{\mu\nu}\in \smallo(A\<X\>)$ are chosen such that $b_{\mu\nu}\to 0$ as $|\nu|\to \infty$ for each fixed $\mu$ with $|\mu|<d$. Put $u_{\mu}:=c^{-1}a_{\mu}(x)\in R$ for
$|\mu|<d$. Then with a tuple $U=(U_{\mu})_{|\mu|<d}$ of new variables and setting
$$F(X,U,Y)\ :=\ \sum_{|\mu|<d} U_{\mu}Y^\nu + \sum_{|\nu|\geqslant d}\big(\sum_{|\mu|<d}b_{\mu\nu}(X)U_{\mu}\big)Y^\nu\in A\<X,U,Y\>$$
we have $g(Y)=F(x,u,Y)\in R\<Y\>$. 
\end{proof} 

\noindent
For $y\in R^n$ and $f(Y)=c^{-1}g(Y)\in K\<Y\>$ with $c\in R^{\ne}, g(Y)\in R\<Y\>$, the element $c^{-1}g(y)\in K$ depends only on $f,y$, not on $c,g$, and so we can define $f(y):=c^{-1}g(y)$. The map $f\mapsto f(y): K\<Y\>\to K$ is a
$K$-algebra morphism, extending the evaluation maps $K[Y]\to K$ and $R[Y]\to R$ sending $Y_1,\dots, Y_n$ to $y_1,\dots,y_n$, respectively. By Lemma~\ref{kerev} its kernel is the maximal ideal $(Y_1-y_1,\dots, Y_n-y_n)K\<Y\>$ of
$K\<Y\>$. 

\medskip\noindent
Suppose the $A$-ring $S$ extends $R$, and is a domain with
fraction field $L$ taken as a field extension of $K$. Then $L[[Y]]$ has
subrings $R\<Y\>,\ K\<Y\>,\ S\<Y\>,\ L\<Y\>$ with $K\<Y\>\subseteq L\<Y\>$. In this situation we have:

\begin{lemma}\label{lembasis} Assume $S$ is integral over $R$ and
 $b_1,\dots, b_m$ is a basis of the $K$-linear space $L$. Then $L\<Y\>$ is a free $K\<Y\>$-module with
 basis $b_1,\dots, b_m$.
\end{lemma}
\begin{proof} Let $g\in L\<Y\>$ and take $c\in L^\times$ such that $cg\in S\<Y\>$.  Corollary~\ref{corintgeny} tells us that $S\<Y\>$ is generated as a ring over its subring $R\<Y\>$ by $S$, so $cg=\sum_{j\in J} a_jg_j$ with finite $J$, and $a_j\in S$, $g_j\in R\<Y\>$ for all $j\in J$, so $g=\sum_j c^{-1}a_jg_j$. 
Each $c^{-1}a_j$ is a $K$-linear combination of $b_1,\dots, b_m$, so
$g=b_1f_1+ \cdots + b_mf_m$ with $f_1,\dots, f_m\in K\<Y\>$. 
Moreover, if $f_1,\dots, f_m\in K\<Y\>$ are not all zero, then $b_1f_1+\cdots + b_mf_m\ne 0$, by considering a monomial $Y^\nu$ for which one of the $f_i$ has a nonzero coefficient.
\end{proof}

\section{Valuation rings with $A$-analytic structure} \label{vrna}

\noindent
We begin with generalities that do not require $A$ to be noetherian. A {\em valuation $A$-ring\/} is an $A$-ring whose underlying ring is a valuation ring. An {\em $A$-field\/} is a valued field whose valuation ring is equipped with an $A$-analytic structure making it a valuation $A$-ring. The language
$\mathcal{L}^A_{\preccurlyeq, D}$ is the language $\mathcal{L}^A$ of $A$-rings augmented with a binary relation symbol $\preccurlyeq$
(to encode the valuation) and a binary function symbol $D$ (for restricted division). We construe an $A$-field $K$ as an $\mathcal{L}^A_{\preccurlyeq, D}$-structure as follows, where $R$ is the valuation $A$-ring of $K$: 
\begin{itemize} 
\item any $f\in A\<Y\>$ is interpreted as the $n$-ary operation on $K$ giving $f(y)$ its $A$-ring value in $R$ for $y\in R^n$, and
$f(y):=0$ for $y\notin R^n$;
\item $y_1\preccurlyeq y_2\   \Leftrightarrow  y_1\in y_2R$;
\item $D(y_1,y_2)=y_1/y_2$ if $y_1\preccurlyeq y_2\ne 0$, and $D(y_1,y_2):=0$ otherwise.
\end{itemize}   
This makes $R$ an $\mathcal{L}^A_{\preccurlyeq, D}$-substructure of $K$. We define an {\em $A$-extension of $K$\/} to be a valued field extension of $K$ whose valuation ring is equipped with an $A$-analytic structure that makes it an extension of the $A$-ring $R$.  Thus any $A$-extension $L$  of
$K$ is an $A$-field and naturally an $\L^A_{\preccurlyeq,D}$-structure so that $K$ is a substructure of $L$.

\begin{lemma} \label{lem10.0} Let $K$ be an $A$-field with valuation $A$-ring $R$. Suppose $E$ is an $\L^A_{\preccurlyeq, D}$-substructure of $K$. Then: \begin{enumerate}
\item[(i)] $R_E:=\{a\in E:\ a\preccurlyeq 1\}$ is an $A$-subring of $R$ and is a valuation ring dominated by $R$;
\item[(ii)] $E\subseteq \Frac(R_E)\subseteq K$, and $E$ is a valuation ring.
\end{enumerate}
\end{lemma}
\begin{proof} It is clear that $R_E$ is an $A$-subring of $R$. For $a,b\in R_E$, either $a\preccurlyeq b$, and then
$a=D(a,b)b$, $D(a,b)\in R_E$, so $a\in R_Eb$, or $b\preccurlyeq a$, and then likewise $b\in R_E a$. Thus $R_E$ is a valuation ring. 
It is also clear that $R$ dominates $R_E$. For $a\in E$, if $a\preccurlyeq 1$, then $a\in R_E$, so $a\in \Frac(R_E)$, and if
$a\succ 1$, then $D(1,a)=a^{-1}\in R_E$, and so again $a\in \Frac(R_E)$. Since $R_E$ is a valuation ring of $\Frac(R_E)$, so is $E$. 
\end{proof}

\subsection *{The $A$-extension generated over $K$ by an element $z$} The language $\L_{\preccurlyeq,D}^{A,K}$ is 
$\L_{\preccurlyeq,D}^A$ augmented by names (constant symbols), one for each element of $K$, and we construe $K$ and any $A$-extension of it
accordingly as an $\L_{\preccurlyeq,D}^{A,K}$-structure.  Let $L$ be an $A$-extension of $K$. 
Let $Z$ be an indeterminate and $z\in L$. Then  any $\L^{A, K}_{\preccurlyeq, D}$-term 
$\tau(Z)$ yields an element $\tau(z)\in L$.  The set $\{\tau(z):\ \tau(Z) \text{ is an $\L^{A, K}_{\preccurlyeq, D}$-term}\}$ underlies a substructure
of the  $\L^A_{\preccurlyeq,D}$-structure $L$, namely the smallest substructure
of the $\L^A_{\preccurlyeq,D}$-structure $L$ that contains $K\cup\{z\}$; we do not claim this set
is the underlying set of a subfield of $L$. Instead we call attention to the $A$-closed subring $R_z$ of $R_L$ with underlying set
$$  \{\tau(z):\ \tau(Z) \text{ is an $\L^{A, K}_{\preccurlyeq, D}$-term and }\tau(z)\preccurlyeq 1\}.$$
Note: $R\subseteq R_z$; if $z\preccurlyeq 1$, then $z\in R_z$; if $z\succ 1$, then $z^{-1}\in R_z$. 

\begin{lemma}\label{z-ring}  $R_z$  is a valuation ring dominated by  $R_L$. 
\end{lemma} 
 \begin{proof} For $L$ in the role of $K$ this is a special case of Lemma~\ref{lem10.0}(i).
 \end{proof} 

\noindent
Let $K_z$ be the fraction field of $R_z$ inside $L$, equipped with the valuation $A$-ring $R_z$. This makes
$K_z$ into an $A$-extension of $K$, and an $\L^A_{\preccurlyeq, D}$-substructure of $L$. Thus $\tau(z)\in K_z$
for every $\L^{A,K}_{\preccurlyeq, D}$-term $\tau(Z)$. 

\begin{corollary}\label{type} $K_z$ is the smallest substructure of the $\L^A_{\preccurlyeq,D}$-structure $L$ that contains $K\cup\{z\}$ and whose underlying ring is a field. As a consequence, if $z\preccurlyeq 1$, then $R_z$ is the smallest $A$-closed subring of $R_L$
that contains $R\cup\{z\}$ and whose underlying ring is a valuation ring dominated by $R_L$. 
\end{corollary}

\noindent
In the rest of this section, $A$ is noetherian with an ideal $\smallo(A)\ne A$, such that $\bigcap_e \smallo(A)^e=\{0\}$ and
$A$ is $\smallo(A)$-adically complete. Also, $R$ is always a valuation $A$-ring, and we let $\smallo(R)$ denote the maximal ideal of $R$. Thus $\smallo(A)R=t R$ for some $t\in \smallo(A)$.   We let $\k=R/\smallo(R)$ denote the residue field of $R$, and $K$ its fraction field.  We construe the $A$-field $K$ as an $\L^A_{\preccurlyeq, D}$-structure as described earlier.
 
 
\subsection*{Viability}  We define $R$ to be {\em viable\/} if $\smallo(A)R=\smallo(R)\ne \{0\}$ (so $R$ is not a field). (This notion of viability is simpler and stricter than in \cite{Neer}, where extra flexibility was needed to be able to pass to $A$-extensions of $K$ of finite degree. In the present set-up we we don't need to do that.)  
In order to make our Weierstrass preparation and division theorems useful for the model theory of $R$ as a
valuation $A$-ring we assume in the rest of this section:
$$ R\ \text{\em is viable}.$$
Thus we have $t\in \smallo(A)$ with $\iota_0(t)\ne 0$ and $\smallo(R)=tR$, and then $vt:=v\big(\iota_0(t)\big)$ is the smallest positive element of $\Gamma$. Below we fix such $t$ and identify $\Z$ with its image in $\Gamma$ via $k\mapsto k\cdot vt$, so $vt=1$ and $\Z$ is a convex subgroup of $\Gamma$. It is clear that viability is inherited by $A$-subfields:

\begin{lemma}\label{subinher}  Suppose $K_0$ is a valued subfield of $K$ and its valuation ring $R_0=R\cap K_0$ is an $A$-subring of $R$. 
Then the valuation $A$-ring $R_0$ is viable.
\end{lemma}

\medskip\noindent
Note that $R$ is henselian, by Lemma~\ref{he5}, so for any field extension $F$ of $K$ which is algebraic over $K$ there is a unique valuation ring of $F$ lying over $R$, and this valuation ring is the integral closure of $R$ in $F$. Thus by Corollary~\ref{Aint4}: 

\begin{corollary}\label{algext} If $L$ is a valued field extension of $K$ and is algebraic over $K$, then
$L$ has a unique expansion to an $A$-extension of $K$. 
\end{corollary}

\noindent
In this corollary $L$ might be an algebraic closure of $K$, in which case its valuation ring is the integral closure of $R$ in $L$, and  unlike the maximal ideal of $R$, the maximal ideal of this integral closure is not principal.  

\begin{corollary}\label{algtype}
If $z$ is algebraic over $K$, then $K(z)$ is the underlying field of $K_z$.
\end{corollary}
\begin{proof} Suppose $z$ is algebraic over $K$. Then
the valued subfield $K(z)$ of $L$ expands uniquely to an $A$-extension of $K$ by Lemma~\ref{algext}. This $A$-extension is then an
$\L_{\preccurlyeq,D}^{A}$-substructure of $L$ by Corollary~\ref{Aint4}. Now use Corollary~\ref{type}.
\end{proof}

\noindent
By the viability assumption on $R$ the model theoretic results at the end of this paper do not apply to
algebraically closed valued fields whose valuation ring is equipped with an $A$-analytic structure.
To avoid this viability assumption one could replace the restricted power series rings over $A$ with rings of separated power series over $A$ where some variables range as before over the valuation ring and the other (formal) variables only over its maximal ideal.
This is the direction taken by Lipshitz~\cite{L}; see also Lipshitz and Robinson~\cite{LR}. Our treatment can probably be extended in this direction as well, but this will not be done here.  

An $A$-extension of $K$ is said to be {\em viable\/} if its valuation $A$-ring is viable.

\subsection*{Weierstrass preparation and division with parameters} 

Let 
 $$f\ =\ \sum_{\nu}a_{\nu}(X)Y^\nu\in A\<X,Y\>,\qquad n\geqslant 1.$$ We now study how Weierstrass preparation
 applies to $f(x,Y)$ for $x\in R^m$, and how this depends on $x$.  Lemma~\ref{n1} with $A\<X\>$ in the role of $A$
 gives $d\geqslant 1$ and  $b_{\mu\nu}\in \smallo(A\<X\>)$ for $|\mu|< d$ and $|\nu|\geqslant d$. As before we set for $|\mu|<d$,    
$$f_{\mu}\ :=\ Y^\mu + \sum_{|\nu|\geqslant d}b_{\mu\nu}Y^\nu\in A\<X,Y\>,\qquad
      f\ =\ \sum_{|\mu|<d} a_{\mu}f_{\mu}.$$
We order $\N^n$ lexicographically and for $\mu$ with $|\mu|< d$ we set
\begin{align*} I(\mu)\ &:=\ \{\lambda:\ |\lambda|<d,\ \lambda < \mu\},\quad
        J(\mu)\ :=\ \{\lambda:\ |\lambda|<d,\ \lambda>\mu\},\ \text{ so}\\
(*)  \qquad \qquad    f\ &=\ \sum_{\lambda\in I(\mu)} a_{\lambda}f_{\lambda} + a_{\mu}f_{\mu} + \sum_{\lambda\in J(\mu)}a_{\lambda}f_{\lambda}.
\end{align*} 
Now fix $\mu$ with $|\mu|<d$ and introduce tuples $$U_{\mu}:=\big(U_{\lambda\mu}:\lambda\in I(\mu)\big), \qquad V_{\mu}:=\big(V_{\lambda\mu}:\lambda\in J(\mu)\big)$$
of indeterminates, different from each other and from the $X_i$ and $Y_j$. Set
\begin{align*} \tilde{F}_\mu\ &:=\ \sum_{\lambda\in I(\mu)}U_{\lambda\mu}f_{\lambda} + f_{\mu} + \sum_{\lambda\in J(\mu)}tV_{\lambda\mu}f_{\lambda}\in A\<U_{\mu}, V_{\mu}, X,Y\>,\\
F_{\mu}\ &:=\ \tilde{F}_{\mu}\big(U_\mu, V_{\mu}, X, T_d(Y)\big)\in A\<U_{\mu}, V_{\mu}, X,Y\>.
\end{align*}
Note that for $n=1$ we have $T_d(Y)=Y$, so $F_{\mu}=\tilde{F}_{\mu}$.

\begin{lemma}\label{reg} $F_{\mu}$ is regular
of degree $\ell:=\mu_1d^{n-1}+ \cdots + \mu_n$ in $Y_n$, and so
$$F_{\mu}\ =\ E\cdot (Y_n^{\ell} + G_1Y_n^{\ell -1} + \cdots + G_{\ell})$$
for a unit $E$ of $A\<U_{\mu}, V_{\mu}, X,Y\>$ and suitable $G_1,\dots, G_{\ell}\in A\<U_{\mu}, V_{\mu}, X,Y'\>$.
\end{lemma}           
      
\noindent
Here is a consequence of Lemma~\ref{reg} for $n=1$ (so $Y=Y_1$):

\begin{corollary}\label{correg} Let $n=1$ and $g(Y)=\sum_{j=0}^\infty c_{j}Y^j\in K\<Y\>$, $g\ne 0$. Then:
\begin{enumerate}
\item[(i)] there is $\mu\in \N$ with $c_i\preccurlyeq c_\mu\succ c_j$ whenever $i\leqslant \mu < j$;
\item[(ii)] for the unique $\mu$ in $\mathrm{(i)}$ we have
$g(Y)=c\cdot r(Y)\cdot(Y^\mu+g_1Y^{\mu-1} + \cdots + g_\mu)$ with $c=c_\mu\in K^\times$, $r(Y)\in R\<Y\>^\times$, and $g_1,\dots, g_\mu\in R$. 
\end{enumerate} 
\end{corollary}
\begin{proof} We multiply $g$ by an element of $K^\times$ to arrange $g\in R\<Y\>$.
Then $g(Y)= f(x,Y)$ with $x\in R^m$ and $f=f(X,Y)=\sum_ja_j(X)Y^j$ in $A\<X,Y\>$,
so $c_j=a_j(x)$ for all $j$.  
Lemma~\ref{n1} with $A\<X\>$ in the role of $A$ gives $d\geqslant 1$ and $b_{ij}\in \smallo(A\<X\>)$ for $i<d\leqslant j$ such that $a_j=\sum_{i<d}a_ib_{ij}$ for all $j\geqslant d$. Set 
$$\gamma\ :=\  \min_{i<d}v(c_i), \qquad \mu\ :=\ \max\{i<d:\ v(c_i)=\gamma\}.$$ 
Then (i) holds for this $\mu$: for $\mu<j$, distinguish the cases $j< d$ and $j\geqslant d$.  

For (ii) we use the identities above for $n=1$ and our $f$.  The identity $(*)$ yields
$ f =\sum_{i<\mu} a_if_i + a_{\mu}f_\mu + \sum_{\mu<i<d} a_if_i$. Substituting $x$ for $X$ and factoring out 
$c: =c_\mu=a_\mu(x)$ (possible because $c\ne 0$) gives
$$c^{-1}g(Y)\ =\  \sum_{i<\mu}(c_i/c) f_i(x,Y) + f_\mu(x,Y) + \sum_{\mu<i<d}(c_i/c)f_i(x,Y),$$
so for $u:= \big(c_i/c: i<\mu\big)\in R^\mu$ and $v:= \big(c_i/t c:\ \mu<i<d\big)\in R^{d-1-\mu}$ we have
$c^{-1}g(Y)=F_\mu(u,v,x,Y)$. Now applying Lemma~\ref{reg} for $n=1$ shows that (ii) holds with
$r(Y)=E(u,v,x,Y)$ and $g_i=G_i(u,v,x)$ for $i=1,\dots,\mu$.
\end{proof}

\noindent
Note that the proof above uses in a crucial way that $\smallo(R)=t R$.

\begin{corollary} \label{correg+} Let $R^*$ be an $A$-ring extending $R$, and suppose $y\in R^*$ is
not integral over  $R$. Then $R\<y\>$ has the following properties, with $n=1$ in $(${\rm i}$)$: \begin{enumerate}
\item[(i)] the morphism $g(Y)\mapsto g(y): R\<Y\>\to R\<y\>$ of $A$-rings is an isomorphism;
\item[(ii)] $R\<y\>$ is a domain but not a valuation ring;
\item[(iii)] inside the ambient field $\operatorname{Frac}(R\<y\>)$ we have $R\<y\>\nsubseteq K(y)$.
\end{enumerate}
\end{corollary} 
\begin{proof} For (ii), use that $Y\notin t R\<Y\>$ and $t\notin YR\<Y\>$. For (iii), if $\operatorname{char} \k \ne 2$, then the polynomial $Z^2-(1+ty)$ has a zero in $R\<y\>$ by Corollary~\ref{he6}, but has no zero in $K(y)$. If $\operatorname{char} \k =2$, use instead the polynomial $Z^3-(1+ty)$. 
\end{proof}

\medskip\noindent
We return to our $f(X,Y)\in A\<X,Y\>$ with $n\geqslant 1$. To find out how Weierstrass preparation for  $f(x,Y)$ depends on $x\in R^m$, we now introduce the quantifier-free $\L^A_{\preccurlyeq}$-formulas $Z(X)$ and $S_{\mu}(X)$ (for $|\mu|<d$) in the variables $X$: 
\begin{align*} Z(X)\ &:=\ \bigwedge_{|\mu|<d} a_{\mu}(X)=0,\\
S_{\mu}(X)\ &:=\ a_{\mu}(X)\ne 0\wedge \big(\bigwedge_{\lambda\in I(\mu)}a_{\lambda}(X)\preccurlyeq a_{\mu}(X)\big) \wedge \big(\bigwedge_{\mu\in J(\mu)} a_{\lambda}(X)\prec a_{\mu}(X)\ \big).
\end{align*} 

\begin{lemma}\label{reg+} For the $\L^A_{\preccurlyeq}$-structure $R$ we have the following: \begin{enumerate}
\item[(i)] for all $x\in R^m$, $Z(x)$ holds or $S_\mu(x)$ holds for some $\mu$ with $|\mu|<d$;
\item[(ii)] suppose $x\in R^m$, $|\mu|<d$, and $S_{\mu}(x)$ holds; so $u_{\lambda\mu}:= a_{\lambda}(x)/a_{\mu}(x)\in R$ for $\lambda\in I(\mu)$  and
$v_{\lambda\mu}:= a_{\lambda}(x)/t a_{\mu}(x)\in R$ for $\lambda\in J(\mu)$. Then with $$u_{\mu}\ :=\ \big(u_{\lambda\mu}:\lambda\in I(\mu)\big), \qquad 
v_{\mu}\ :=\ \big(v_{\lambda\mu}:\lambda\in J(\mu)\big),$$
and $E, G_1,\dots, G_{\ell}$ as in Lemma~\ref{reg} we have 
$$f\big(x, T_d(Y)\big)\ =\ a_{\mu}(x)F_{\mu}(u_{\mu}, v_{\mu},x,Y) \text{ in }R\<Y\>$$ 
and $F_{\mu}(u_{\mu}, v_{\mu}, x,Y)$ equals, in $R\<Y\>$,  the product $$\quad E(u_{\mu}, v_{\mu},x,Y)\cdot\big(Y_n^{\ell}+G_1(u_{\mu}, v_{\mu},x, Y')Y_n^{\ell -1}+\cdots + G_{\ell}(u_{\mu}, v_{\mu},x, Y')\big).
$$
\end{enumerate}
\end{lemma} 

\noindent
We can now prove a converse of Lemma~\ref{n2}:  

\begin{lemma} Suppose $x\in R^m$ and
$f(x,y)=0$ for all $y\in R^n$. Then $f(x,Y)=0$. 
\end{lemma}
\begin{proof} If $Z(x)$ holds, then $a_{\nu}(x)=0$ for all $\nu$, that is, $f(x,Y)=0$. Next assume $|\mu|< d$ and $S_{\mu}(x)\ne 0$. Then
by (ii) of Lemma~\ref{reg+} we have a monic polynomial in $R[Y_n]$ vanishing identically on $R$. This is  impossible as $R$ is infinite.
\end{proof}  

\begin{corollary} If $g\in R\<Y\>$ and $g(y)=0$ for all $y\in R^n$, then $g=0$. 
\end{corollary}          

\noindent
By the last corollary, the map 
$$ K\<Y\> \to \text{ring of $K$-valued functions on $R^n$}$$
that assigns to each $g\in K\<Y\>$ the function $y\mapsto g(y)$ on $R^n$
is an injective morphism of $K$-algebras. 

\subsection*{Consequences for $K\<Y\>$ of Weierstrass division}
For an algebraic closure $K_{\alg}$ of $K$, the integral closure
$R_{\alg}$ of $R$ in $K_{\alg}$ is the unique valuation ring of $K_{\alg}$
dominating $R$, and has a unique
$A$-analytic structure extending that of $R$.

More generally, we fix below an algebraically closed valued field extension
$K^{\a}$ of $K$ (not necessarily an algebraic closure of $K$), whose valuation ring $R^{\a}$ is equipped with an $A$-analytic structure extending that of $R$. This gives rise to
$K\<Y\>\subseteq K^{\a}\<Y\>$ and for $y\in (R^{\a})^n$ we have the evaluation map
$g\mapsto g(y): K^{\a}\<Y\>\to K^{\a}$, which for $y\in R^n$ extends the previous evaluation map
$ K\<Y\>\to K$.  

\begin{lemma} If $E$ is a unit of $R\<Y\>$, then $E(y)\asymp 1$ for all $y\in (R^{\a})^n$.
\end{lemma} 

\noindent
This is clear. The next two lemmas follow easily from $(*)$ and Lemma~\ref{reg+}.

\begin{lemma}\label{wdp1} Let $g(Y)=\sum_{\nu}c_{\nu}Y^\nu\in R\<Y\>$, $g\ne 0$. Then:
\begin{enumerate}
\item[(i)] there is a $d\geqslant 1$ and an index $\mu\in \N^n$  with $|\mu|<d$ such that
$$ c_{\nu}\preccurlyeq c_{\mu} \text{ whenever }|\nu| < d, \quad c_{\nu}\prec c_{\mu} \text{ whenever }|\nu|\geqslant d;$$
\item[(ii)] if $c_{\nu}\prec 1$ for all $\nu$, then $g(y)\prec 1$ for all $y\in (R^{\a})^n$.
\end{enumerate} 
\end{lemma}

\begin{lemma}\label{wdp2} Let $g(Y)\in K\<Y\>^{\ne}$, $n\geqslant 1$. Then for some $d\in \N^{\geqslant 1}$ and $\ell\in \N$,\begin{enumerate}
\item[(i)] $g\big(T_d(Y)\big)\ =\ c\cdot E(Y)\cdot\big(Y_n^\ell + c_1(Y')Y_n^{\ell -1} + \cdots + c_l(Y')\big)$ \newline
where $c\in K^\times$, $E(Y)\in R\<Y\>$ is a unit, and $c_1(Y'),\dots, c_\ell(Y')\in R\<Y'\>$.
\item[(ii)] $R\<Y\>\ =\ \big(Y_n^\ell + c_1(Y')Y_n^{\ell -1} + \cdots + c_l(Y')\big)R\<Y\> + \sum_{i< \ell}R\<Y'\>Y_n^i$ and
$$ K\<Y\>\ =\ g\big(T_d(Y)\big)K\<Y\> + \sum_{i< \ell}K\<Y'\>Y_n^i.$$
\end{enumerate}
\end{lemma}
\begin{proof} For (ii), use a reduction to $A\<X,Y\>$ and appeal to Lemma~\ref{d1}. 
\end{proof}

\noindent
Weierstrass division leads in the usual way to noetherianity of $K\<Y\>$ and more:

\begin{theorem}\label{wdp4} The integral domain $K\<Y\>$ has the following properties:
\begin{enumerate}
\item[(i)] $K\<Y\>$ is noetherian, 
\end{enumerate}
and for every proper ideal $I$ of $K\<Y\>$: \begin{enumerate}
\item[(ii)] there is an injective $K$-algebra
morphism $K\<Y_1,\dots, Y_m\>\to K\<Y\>/I$ with $m\leqslant n$, making $K\<Y\>/I$ into a finitely generated $K\<Y_1,\dots, Y_m\>$-module; 
\item[(iii)] there is $y\in (R^{\a})^n$ such that $f(y)=0$ for all $f\in I$.
\end{enumerate} 
\end{theorem}
\begin{proof} By induction on $n$. The case $n=0$ being obvious, let $n\geqslant 1$. 
Recall that for $d\in \N^{\geqslant 1}$ we have the automorphism $g(Y)\mapsto g\big(T_d(Y)\big)$ of the $K$-algebra $K\<Y\>$.  
Let $I$ be an ideal of $K\<Y\>$, $I\ne \{0\}$. Take a nonzero $g\in I$. To show
$I$ is finitely generated we apply an automorphism as above and
use Lemma~\ref{wdp2} to arrange $g= Y_n^\ell + c_1(Y') Y_n^{\ell -1} + \cdots + c_{\ell}(Y')$ with $\ell\in \N$,  $c_1,\dots, c_\ell\in R\<Y'\>$, and
$$R\<Y\>\ =\ gR\<Y\> + \sum_{i<\ell}R\<Y'\>Y_n^{i}, \qquad K\<Y\>\ =\ gK\<Y\> + \sum_{i<\ell}K\<Y'\>Y_n^{i}.$$
For $\ell=0$ this means $g=1$, and we are done, so assume $\ell \geqslant 1$. Then the inclusion $K\<Y'\>\to K\<Y\>$ followed by the canonical map
$K\<Y\>\to K\<Y\>/(g)$ makes $K\<Y\>/(g)$ a $K\<Y'\>$-module that is generated by the images of the $Y_n^i$ with $i<\ell$. Assuming inductively that
$K\<Y'\>$ is noetherian, it follows that $K\<Y\>/(g)$ is noetherian as a $K\<Y'\>$-module, and thus as a ring. Hence the image of $I$ in $K\<Y\>/(g)$ is finitely generated, say by the images of
$g_1,\dots, g_k\in I$, $k\in \N$. Then $I$ is generated by $g, g_1,\dots, g_k$. This proves noetherianity of $K\<Y\>$. Let now $I$ also be proper, that is, $1\notin I$, and set $I':= I\cap K\<Y'\>$. The natural $K$-algebra embedding
$K\<Y'\>/I'\to K\<Y\>/I$ makes $K\<Y\>/I$ a finitely generated $K\<Y'\>/I'$-module by the above. Assuming inductively that (ii) holds
for $n-1,\ K\<Y'\>,\ I'$ instead of $n, K\<Y\>, I$ yields (ii).   For (iii) we can arrange that $I$ is a maximal ideal of $K\<Y\>$.
Then in (ii) we have $m=0$, so $K\<Y\>/I$ is finite-dimensional as a vector space over $K$, hence algebraic over $K$ as a field extension of $K$. This gives a
$K$-algebra morphism $\phi: K\<Y\>\to K^{\a}$ with kernel $I$ and $\phi(K\<Y\>)$ algebraic over $K$.
We set $y:=(y_1,\dots, y_n)=\big(\phi(Y_1),\dots,\phi(Y_n)\big)\in (K^{\a})^n$.  We claim that  $\phi(R\<Y\>)\subseteq R^{\a}$ (and thus $\phi(R\<Y\>)$ is integral over $R$).

Using $\phi(g)=0$ gives $\phi(R\<Y\>)=\sum_{i< \ell} \phi(R\<Y'\>)y_n^i$.
Since $I'$ is a maximal ideal of $K\<Y'\>$ we can assume
inductively that $\phi(R\<Y'\>)\subseteq R^{\a}$, so
$\phi(R\<Y\>)\subseteq \sum_{i<\ell}R^{\a}y_n^i$. Now $\phi(g)=0$ means
$$y_n^\ell + \phi\big(c_1(Y')\big)y_n^{\ell -1} + \cdots + \phi\big(c_{\ell}(Y')\big)\ =\ 0,$$
with $\phi\big(c_j(Y')\big)\in R^{\a}$ for $j=1,\dots,\ell$. Hence 
$y_n\in R^{\a}$, which proves the claim. Therefore $y\in (R^{\a})^n$, and by Corollary~\ref{Aint5} the restriction of $\phi$ to a map $R\<Y\>\to R^{\a}$ is a morphism of $A$-rings. Thus for $f(Y)\in R\<Y\>$ we have 
$\phi\big(f(Y)\big)=f(y)$, in particular, $f(y)=0$ for all $f\in I$.
\end{proof}

\section{Immediate $A$-Extensions}\label{immAE}

\noindent
The study of immediate extensions of valued fields plays a key role in proving AKE-results via model theory and valuation theory. We try to follow this pattern. By Lemma~\ref{he5} and Corollary~\ref{algext}, the case of algebraic immediate extensions is under control (at least in the equicharacteristic $0$ case), so we are left with proving
that a pseudocauchy sequence of transcendental type ``generates''  an immediate extension.  The problem is that the valuation ring of such an extension should now be an $A$-ring,  and thus closed under many more operations than in the non-analytic setting. In this section we show how to overcome this problem.  This section uses only the material of Section~\ref{vrna} that precedes Lemma~\ref{reg+}.

\medskip\noindent
Below we assume some familiarity with \cite[Section 4]{Lou}; when using a result from those lecture notes we shall indicate the specific reference.

\bigskip\noindent
We continue with the previously set assumptions on $A$ and $R$: {\em $A$ is noetherian with an ideal 
$\smallo(A)\ne A$ such that $\bigcap_e \smallo(A)^e=\{0\}$ and $A$ is $\smallo(A)$-adically complete; $R$ is a viable valuation $A$-ring}. We fix $t\in \smallo(A)$ with $\smallo(R)=tR$, and adopt the notations and terminology concerning $R$ and its fraction field $K$ from Section~\ref{vrna},  with the valuation
$v: K^\times \to \Gamma$ on $K$ such that $R=\{a\in K:\ va\geqslant 0\}$, so $vt$ is the least positive element of $\Gamma$. For any valued field extension $L$  of $K$ we let $\Gamma_L\supseteq \Gamma$ be the value group of $L$  and denote the valuation of $L$ also by $v$, so that $v: L^\times \to \Gamma_L$ extends $v: K^\times \to \Gamma$. 

By \cite[Lemma 4.3]{Lou} and the remark following its proof, any pc-sequence in $K$ has a pseudolimit in some elementary $\L^A_{\preccurlyeq}$-extension of  $K$; any such extension is an 
$A$-extension of $K$ whose valuation $A$-ring inherits the conditions we imposed on $R$.

\subsection*{Immediate $A$-extensions generated by a pseudocauchy sequence} In this subsection $L$ is an $A$-extension of $K$. Thus the valuation $A$-ring $S$ of $L$ extends the $A$-ring $R$ and dominates $R$.
We also view any subfield $F$ of $L$ as a {\em valued\/} subfield of $L$, and thus as a valued field 
extension of $K$ if $K\subseteq F$. 

With {\em pc\/}  abbreviating {\em pseudocauchy}, let $(a_{\rho})$ be a pc-sequence in $K$  of transcendental type over $K$, with all $a_{\rho}\in R$, and
with pseudolimit $a\in L$. Then $a\in S$, $a$ is transcendental over $K$, and the valued subfield $K(a)$ of $L$ is an immediate extension of $K$, by \cite[Theorem 4.9]{Lou}. But the valuation ring of $K(a)$ does not
contain $R\<a\>$ by Corollary~\ref{correg+}, and so is not $A$-closed in $S$. 

{\em Is there a valued subfield $K_a\supseteq K(a)$ of $L$  that is an immediate extension of $K$ and whose valuation ring $R_a$ is $A$-closed in $S$?}
Such $R_a$ must contain $R\<a\>$, but has to be strictly larger, since $R\<a\>$ is not a valuation ring, by
Corollary~\ref{correg+}. 
  
 To answer the question above affirmatively we proceed as follows. Take an index $\rho_0$ such that for $\rho>\rho_0$,
 $$a\ =\ a_{\rho} + t_{\rho}u_{\rho}, \quad t_{\rho}\in K^\times,\ t_{\rho}\prec 1,\  u_{\rho}\in K(a),\  u_{\rho}\asymp 1,$$  
and $v(t_{\rho})$ is strictly increasing as a function of $\rho>\rho_0$. Then for indices $\sigma>\rho > \rho_0$ we have $R[u_{\rho}]\subseteq R[u_{\sigma}]$, and thus
$$ R\<a\>\ \subseteq\ R\<u_{\rho}\>\ \subseteq\ R\<u_{\sigma}\>.$$
This yields an $A$-closed subring $R_a:= \bigcup_{\rho>\rho_0} R\<u_{\rho}\>$ of $S$.
Note that $R_a$ does not change upon increasing $\rho_0$, and the next proposition shows more:  as the notation suggests, $R_a$
depends only on $R$ and $a$, not on $(a_{\rho})$.   

\begin{proposition} \label{imm1} The subring $R_a$ of $S$ has the following properties: \begin{enumerate}
\item[(i)] the valued subfield $K_a:= \Frac(R_a)$ of $L$ is an immediate extension of $K$;
\item[(ii)] $R_a$ is the least $A$-closed subring of $S$, with respect to inclusion, that contains $R[a]$ and is a valuation ring dominated by $S$;
\end{enumerate}
\end{proposition}
\begin{proof} Let $P\in K[Y]\setminus K$ where $n=1$, so $Y=Y_1$.  Let $I$ be the set of $i$ in $\{1,\dots, \deg P\}$
with $P_{(i)}(Y)\ne 0$. Then $I\ne \emptyset$ and for all $\rho>\rho_0$,
$$P(a)\ =\ P(a_{\rho}) +\sum_{i\in I} P_{(i)}(a_{\rho})(a-a_{\rho})^i\ =\  P(a_{\rho})+\sum_{i\in I} t_{\rho}^i P_{(i)}(a_{\rho})u_{\rho}^i.
$$  
The proof of \cite[Proposition 4.7]{Lou} gives $i_0\in I$ such that, eventually, 
\begin{align*}  \text{ for all }i\in I\setminus \{i_0\}, &\quad
   t_{\rho}^{i_0}P_{(i_0)}(a_{\rho})\ \succ\   t_{\rho}^{i}P_{(i)}(a_{\rho}),\\
P(a)-P(a_{\rho})\ &\sim\  t_{\rho}^{i_0}P_{(i_0)}(a_{\rho}),
\end{align*} 
and $v\big( t_{\rho}^{i_0}P_{(i_0)}(a_{\rho})\big)=v\big(P(a)-P(a_{\rho})\big)$ is eventually strictly increasing. 
Now $(a_{\rho})$ is of transcendental type over $K$, so $v\big(P(a_{\rho})\big)$ is eventually constant, and thus
 $P(a_{\rho})\succ P(a)-P(a_{\rho})$, eventually. Thus eventually,
 $$P(a)\ =\ P(a_{\rho})\cdot \big(1 + \sum_{i\in I}\frac{t_{\rho}^i P_{(i)}(a_{\rho})}{P(a_{\rho})}u_{\rho}^i\big)\ 
 \in\  P(a_{\rho})\cdot (1+t R\<u_{\rho}\>).$$
Now suppose $Q(Y)\in K[Y]^{\ne}$. Then likewise we have for
 $j=1,\dots, \deg Q$ that eventually $0\ne Q(a_{\rho})\succ t_{\rho}^j Q_{(j)}(a_{\rho})$, so eventually
$$ Q(a)\ =\ Q(a_{\rho})\cdot \big(1 + \sum_{j=1}^{\deg Q}\frac{t_{\rho}^j Q_{(j)}(a_{\rho})}{Q(a_{\rho})}u_{\rho}^j\big)\
\in\ Q(a_{\rho})\cdot (1+t R\<u_{\rho}\>).$$
Therefore, if $P(a)\preccurlyeq Q(a)$, then eventually $\frac{P(a_{\rho})}{Q(a_{\rho})}\in R$,  and so eventually
$$\frac{P(a)}{Q(a)}\ \in\ \frac{P(a_{\rho})}{Q(a_{\rho})}\cdot (1+ t R\<u_{\rho}\>)\ \subseteq\ R\cdot (1+t R\<u_{\rho}\>)\ \subseteq\ R\<u_{\rho}\>.$$
Thus the valuation ring of the valued subfield $K(a)$ of $L$ is contained in $R_a$. Now we use the reduction to  polynomials from Corollary~\ref{correg}(ii) to the effect that for $g,h$ in $R\<Y\>$ with $h\ne 0$, if $g(a)\preccurlyeq h(a)$, then
$g(a)/h(a)\in R_a$.  Thus the valuation ring of the valued subfield $\text{Frac}(R\<a\>)$ of $L$ is contained in $R_a$,
and it also follows from the last display that $\Frac(R\<a\>)$ is an immediate extension of $K$.  

Next, fix $\rho>\rho_0$ and note that for $\sigma> \rho$ we have 
$$u_{\rho}\ =\ a_{\sigma\rho}+t_{\sigma\rho}u_{\sigma}, \qquad a_{\sigma\rho}\ :=\ \frac{a_{\sigma}-a_{\rho}}{t_{\rho}}\in R,\qquad t_{\sigma\rho}\ :=\  \frac{t_{\sigma}}{t_{\rho}},$$
and $(a_{\sigma\rho})_{\sigma>\rho}$ is a pc-sequence in $K$ and of transcendental type over $K$
such that $a_{\sigma\rho}\leadsto \frac{a-a_{\rho}}{t_{\rho}}=u_{\rho}$.  
Hence the above arguments applied to $u_{\rho}$ instead of $a$ show that $\Frac(R\<u_{\rho}\>)$ as a valued subfield of $L$
is an immediate extension of $K$, and that the valuation ring of $\Frac(R\<u_{\rho}\>)$ is contained in
$\bigcup_{\sigma>\rho} R\<u_{\sigma}\>= R_a$.
Taking the union over all $\rho>\rho_0$  and using
$R_a\subseteq S$ yields that $R_a$ is the valuation ring of the valued subfield $K_a:= \Frac(R_a)$ of $L$, and
that $K_a$ is an immediate extension of $K$. This proves  (i) and also shows that  $S$ dominates $R_a$. 

As to (ii), let $R^*$ be any $A$-closed subring of $S$ containing
$R[a]$ such that $R^*$ is  a valuation ring dominated by $S$. Then clearly $u_{\rho}\in R^*$ for all
$\rho>\rho_0$, and thus $R_a\subseteq R^*$.  
\end{proof}

\noindent
We keep $(a_{\rho})$ for now, and show that $K_a$ is essentially unique:

\begin{corollary}\label{Aimmu} Let $L'$ be an $A$-extension of $K$ with valuation $A$-ring $S'$. 
Suppose $a_{\rho} \leadsto a'\in S'$, thus giving rise to $R_{a'}\subseteq K_{a'}\subseteq L'$. Then there is a unique isomorphism $R_a\to R_{a'}$ of $A$-rings that is the identity on $R$ and sends $a$ to $a'$. It extends to a valued field isomorphism $K_a\to K_{a'}$.
\end{corollary}
\begin{proof} Using notations from the proof of Proposition~\ref{imm1} we have $a'=a_{\rho}+ t_{\rho}u_{\rho}' $ with $u_{\rho}'\in K(a')$, $u'_{\rho}\asymp 1$ for $\rho>\rho_0$.  That same proof and Corollary~\ref{correg+} yields for all  $\rho>\rho_0$ a unique isomorphism 
$R\<u_{\rho}\>\to R\<u'_{\rho}\>$ of $A$-rings that is the identity on $R$ and sends $u_{\rho}$ to $u'_{\rho}$.
Moreover, for $\sigma>\rho>\rho_0$ we have 
$$u_{\rho}\ =\ a_{\sigma\rho} + t_{\sigma\rho}u_{\sigma}, \qquad u_{\rho}'\ =\ a_{\sigma\rho} + t_{\sigma\rho}u'_{\sigma},$$ and so the above isomorphism $R\<u_\sigma\>\to R\<u'_{\sigma}\>$ extends the  above isomorphism
$R\<u_{\rho}\>\to R\<u'_{\rho}\>$.
Taking the union over all $\rho>\rho_0$ yields an isomorphism $R_a\to R_{a'}$ of $A$-rings that is the identity on $R$ and sends $a$ to $a'$. 
Any such isomorphism sends $u_{\rho}$ to $u'_{\rho}$ for $\rho>\rho_0$, and this gives uniqueness. Now $R_a$ and $R_{a'}$ are the valuation rings of $K_a$ and $K_{a'}$, so this isomorphism $R_a\to R_{a'}$  extends to an isomorphism $K_a\to K_{a'}$ of valued fields. 
\end{proof}

\subsection*{Uniqueness of maximal immediate extensions over $A$} The results in this subsection about maximal immediate $A$-extensions will not be used later, but are included for their intrinsic interest. So far we did not restrict the characteristic of $\k$ or $K$,  but now we also assume: 

\medskip\noindent
{\em Either $\ch(\k)=0$ $($the equicharacteristic $0$ case$)$, or $K$ as a valued field is finitely ramified of mixed characteristic}. 

\medskip\noindent
This is a well-known sufficient condition for an ordinary valued field to have an essentially unique
maximal immediate extension; see \cite[4.29]{Lou}.  We now adapt this to our $A$-setting. 
A first consequence of
the present assumptions is that $K$ has no proper algebraic immediate $A$-extension, by \cite[Corollary 4.22]{Lou}.  Note that any immediate $A$-extension of $K$ inherits all the conditions we imposed so far on $K$. By a {\em maximal immediate
$A$-extension of $K$\/} we mean an immediate $A$-extension $L$ of $K$ such that  $L$ has no proper immediate
$A$-extension. The previous subsection, the nonexistence of proper algebraic
 immediate $A$-extensions of $K$, and \cite[Section 4]{Lou} yield for an immediate $A$-extension $L$ of $K$ that the following are equivalent: \begin{enumerate}
 \item $L$ is a maximal immediate $A$-extension of $K$,
 \item $L$ is maximal as a valued field, 
 \item $L$ is spherically complete.
 \end{enumerate}


\begin{corollary}\label{im1} $K$ has a maximal immediate $A$-extension, and such an extension is unique up
to $\L^{A}_{\preccurlyeq}$-isomorphism over $K$. 
\end{corollary}
\begin{proof} This goes along the same lines as the proof for ordinary valued fields: First, existence of
a maximal immediate $A$-extension of $K$ follows by Zorn and Krull's cardinality bound, like \cite[Corollary 4.14]{Lou}. As to uniqueness, using Corollary~\ref{Aimmu} this goes as in the proof of \cite[Corollary 4.29]{Lou}.
\end{proof}

\noindent
Using Corollary~\ref{Aimmu} we obtain in the same way: 

\begin{corollary}\label{im2}
Any maximal  immediate $A$-extension of $K$ can be embedded, as an $\L^A_{\preccurlyeq}$-structure, into any $|\Gamma|^+$-saturated $A$-extension of $K$.
\end{corollary}

\section{Truncation}\label{trunc} 

\noindent
The aim of this section is to prove an $A$-version of Kaplansky's embedding theorem from \cite{Kap} ``with truncation''.  This section is not needed for the later AKE-results, but is included for its independent interest.
Returning to the Hahn field example from the beginning of Section~\ref{ras} we are given: \begin{enumerate}
\item  a ring $A_0$ with $1\ne 0$, 
\item $A=A_0[[t]]$ with the norm specified there,
\item a ring morphism $\iota: A_0\to \k$ into a field $\k$,
\item an ordered abelian group $\Gamma$ with a distinguished element $1>0$ (allowing the possibility
that there are $\gamma\in \Gamma$ with $0<\gamma < 1$). 
\end{enumerate}
As in the example mentioned we use this to make the valuation ring $\k[[t^{\Gamma^{\geqslant}}]]$ of the Hahn field $K:=\k(\!(t^\Gamma)\!)$ into an $A$-ring. 

\subsection*{Preserving truncation closedness} 
We refer to \cite{tr} for notations and terminology concerning truncation in
 $K$.  The $A$-closed subrings $\iota_0(A)=\iota(A_0)[[t]]$ and $\k[[t]]$  of $\k[[t^{\Gamma^{\geqslant}}]]$ are also truncation closed.  This subsection uses the ``Hahn field example'' from Section~\ref{ras}, but little else from the present paper.

\begin{lemma}\label{tr1} Let $E$ be a truncation closed subring of $\k[[t^{\Gamma^{\geqslant}}]]$. Then the $A$-closure $R$ of $E$ in $\k[[t^{\Gamma^{\geqslant}}]]$ is 
also truncation closed. 
\end{lemma} 
\begin{proof}  We have $\iota_0(A)=\iota(A_0)[[t]]\subseteq R$, so $E[\iota_0(A)]\subseteq R$.
Now $\iota_0(A)$ is truncation closed, hence $E\cup \iota_0(A)$ is as well, and so is $E[\iota_0(A)]$ by \cite[Corollary 2.5]{tr}. Thus replacing
$E$ by $E[\iota_0(A)]$ we arrange $\iota_0(A)\subseteq E$.  Let $F$ be a truncation closed subring of $R$ containing $E$ such that $F\ne R$;  in view of \cite[Corollary 2.6]{tr} and Zorn it suffices to show that then some element of $R\setminus F$ has all its proper truncations in $F$. Let $n$ be minimal such that there are $y\in F^n$ and $f\in A\<Y\>$ with
$f(y)\notin F$. Because of $\iota_0(A)\subseteq F$ we have $n\geqslant 1$.   With the lexicographic ordering on $n$-tuples $(\lambda_1,\dots, \lambda_n)$ of ordinals we take
$y\in F^n$ with minimal $\big(o(y_1),\dots, o(y_n)\big)$ such that $f(y)\notin F$ for some $f\in A\<Y\>$. Fix such $f$; it suffices to show that then all proper truncations of $f(y)$ lie in $F$.
Minimality of $n$ gives  $y_1,\dots, y_n\ne 0$, so $o(y_1),\dots, o(y_n) \geqslant 1$. 

Let $c$ be a proper truncation of $f(y)$. Take $\gamma\in \{1\}\cup \supp y_1\cup\cdots \cup \supp y_n$ and a positive integer $N$ such that $N\gamma > \supp c$.  We first consider the case $\gamma=1$.
Then $f(Y)=P(Y)+ t^NQ(Y)$ with $P(Y)\in A[Y]$ and $Q(Y)\in A\<Y\>$. Hence 
$f(y)=P(y)+t^NQ(y)$ with $v\big(t^NQ(y)\big)\geqslant N\gamma$, so $c$ is a truncation of $P(y)$, and as
$P(y)\in F$ and $F$ is truncation closed, this gives $c\in F$. 

Next assume $\gamma\in  \supp y_n$. (For any $j\in \{1,\dots,n-1\}$ the case
$\gamma\in  \supp y_j$  is similar.) Then $y_n= y_{n0}+ z$ with $y_{n0},z\in F$ and $\supp y_{n0} < \gamma$, $v(z)=\gamma$. We have $f_0,\dots, f_{N-1}\in A\<Y\> ,\ g(Y,Z)\in A\<Y,Z\>$ such that
\begin{align*} f(Y_1,&\dots, Y_{n-1}, Y_n+Z)\ =\ \sum_{i<N} f_i(Y)Z^i + g(Y,Z) Z^N, \text{ so }\\
f(y)\ &=\ \sum_{i<N} f_i(y_1,\dots, y_{n-1}, y_{n0})z^i + \varepsilon, \quad v(\varepsilon)\geqslant N\gamma.
\end{align*}
Thus $c$ is a truncation of $d:= \sum_{i<N} f_i(y_1,\dots, y_{n-1}, y_{n0})z^i $. Since $o(y_{n0})< o(y_n)$, the minimality of $\big(o(y_1),\dots, o(y_n)\big)$ gives $f_i(y_1,\dots, y_{n-1}, y_{n0})\in F$ for all $i<N$, so
$d\in F$. Since $F$ is truncation closed, this gives $c\in F$.
\end{proof}

\begin{corollary}\label{tr2} Let $R$ be an $A$-closed subring of $\k[[t^{\Gamma^{\geqslant}}]]$ that is also truncation closed, and let $B\subseteq \k[[t^{\Gamma^{\geqslant}}]]$ be such that all proper truncations of all $b\in B$ lie in $R$.
Then the $A$-closure $R\<B\>$ of $R\cup B$ in $\k[[t^{\Gamma^{\geqslant}}]]$ is truncation closed. 
\end{corollary}
\begin{proof} By~\cite[Corollary 2.6]{tr} the subring $R[B]$ is truncation closed, so Lemma~\ref{tr1} applied to $E:= R[B]$ gives the desired result. 
\end{proof}

\subsection*{$A$-embedding with truncation} In addition to (1)--(4) above we now assume:
\begin{enumerate}
\item[(5)] the ring $A_0$ is noetherian;
\item[(6)] $\operatorname{char} \k=0$;
\item[(7)] there is no $\gamma\in \Gamma$ with $0 < \gamma< 1$.
\end{enumerate}
By (5)  and \cite[Theorem 3.3]{Mat}) the ring $A=A_0[[t]]$ is noetherian. The
assumptions on $A$ and $K$ made in Section~\ref{immAE} are thus satisfied, with $t:= t^1$. 
In Section~\ref{vrna} we considered $A$-extensions of 
the ``base'' structure $K$, but below $K$ plays the opposite role of an ambient structure. 

\medskip\noindent
In the next lemma and corollary $E\supseteq \k$ is a truncation closed valued subfield of $K$  whose valuation ring $R_E$ is $A$-closed in $\k[[t^{\Gamma^{\geqslant}}]]$.  Thus $\k[t]\subseteq E$, and $E$ is an $\L^A_{\preccurlyeq}$-substructure of $K$.
Note also that $t^{\Delta}\subseteq E$ with $\Delta:= v(E^\times)\subseteq \Gamma$. 

Let $(a_{\rho})$ be a divergent pc-sequence in $E$ with all $a_{\rho}\in R_E$. Since $\operatorname{char}(\k)=0$, $E$ is algebraically maximal by \cite[4.22]{Lou} and Lemma \ref{he5}. 
Hence $(a_{\rho})$ is of transcendental type over $E$. Since $K$ is spherically complete, $a_{\rho}\leadsto a$ 
for some $a \in \k[[t^{\Gamma^{\geqslant}}]]$; we choose such $a$ so that $o(a)$ is minimal. With $E, K$ in the role of $K,L$ earlier
in this section we obtain the valued subfield $E_a\supseteq E(a)$ of $K$ whose valuation ring 
$R_{E_a}=(R_E)_a$ is uniquely determined by $R_E$ and $a$ within $\k[[t^{\Gamma^{\geqslant}}]]$, as described in Proposition~\ref{imm1}(ii). Recall also that $E_a$ is an immediate extension of $E$. 

\begin{lemma} \label{tr3}  $E_a$ is truncation closed. 
\end{lemma} 
\begin{proof} We first show that all proper truncations of $a$ lie in $E$. We have $a=\sum_{\lambda} c_{\lambda}t^{\gamma_{\lambda}}$ where $\lambda$ ranges over all ordinals $<o(a)$, all $c_{\lambda}\in \k^\times$ and
$(\gamma_{\lambda})$ is a strictly increasing enumeration of $\supp a$. Consider a
proper truncation $a|_{\gamma}$ of $a$, with $\gamma\in \supp a$. Then $\gamma= v(a-a|_{\gamma}) < v(a-a_{\rho})$ for some $\rho$, so
$a|_{\gamma}$ is a truncation of such $a_{\rho}$, and therefore $a|_{\gamma}\in R_E$, as claimed. 

It follows that $\supp a$ has no largest element: if $\gamma=\gamma_{\mu}$ were the largest element, then
$a|_\gamma\in E$ and $a-a|_{\gamma}=c_{\mu}t^\gamma\in E_a$, so
$\gamma\in v(E_a^\times)=v(E^\times)$, hence $c_{\mu}t^\gamma\in E$, contradicting $a\notin E$. 
This yields a  divergent pc-sequence $(a|_{\gamma_{\lambda}})_{\lambda}$ in $E$ with pseudolimit $a$, and we now use it instead of $(a_{\rho})$ to describe $E_a$ (which after all does not depend on the particular approximating pc-sequence).  We have 
$$a\ =\ a|_{\gamma_{\lambda} }+ t^{\gamma_{\lambda}}u_{\lambda},\quad u_{\lambda}\ :=\  \sum_{\mu\geqslant \lambda} c_{\mu}t^{\gamma_{\mu}-\gamma_{\lambda}}.$$
All proper truncations of all $u_{\lambda}$ lie clearly in $R_E$, so
all $R_E\<u_{\lambda}\>$ are truncation closed by Corollary~\ref{tr2}, hence so is 
$R_{E_a}=\bigcup_\lambda  R_E\<u_{\lambda}\>$, and thus $E_a$ as well. 
 \end{proof}

\begin{corollary}\label{tr4}
Let $F$ be an immediate $A$-extension of $E$. Then there exists an 
$\L^{A}_{\preccurlyeq}$-embedding $F\to K$ over $E$ with truncation closed image. 
\end{corollary}
\begin{proof}
 Let $R_F$ be the valuation ring of $F$ and $f\in R_F\setminus R_E$. Take a divergent pc-sequence $(a_{\rho})$ in $E$ with all $a_{\rho}\in R_E$ such that $a_{\rho}\leadsto f$.
Then $(a_{\rho})$ is of transcendental type over $E$. Thus with $E, F, f$ in the role of  $K, L,a$ earlier in this section we obtain the valued subfield $E_f \supseteq E(f)$ of $F$ whose valuation ring $R_{E_f}$ is 
$A$-closed in $R_F$. Using Zorn it suffices to show that then there is an
$\L^A_{\preccurlyeq}$-embedding $E_f\to K$ over $E$ with truncation closed image. 

Lemma~\ref{tr3} and the remarks preceding it provide a pseudolimit $a\in \k[[t^{\Gamma^{\geqslant}}]]$  of $(a_{\rho})$ and a 
truncation closed valued subfield $E_a\supseteq E(a)$ of $K$ whose valuation ring $R_{E_a}$ is
$A$-closed in $\k[[t^{\Gamma^{\geqslant}}]]$. 
By Corollary~\ref{Aimmu} this gives a unique isomorphism $R_{E_f }\to R_{E_a}$ of $A$-rings that is the identity on $R_E$ sending $f$ to $a$. It extends to an $\L^{A}_{\preccurlyeq}$-embedding $F\to K$ over $E$ with truncation closed image $E_a$. 
\end{proof}

\section{Quantifier-free 1-types}\label{division}

\noindent
Functions given by one-variable terms in the language $\L_{\preccurlyeq,D}^{A,K}$ are  piecewise
given by  analytic functions on annuli. A precise statement of this is Proposition~\ref{prtau}, which is essential for all that follows.
The requisite notions of ``separated $A$-analytic structure'' and ``annulus'' come from \cite{CLR, CL} from which we also borrow results.

\medskip\noindent 
We keep the assumptions from Section~\ref{vrna} on $A,\ R,\ t, \ K$, so $R$ is a viable $A$-valuation ring, $t\in \iota_0(A),\ \smallo(R)=tR$. For now we  fix an algebraically closed $A$-extension $K^{\a}$ of $K$ with $A$-valuation ring $R^{\a}$.
Thus $R,\ K,\ R^{\a},\ K^{\a}$ are $\L^{A}_{\preccurlyeq, D}$-structures as specified earlier. 
Let $K_{\alg}$ be the algebraic closure of $K$ in $K^{\a}$. With Corollary~\ref{algext} we make $K_{\alg}$ an $A$-extension of $K$ by taking as its $A$-valuation ring the integral closure $R_{\alg}$ of $R$ in $K_{\alg}$; note that $R_{\alg}= R^{\a}\cap K_{\alg}$. 

Recall the language $\L_{\preccurlyeq,D}^{A,K}$ introduced in Section~\ref{vrna} and the sublanguage $\L_{\preccurlyeq}$ (for valued fields) of $\L^{A}_{\preccurlyeq, D}$; augmenting $\L_{\preccurlyeq}$ with names for the elements of $K$ gives the sublanguage  
$\L_{\preccurlyeq}^K$ of
$\L^{A,K}_{\preccurlyeq, D}$. The $\L_{\preccurlyeq}$-theory of algebraically closed valued fields with nontrivial valuation has quantifier elimination \cite[Theorem 3.29]{Lou}, so
\leqnomode
\begin{equation}\label{QE} \tag{$\star$}
K_{\text{alg}} \text{ is an elementary  $\L_{\preccurlyeq}$-substructure of } K^{\a}. 
\end{equation}

\subsection*{Separated $A$-analytic structures} 
A {\em separated $A$-analytic structure on $R$} is a  family $(\iota_{m,n})$ of ring morphisms
$$\iota_{m,n}\ :\  A\<X_1,\dots,X_m\>[[Y_1,\dots, Y_n]] \to\ \text{ring of $R$-valued functions on $R^m\times\smallo(R)^n$}$$
indexed by the pairs $(m,n) \in \N\times \N$, such that:
\begin{enumerate}
\item[(S1)] for $(x_1,\dots,x_m,y_1,\dots, y_n)\in R^m\times\smallo(R)^n$, $$\iota_{m,n}(X_k)(x_1,\dots,x_m,y_1,\dots, y_n)=x_k \qquad \text{ for } \qquad k=1,\dots,m,$$ 
$$\iota_{m,n}(Y_l)(x_1,\dots,x_m,y_1,\dots, y_n)=y_l \qquad \text{ for } \qquad l=1,\dots,n;$$
\item[(S2)] for $f\in A\<X_1,\dots, X_m\>[[Y_1,\dots,Y_n]]\subseteq A\<X_1,\dots, X_m, X_{m+1}\>[[Y_1,\dots,Y_n]]$ and $(x_1,\dots,x_m,x_{m+1},y_1,\dots, y_n)\in R^{m+1}\times\smallo(R)^n$ we have $$\iota_{m,n}(f)(x_1,\dots,x_m,y_1,\dots, y_n)=\iota_{m+1,n}(f)(x_1,\dots,x_m,x_{m+1},y_1,\dots, y_n),$$
and similarly with the $Y$-variables;
\item[(S3)] for $m\geqslant 1$, $f\in A\<X_1,\dots, X_m\>[[Y_1,\dots,Y_n]]$, $$g:=f(X_{m+1},\dots, X_{2m},Y_1,\dots,Y_n)\in A\<X_1,\dots, X_{2m}\>[[Y_1,\dots,Y_n]],$$ 
and $(x_1,\dots, x_{2m}, y_1,\dots,y_n)\in R^{2m}\times \smallo(R)^n$ we have: $$\iota_{m,n}(f)(x_{m+1},\dots, x_{2m},y_1,\dots,y_n)=\iota_{2m,n}(g)(x_1,\dots, x_{2m}, y_1,\dots,y_n),$$
 and similarly with the $Y$-variables.
\end{enumerate}

\medskip
\noindent
{\bf Remark.} Since $R$ is viable, a separated $A$-analytic structure on $R$ gives a {\em separated analytic $A$-structure on $K$} in the sense of \cite[Definition 2.7]{CLR}. Compared to \cite{CLR}, we include an extra  axiom (S3), parallel to (A3) from Section~\ref{ras}. We believe (S3) is needed for a proof of \cite[Proposition 2.8]{CLR}, just as (A3) in proving Lemma~\ref{A3}.

\subsection*{Annuli and the corresponding rings of analytic functions} 
The $A$-analytic structure of $R$ induces a separated $A$-analytic structure on $R$ as follows.  For $f(X,Y)\in A\<X\>[[Y]]$ we have $$\tilde{f}(X,Y)\ :=\ f(X,tY)\in A\<X,Y\>,$$ and we associate to
the series $f$ the function $f: R^m\times \smallo(R)^n\to R$ given by 
$$f(x,ty)\ :=\ \tilde{f}(x,y)\quad (x\in R^m, y\in R^n).$$
It is straightforward to check that the axioms (S1), (S2), (S3) are satisfied. Hence by the remark above, \cite{CLR, CL} applies to our setting; see \cite[4.4(1)]{CL}.

In the rest of this section we borrow terminology and results  from \cite{CLR, CL}. 
At the start of \cite[Section 5]{CL} the authors impose a condition that in our setting would correspond to $\mathrm{ker}(\iota_0)=\{0\}$. Nevertheless, we can use their work by replacing $A$ with 
$A/\mathrm{ker}(\iota_o)$ in view of Corollary~\ref{Iar} for $I=\mathrm{ker}(\iota_0)$. A more important difference is that
\cite{CL} takes $K^{\a}:=K_{\alg}$ as the ambient structure, whereas for later model-theoretic use we allow $K^{\a}$ to be any algebraically closed $A$-extension of $K$.  Fortunately, the results we need from \cite{CL} about $K_{\alg}$ will  readily transfer
to our $K^{\a}$: it helps that $K^{\a}$ as a valued field is an elementary extension of $K_{\alg}$.  In the rest of this section we fix an indeterminate $Z$, also to be used as a syntactic variable in $\L_{\preccurlyeq,D}^{A,K}$-formulas. 

Adopting  \cite[Definition 5.1.1]{CL}, an {\em $R$-annulus in $K^{\a}$}, or just {\em $R$-annulus} if $K^{\a}$ is clear from the context,  is a set
$F\subseteq K^{\a}$ given by
monic polynomials $p_0, \dots, p_n$ in $R[Z]$, irreducible in $K[Z]$, $l_0, \dots, l_n\in \N^{\geqslant 1},$ and 
$\pi_0, \dots,\pi_n\in R\setminus \{0\}$, as follows:
$$F\ =\ \{z\in K^{\a}:\  p_0^{l_0}(z) \preccurlyeq \pi_0,\  p_1^{l_1}(z) \succcurlyeq \pi_1\,\dots, p_n^{l_n}(z)  \succcurlyeq \pi_n\}$$
where the ``holes" $\{z\in K^{\a}:\   p_i^{l_i}(z)\prec \pi_i\} $, $1\leqslant i\leqslant n$,  are pairwise disjoint and contained in  $\{z\in K^{\a}:\ p_0^{l_0}(z) \preccurlyeq \pi_0\}$.  Such $F$ is said to be {\em given by\/} $(p_0^{l_0},\dots, p_n^{l_n};\pi_0,\dots,\pi_n)$. Thus $R^{\a}$ is an $R$-annulus given by $(Z;1)$ (with $n=0$). With respect to the valuation topology on $K^{\a}$, every $R$-annulus is a nonempty open-and-closed subset of $R^{\a}$, and so infinite without any isolated point. 
Note also that every $R$-annulus is defined in $K^{\a}$ by a quantifier-free $\L_{\preccurlyeq}^K$-formula $\phi(Z)$.

\medskip
\noindent
{\bf Remark.} In \cite[Definition 5.1.1]{CL}, with  $K^{\a}=K_{\alg}$, the above notion of $R$-annulus is less general than that of $K$-annulus; there our $R$-annuli, for $K^{\a}=K_{\alg}$,  would be among {\it closed} $K$-annuli. 

\medskip\noindent
In what follows we deviate from the convention that $Y=(Y_1,\dots, Y_n)$, and instead use $Y=(Y_0,\dots, Y_n)$ with an extra
indeterminate $Y_0$.  Let the $R$-annulus $F$ be given by $(p_0^{l_0},\dots, p_n^{l_n};\pi_0,\dots,\pi_n)$ as above, and consider the ideal $$I(F)\ :=\ \big(p_0^{l_0}(Z)-\pi_0Y_0,\ p_1^{l_1}(Z)Y_1-\pi_1,\dots,p_n^{l_n}(Z)Y_{n}-\pi_n\big)$$ of $K\<Z,Y\>.$ 
Define $\psi: F \to (R^{\a})^{2+n}$ by
$$\psi(z)\ :=\ \Big(z,\ \frac{p_0^{l_0}(z)}{\pi_0},\ \frac{\pi_1}{p_1^{l_1}(z)},\ \dots,\ \frac{\pi_n}{p_n^{l_n}(z)}\Big).$$
It is routine to verify that $$\psi(F)\ =\ \operatorname{Z}\big(I(F)\big)\ :=\ \{(z,y)\in (R^{\a})^{2+n}:\  g(z,y)=0 \text{ for all }g\in I(F)\}.$$
Let $\mathcal{R}(F)$ be the $K$-algebra of $K^{\a}$-valued functions on $F$. Then
$$\psi^*\ :\  K\<Z,Y\>\to \mathcal{R}(F), \qquad \psi^*(g)(z)\ :=\ g\big(\psi(z)\big) \text{ for }g\in K\<Z,Y\>,\ z\in F,$$
is a $K$-algebra morphism. We set $\O(F):=\psi^*\big(K\<Z,Y\>\big)$, a $K$-subalgebra of $\mathcal{R}(F)$. We refer to $\O(F)$ as {\em the ring of analytic functions on $F$}. Clearly, $I(F)\subseteq \ker(\psi^*)$, and if $K^{\a}=K_{\alg}$, then $I(F) = \ker(\psi^*)$
by  \cite[Corollary 5.6.6]{CL}, so $$K\<Z,Y\>/I(F)\ \cong\ \O(F)\ \text{ as $K$-algebras}.$$ 
(For $K^{\a}=K_{\alg}$ the $K$-algebras $K\<Z,Y\>/ I(F)$ and  $\O(F)$ are denoted by $\O^{\dag}_K(F)$ and $\O^{\sigma}_K(F)$ in \cite[Definition 5.1.4]{CL}.) 

Of course, $I(F)$ depends on how $F$ is given. However, $\O(F)$ is independent of the choice of the tuple
$(p_0^{l_0},\dots, p_L^{l_L};\pi_0,\dots,\pi_L)$ that gives $F$, by \cite[5.3.3]{CL}. Strictly speaking, we don't need this rather subtle fact, since any $R$-annulus $F$ below is assumed to come with a tuple that gives $F$, with $\O(F)$ defined accordingly. 

\medskip\noindent
{\bf Example.} For the $R$-annulus $F=R^{\a}$ given by $(Z;1)$ we have $\O(F)=K\<Z\>$ where $g\in K\<Z\>$ is 
identified with the function $z\mapsto g(z): R^{\a}\to K^{\a}$.



\subsection*{Univariate functions given by terms involving restricted division.} 
For our AKE-theory for valuation $A$-rings we need to understand what data about $z\in K^{\a}$ determine the isomorphism type  of the $A$-extension $K_z$ over $K$. Section~\ref{immAE} basically settles this issue for the case when $K_z$ is an immediate $A$-extension. For the general case we exploit below results from \cite{CL}.  Let $F$ be an $R$-annulus given by $(p_0^{l_0},\dots, p_n^{l_n};\pi_0,\dots,\pi_n)$ with corresponding map $\psi: F \to (R^{\a})^{2+n}$. 
We can  represent any function $f\in \O(F)$  by an $\L_{\preccurlyeq,D}^{A,K}$-term $\tau(Z)$: Let $f=\psi^*(g)$ with $g\in K\<Z,Y\>$. Take $c\in K^\times$, $x\in R^m$, and $G\in A\<X,Z,Y\>$, $X=(X_1,\dots, X_m)$, such that $g=g(Z,Y)=c\cdot G(x, Z,Y)$, and thus $g(z,y)=c\cdot G(x,z,y)$ for all $(z,y)\in (R^{\a})^{2+n}$. Then for all $z\in F$, $$  f(z)\ =\ g\big(\psi(z)\big)\ =\ c\cdot G\Big(x,z,D\big(p^{l_0}_0(z),\pi_0\big),D\big(\pi_1, p^{l_1}_1(z)\big),\dots,D\big(\pi_L, p_n^{l_n}(z)\big)\Big). $$

\begin{lemma}\label{div2}  For each polynomial $p(Z)\in K[Z]$ the function $z\mapsto p(z): F \to K^{\a}$ belongs to $\O(F)$. If
$f\in \O(F)$ and $f(z)\ne 0$ for all $z\in F$, then $f$ is a unit in the ring $\O(F)$. If $r(Z)\in K(Z)$ has no pole in $F$, then
the function $$z\mapsto r(z): F \to K^{\a}$$ belongs to $\O(F)$.
\end{lemma}
\begin{proof} The first claim follows from $K[Z]\subseteq K\<Z\>$. For the second claim, suppose $f\in \O(F)$ and $f(z)\ne 0$ for all
$z\in F$. 
 Take $g\in K\<Z,Y\>$ such that $f=\psi^*(g)$. Then $g$ has no zero on $\psi(F)=\operatorname{Z}(I(F))$. Hence
$1\in gK\<Z,Y\>+I(F)$ by Theorem~\ref{wdp4}(iii). Take $h\in K\<Z,Y\>$ with $1\in gh+I(F)$. Then $f\psi^*(h)=1$ in $\O(F)$.  The claim about $r(Z)\in K(Z)$ follows from the first claim and the second claim. 
\end{proof} 

\noindent
Next we transfer facts from \cite{CL}, where $K_{\alg}$ is the  ambient $A$-extension, to our $K^{\a}$. 
 For $\L_{\preccurlyeq,D}^{K}$-definable $P\subseteq K^{\a}$, let $P_{\alg}\subseteq K_{\text{alg}}$ be the
corresponding definable subset of $K_{\alg}$: any $\L_{\preccurlyeq,D}^{K}$-formula defining $P$ in  $K^{\a}$ defines $P_{\alg}$ in $K_{\text{alg}}$ (and $P\cap K_{\alg}=P_{\alg}$). To apply this to annuli, note that 
 $F_{\alg}$ is the $R$-annulus in $K_{\alg}$ given by the same tuple  $(p_0^{l_0},\dots, p_n^{l_n};\pi_0,\dots,\pi_n)$, with $I(F)=I(F_{\alg})$, and the corresponding map $\psi_{\alg}: F_{\alg}\to R_{\alg}^{2+n}$ is the restriction of $\psi$ to $F_{\alg}$.
This also yields the $K$-algebra $\O(F_{\alg})$ of analytic functions on $F_{\alg}$ 
(in the sense of $K_{\text{alg}}$ as the ambient $A$-extension of $K$). 

\begin{lemma}\label{isolem} For $f\in \O(F)$, let $f_{\alg}: F_{\alg}\to K_{\alg}$ be defined by
$f_{\alg}(z)= f(z)$ for $z\in F_{\alg}$. This yields an isomorphism of $K$-algebras: 
 $$f\mapsto  f_{\alg}\ :\  \O(F)\to \O(F_{\alg}).$$ 
\end{lemma}
\begin{proof} Let $f\in \O(F)$ and take $g\in K\<Z,Y\>$ such that $f=\psi^*(g)$. Then for $z\in F_{\alg}$ we have
$f_{\alg}(z)=f(z)=g(\psi(z))=g(\psi_{\alg}(z))$, so $f_{\alg}=\psi_{\alg}^*(g)\in \O(F_{\alg})$. This also shows surjectivity of
$f\mapsto  f_{\alg}\ :\  \O(F)\to \O(F_{\alg})$. With $f=\psi^*(g)$ as above, if $f_{\alg}=0$, then 
$\psi_{\alg}^*(g)=0$, so $g\in\ker(\psi_{\alg}^*)= I(F)\subseteq \ker(\psi)$, hence $f=0$. 
\end{proof} 

\noindent
The proof of this lemma yields $I(F)=\ker(\psi^*)$ (without assuming $K^{\a}=K_{\alg}$).

\begin{lemma}\label{powerb} For $g\in K\<Z,Y\>$ and $f=\psi^*(g)\in \O(F)$ the following are equivalent: \begin{enumerate}
\item[\rm(i)] $f(z)\preccurlyeq 1$ for all $z\in F$;
\item[\rm(ii)] $f(z)\preccurlyeq 1$ for all $z\in F_{\alg}$;
\item[\rm(iii)]  for some $d\in \N^{\geqslant 1}$ and $h_1, \dots, h_d \in R\<Z,Y\>$ we have
$$ g^d + h_1g^{d-1} + h_2g^{d-2} +\dots + h_d\ \in\  I(F).$$
\end{enumerate}
\end{lemma}
\begin{proof} The implications (i)$\Rightarrow$(ii) and  (iii)$\Rightarrow$(i) are clear. As to (ii)$\Rightarrow$(iii), this holds if $K^{\a}=K_{\alg}$ by Proposition~\ref{wdp4}(ii) and \cite[Proposition 5.2.12(b)]{CL}. Hence it holds for our $K^{\a}$ by Lemma~\ref{isolem}.  
\end{proof}

\noindent
By a {\em very strong unit on $F$\/} we mean a function $f\in \O(F)$ such that $f(z)\sim 1$ for all $z\in F$; cf. \cite[Definition 5.1.4]{CL}. 

\begin{lemma}\label{aff3+}  Let $f\in \O(F)$. Then there is a very strong unit $u$ on $F$ and a rational function $r\in K(Z)$ without pole in $F$ such that $f(z)= u(z)r(z)$ for all $z\in F$. In particular, if $f\ne 0$, then $f$ has only finitely many zeros in $F$.   
\end{lemma}
\begin{proof} If $K^{\a}=K_{\alg}$, this is a consequence of the Mittag-Leffler Decomposition from \cite[Theorem 5.5.2]{CL}. 
By Lemma~\ref{isolem} this yields $u\in \O(F)$ and $r\in K(Z)$ without pole in $F$ such that $u_{\alg}$ is a very strong unit
on $F_{\alg}$ and $f(z)=u(z)r(z)$ for all $z\in F_{\alg}$, and thus for all $z\in F$ by Lemma~\ref{div2}. It remains to show that $u$ is a very strong unit on $F$. Applying the above to $u-1$ instead of $f$ we have $u^*\in \O(F)$ and $r^*\in K(Z)$
without pole in $F$
such that $u^*_{\alg}$ is a very strong unit on $F_{\alg}$ and $u_{\alg}(z)-1=u^*_{\alg}(z)r^*(z)$ for all $z\in F_{\alg}$. Then $r^*(z)\prec 1$ for all $z\in F_{\alg}$, hence for all $z\in F$. Lemma~\ref{isolem} gives $u(z)-1=u^*(z)r^*(z)$ for all $z\in F$, and
Lemma~\ref{powerb} yields $u^*(z)\preccurlyeq 1$ for all $z\in F$, hence $u(z)\sim 1$ for all $z\in F$.
\end{proof} 

\noindent
The proof above also yields another useful fact:

\begin{corollary}\label{vsu} For $u\in \O(F)$ we have the following equivalence:
$$u \text{ is a very strong unit  on }F\ \Longleftrightarrow\ u_{\alg} \text{ is a very strong unit on }F_{\alg}.$$
\end{corollary}


\begin{lemma}\label{closure+}
Let $f_1,\dots, f_m\in \O(F)$ be such that $f_1(z),\dots, f_m(z)\preccurlyeq 1$ for all $z\in F$. Then for $G\in R\<X\>$, $X=(X_1, \dots, X_m)$, the function 
$$z\mapsto G\big(f_1(z),\dots,f_m(z)\big)\ :\ F \to R^{\a}$$ belongs to $\O(F)$.
\end{lemma}
\begin{proof} For $i=1,\dots,m$, take $g_i\in K\<Z,Y\>$  such that $f_i(z)=g_i\big(\psi(z)\big)$ for all $z\in F$.
Lemma~\ref{powerb} gives for $i=1,\dots,m$ a polynomial $P_i=P_i(Z,Y, X_i)\in R\<Z,Y\>[X_i]$ over $R\<Z,Y\>$, monic and of degree $d_i\geqslant 1$ in $X_i$, such that $P_i\big(Z,Y, g_i(Z,Y)\big)\in I(F)$, and thus $P_i\big(z,y, g_i(z,y)\big)=0$ for all $(z,y)\in \psi(F)$. 
Since elements of $R\<Z,Y,X\>$ are 
 specializations of restricted power series over $A$, Lemma~\ref{d2} gives in $R\<Z,Y,X\>$ an equality
$$G(X)\ =\ \sum_{i=1}^mq_iP_i+ r, \quad \text{with }q_1,\dots, q_m\in R\<Z,Y,X\>,\  r\in R\<Z,Y\>[X].$$  
Let $z\in F$; so $\psi(z)=(z,y)\in \psi(F)$, $y\in (R^{\a})^{1+n}$.
Substituting in the equality above $(z,y)$ for $(Z,Y)$ and $f_i(z)$ for $X_i$ and using $P_i\big(z,y, g_i(z,y)\big)=0$ we obtain
$$G\big(f_1(z),\dots, f_m(z)\big)\ =\ r\big(z,y, f_1(z),\dots, f_m(z)\big). $$
Now $r=\sum_j r_j(Z,Y)X^j$ with $j=(j_1,\dots, j_m)$ ranging over a finite subset of  $\N^m$ and all $r_j\in R\<Z,Y\>$. So
$r\big(z,y, f_1(z),\dots, f_m(z)\big)=\sum_j \psi^*(r_j)(z)f_1(z)^{j_1}\cdots f_m(z)^{j_m}$, which exhibits this function of $z\in F$ as being in
$\O(F)$. 
\end{proof}


\begin{lemma}\label{FvdP+} For any rational function $r(Z)\in K(Z)$ the set  
 $\{z\in R^{\a}:\  r(z)\preccurlyeq 1\}$ is a finite union of $R$-annuli. $($Here $r(z)\preccurlyeq 1$ includes $z$ not being a pole of $r.)$
 
\end{lemma}
\begin{proof} If $K^{\a}=K_{\alg}$, this can be shown along the lines of the proof of \cite[Lemma 3.16]{CLR}. It then holds for our $K^{\a}$ by (\ref{QE}). 
\end{proof}

\begin{lemma} \label{F'} 
If $F'$ is an $R$-annulus and $F'\cap F\neq \emptyset$, then $F'\cap F$ is an $R$-annulus. 
\end{lemma}
\begin{proof}  If $K^{\a}=K_{\alg}$,  this holds by  ~\cite[Lemma 5.1.2 (iv)]{CL}.  It then holds for our $K^{\a}$ by
 (\ref{QE}).  \end{proof}

\begin{lemma}\label{restrf1} Let $F_1$ be an $R$-annulus, $F_1\subseteq F$, and $f\in \O(F)$. Then $f|_{F_1}\in \O(F_1)$.
\end{lemma}
\begin{proof} Let $F_1$ be given by $(q_0^{e_0},\dots, q_{m}^{e_m}; \rho_0,\dots, \rho_m)$ with corresponding map
$$\psi_1\ :\ F_1\to (R^{\a})^{2+m},\quad \psi_1(z)\ :=\  \Big(z, \frac{q_0^{e_0}(z)}{\rho_0}, \frac{\rho_1}{q_1^{e_1}(z)},\dots, \frac{\rho_m}{q_m^{e_m}(z)}\Big).$$ This also yields the corresponding ideal $I(F_1)$ of $K\<Z, V\>$ where we use a tuple $V=(V_0,\dots, V_m)$ of new indeterminates $V_0,\dots, V_m$. Accordingly, $\psi_1$ yields the surjective  $K$-algebra morphism
$\psi_1^*: K\<Z, V\>\to \O(F_1)$ with kernel $I(F_1)$.  We now define $I$ to be the ideal
of $K\<Z, V, Y\>$ generated by $I(F_1)\subseteq K\<Z,V\>$ and $I(F)\subseteq K\<Z,Y\>$. This yields the $K$-algebra morphism
$$\iota\ :\  K\<Z, V\>/I(F_1)\to K\<Z, V, Y\>/I, \quad g+I(F_1)\mapsto g+I\qquad (g\in K\<Z,V\>).$$
The proof of [CL, Proposition 5.3.2] shows that $\iota$ is an isomorphism.
For $s$ in $K\<Z, V, Y\>$ we define $s^*: F_1\to (R^{\a})^{3+m+n}$ by
$$s^*(z)\ :=\  s\Big(z, \frac{q_0^{e_0}(z)}{\rho_0}, \frac{\rho_1}{q_1^{e_1}(z)},\dots, \frac{\rho_m}{q_m^{e_m}(z)},\ \frac{p_0^{l_0}(z)}{\pi_0},\ \frac{\pi_1}{p_1^{l_1}(z)},\ \dots,\ \frac{\pi_n}{p_n^{l_n}(z)}\Big),$$
so $s^*(z)$ is $s$ evaluated at a combination of $\psi_1(z)$ and $\psi(z)$. Using $I(F_1)\subseteq \ker \psi_1^*$ and $I(F)\subseteq \ker \psi^*$ we see
that for $s\in I$ we have $s^*(z)=0$ for all $z\in F_1$.

Take $g\in K\<Z,Y\>$ with $f=\psi^*(g)$. Surjectivity of
$\iota$ gives  $h\in K\<Z,V\>$  with $g-h\in I$. It follows that for $z\in F_1$ 
we have  $(g-h)^*(z) =g^*(z)- h^*(z)= 0$. For $z\in F_1$ we have
 $g^*(z)=\psi^*(g)(z)$ and $h^*(z)=\psi_1^*(h)(z)$, so
$f(z)=\psi^*(g)(z)=\psi_1^*(h)(z)$. Thus $f|_{F_1} =\psi_1^*(h)\in \O(F_1)$. 
\end{proof}

\noindent
The next result for $K^{\a}:=K_{\alg}$ is  close to \cite[Theorem 5.5.3]{CL}; see also \cite[A.1.10]{LC2}. We give a complete proof 
because we require $R$-annuli where [loc. cit.] allows more general annuli, and because the details are used to obtain Corollary~\ref{id}.

\begin{proposition}\label{prtau} Let $\tau(Z)$ be an 
$\L_{\preccurlyeq,D}^{A,K}$-term. Then there are  quantifier-free $\L^K_{\preccurlyeq}$-formulas
$\phi_1(Z),\dots, \phi_n(Z)$, $R$-annuli $F_1,\dots, F_n$, and $f_1\in \O(F_1),\dots, f_n\in \O(F_n)$, such that:  \begin{enumerate}
\item[(i)] $R^{\a}= \phi_1(R^{\a})\cup\dots \cup \phi_n(R^{\a})$;
\item[(ii)] $\phi_j(R^{\a})\subseteq F_j$ and $\tau(z)=f_j(z)$ for all $z\in \phi_j(R^{\a})$,  for $j=1,\dots, n$.
\end{enumerate}
\end{proposition}
\begin{proof}  By induction on the complexity of $\tau=\tau(Z)$. For $\tau$ the name of an element of $K$ or just the variable $Z$ one can take $n=1$, and make the obvious choices of $\phi_1, F_1, f_1$. 
Next, given $ \phi_1,\dots, \phi_n, F_1,\dots, F_n, f_1,\dots, f_n$ as in the proposition, we only need to replace each $f_j$ by $-f_j$ to make it work for $-\tau$ instead of $\tau$.

Suppose $\tau=\tau_1+\tau_2$. The inductive assumption gives quantifier-free
$\L^K_{\preccurlyeq}$-formulas $\phi_{11}(Z),\dots, \phi_{1n_1}(Z)$ and $\phi_{21}(Z),\dots, \phi_{2n_2}(Z)$, $R$-annuli $$F_{11},\dots, F_{1n_1}, F_{21},\dots, F_{2n_2},$$ and
$f_{ij}\in \O(F_{ij})$ for $i=1,2$ and $j=1,\dots, n_i$ such that
\begin{itemize}
\item $R^{\a}= \phi_{11}(R^{\a})\cup\cdots\cup\phi_{1n_1}(R^{\a})= \phi_{21}(R^{\a})\cup\cdots\cup\phi_{2n_2}(R^{\a})$;
\item  $\phi_{ij}(R^{\a})\subseteq F_{ij}$ and $\tau_i(z)=f_{ij}(z)$ for all $z\in \phi_{ij}(R^{\a})$.
\end{itemize}
Let $1\leqslant j_1\leqslant n_1$ and $1\leqslant j_2\leqslant n_2$ and set $\phi_{j_1j_2}:= \phi_{1j_1}\wedge \phi_{2j_2}$ and 
$F_{j_1j_2}:= F_{1j_1}\cap F_{2j_2}$. Then $\phi_{j_1j_2}(R^{\a})\subseteq F_{j_1j_2}$ and 
$\tau(z)=f_{1j_1}(z) + f_{2j_2}(z)$ for
$z\in \phi_{j_1j_2}(R^{\a})$. 
Thus listing the nonempty $F_{j_1j_2}$ as $F_1,\dots, F_n$, the corresponding
$f_{1j_1}|_{F_{j_1j_2}}+f_{2j_2}|_{F_{j_1j_2}}$ as $f_1,\dots, f_n$, and the corresponding
$\phi_{j_1j_2}$ as $\phi_1,\dots, \phi_n$ yields (i) and (ii).  (This uses Lemmas~\ref{F'} and ~\ref{restrf1}.) The case
$\tau=\tau_1\cdot\tau_2$ is handled in the same way.  

Next, suppose $\tau=D(\tau_1, \tau_2)$, and let
the $\phi_{ij}, F_{ij}, f_{ij}$ be as before and also define
$\phi_{j_1j_2}$ and $F_{j_1j_2}$ as before. Consider one such
pair $j=(j_1,j_2)$ with  $F_{j_1j_2}\ne \emptyset$ and set $\phi_j=\phi_{j_1j_2}$ and $F_j=F_{j_1j_2}$. 
If $f_{1j_1}=0$ or $f_{2j_2}=0$, then $D(\tau_1, \tau_2)(z)=0$ for all $z\in F_j$, a trivial case. Assume $f_{1j_1}\ne 0$ and
$f_{2j_2}\ne 0$. Then Lemma~\ref{aff3+} yields very strong units
$u_1, u_2$ on $F_j$ and rational functions $r_1, r_2\in K(Z)^\times$ without pole in $F_j$ such that
$f_{1j_1}(z)=u_1(z)r_1(z)$ and $f_{2j_2}(z)=u_2(z)r_2(z)$ for all $z\in F_j$.  Set $r=r_1/r_2\in K(Z)^\times$. If $z\in F_j$ and $r_2(z)\ne 0$, this gives $f_{2j_2}(z)\ne 0$ and $f_{1j_1}(z)/f_{2j_2}\asymp r(z)$. 
Hence by Lemma~\ref{FvdP+} we have $N\in \N$ such that
$$\{z\in F_j:\  f_{1j_1}(z)\preccurlyeq f_{2j_2}(z)\ne 0\}\ =\  \big(F^{j,1}\cup \cdots \cup F^{j,N}\big)\setminus E, $$
with  $R$-annuli $F^{j,1},\dots, F^{j,N}\subseteq F_j$, and finite $E=\{z\in F_j: r_2(z)=0\}$. Let $1\leqslant \nu\leqslant N$. Then $r$ has no pole in $F^{j,\nu}$: if $z\in F^{j,\nu}$ were a pole, then there would be $z'\in F^{j,\nu}\setminus E$ arbitrarily close to $z$ with $r(z')\succ 1$, a contradiction.  Thus by setting $f^{j,\nu}(z):= \frac{u_1(z)}{u_2(z)}r(z)$ for $z\in F^{j,\nu}$ we obtain $f^{j,\nu}\in \O(F^{j,{\nu}})$ with $D\big(f_{1j_1}(z),f_{2j_2}(z)\big)=f^{j,\nu}(z)$ for all $z\in F^{j,\nu}\setminus E$.  
Take a quantifier-free 
$\L^K_{\preccurlyeq}$-formula $\phi^{j,\nu}(Z)$ 
such that $\phi^{j,\nu}(R^{\a})=F^{j,\nu}\setminus E$. Then for $z\in (\phi_j\wedge \phi^{j,\nu})(R^{\a})$ we have
$\tau(z) = f^{j,\nu}(z)$. Lemma~\ref{aff3+} also gives a quantifier-free $\L^K_{\preccurlyeq}$-formula $\theta_{j}(Z)$
such that for all $z\in R^{\a}$,  
$$ R^{\a}\models  \theta_{j}(z)\  \Longleftrightarrow\  z\in F_j \text{ and }\big( f_{1j_1}(z)\succ  f_{2j_2}(z) \text{ or  }f_{2j_2}(z)=0\big).$$
Thus $\theta_j(R^{\a})\subseteq F_j$, and  for $z\in (\phi_j\wedge \theta_j)(R^{\a})$ we have $\tau(z)=0$.

Finally, suppose $\tau=G(\tau_1,\dots, \tau_m)$, where $G\in A\<X\>$.  The inductive assumption gives for $i=1,\dots,m$
quantifier-free
$\L^K_{\preccurlyeq}$-formulas $\phi_{i1}(Z),\dots, \phi_{in_i}(Z)$, $R$-annuli $F_{i1},\dots, F_{in_i}$, and  functions
$f_{i1}\in \O(F_{i1}),\dots, f_{in_i}\in \O(F_{in_i})$,
such that 
\begin{itemize}
\item $R^{\a}= \phi_{i1}(R^{\a})\cup\cdots\cup\phi_{in_i}(R^{\a})$;
\item  $\phi_{ij}(R^{\a})\subseteq F_{ij}$ and $\tau_i(z)=f_{ij}(z)$, for all $z\in \phi_{ij}(R^{\a})$ and $j=1,\dots,n_i$.
\end{itemize}
Let $j=(j_1,\dots, j_m)$ with $1\leqslant j_1\leqslant n_1,\dots, 1\leqslant j_m\leqslant n_m$ and set
$$\phi_j\ :=\ \phi_{1j_1}\wedge\cdots \wedge \phi_{mj_m},\qquad F_j\ :=\  F_{1j_1}\cap \cdots\cap F_{mj_m}.$$
Using Lemmas~ \ref{aff3+},  \ref{FvdP+}, and \ref{F'} we get $N\in \N$ such that
$$\{z\in F_j:\ f_{1j_1}(z)\preccurlyeq 1,\dots, f_{mj_m}(z)\preccurlyeq 1\}\ =\  F^{j,1} \cup \dots \cup F^{j,N}$$
where $F^{j,1},\dots, F^{j,N}$ are $R$-annuli.
Let $1\leqslant \nu\leqslant N$.  Take a quantifier-free 
$\L^K_{\preccurlyeq}$-formula  $\phi^{j,\nu}(Z)$ such that $\phi^{j,\nu}(R^{\a})=F^{j,\nu}$.
For $i=1,\dots,m$, set $f^{j,\nu}_{i}:= f_{ij_i}|_{F^{j,\nu}}$, a function in $\O(F^{j,\nu})$ with $|f^{j,\nu}_i(z)|\leqslant 1$ for all $z\in F^{j,\nu}$, so by Lemma~\ref{closure+},
$$f^{j,\nu}\ :=\ G\big(f^{j,\nu}_{1},\dots, f^{j,\nu}_{m}\big)\in \O(F^{j,\nu}).$$
Then $\tau(z)=f^{j,\nu}(z)$ for all $z\in (\phi_j\wedge\phi^{j,\nu})(R^{\a})$. It remains to note that $\tau(z)=0$
for all $z\in (\phi_j\wedge \neg \phi^{j,1}\wedge \neg \phi^{j,2}\wedge\cdots \wedge \neg \phi^{j,N})(R^{\a})$.  
\end{proof}

\begin{corollary}\label{corcard} Let $z\in K^{\a}$. Then $K_z$ is an immediate extension of $K(z)$:
$$ \res K_z\ =\ \res K(z)\ \subseteq\ \res K^{\a},\qquad v(K_z^\times)\ =\ v\big( K(z)^\times\big)\subseteq\ v\big( (K^{\a})^\times\big).$$
As a consequence, $\Gamma=v(K^\times)$ and $v(K_z^\times)$ have the same cardinality, and if $\res K$ is infinite, then $\res K$ and $\res K_z$ have the same cardinality. 
\end{corollary} 
\begin{proof}  Replacing $z$ by $z^{-1}$ if $z\succ 1$, we arrange $z\in R^{\a}$. Consider a nonzero element $\tau(z)$ of $K_z$, where $\tau(Z)$ is an $\L_{\preccurlyeq, D}^{A,K}$-term.
Let $\phi_1,\dots, \phi_n, F_1,\dots, F_n, f_1,\dots, f_n$ be as in Proposition~\ref{prtau}. Take $j\in \{1,\dots,n\}$ with $z\in\phi_j(R^{\a})$. Then Lemma~\ref{aff3+} applied to $F_j, f_j$ in the role of $F, f$ yields  $r(Z)\in  K(Z)$
without pole in $F$ such that $\tau(z)\sim r(z)$. Thus $K_z$ is an immediate extension of $K(z)$. The rest now follows from \cite[Corollary 5.19]{Lou}.  
\end{proof} 




\subsection*{Uniformity with respect to $K^{\a}$}  So far we we kept $K^{\a}$ fixed, 
but  in the rest of this section $K^{\a}$ ranges over arbitrary algebraically closed  $A$-extensions of $K$.
Let $\tau(Z)$ be an
$\L_{\preccurlyeq,D}^{A, K}$-term. To enable model-theoretic arguments we need to show that 
  $$\phi_1,\dots, \phi_n, F_1,\dots, F_n, f_1,\dots, f_n$$ in Proposition~\ref{prtau} is ``independent'' of $K^{\a}$.
To make sense of this we consider tuples $\big(\phi_j, \Phi_j\big)_{j=1}^n$ of
quantifier-free $\L^{K}_{\preccurlyeq}$-formulas $\phi_j(Z)$ and $\Phi_j(Z)$, $j=1,\dots,n$. 
Call such a tuple  $\big(\phi_j, \Phi_j\big)_{j=1}^n$ {\em good for $\tau$ in $K^{\a}$} if the following hold: 
\begin{enumerate}
\item[(1)] $F_j:=\Phi_j(K^{\a})$ is an $R$-annulus in $K^{\a}$ for $j=1,\dots,n$;
\item[(2)] $z\mapsto \tau(z): F_j\to K^{\a}$ is a function $f_j\in\O(F_j)$ for $j=1,\dots,n$;
\item[(3)] $\phi_1,\dots, \phi_n, F_1,\dots, F_n, f_1,\dots, f_n$ satisfy (i) and (ii) in Proposition~\ref{prtau}.
 \end{enumerate} 
 Of course, $\O(F_j)$ in (2) is meant in the sense of $K^{\a}$. 
 
 \begin{corollary}\label{id}
Some tuple $\big(\phi_j, \Phi_j\big)_{j=1}^n$ is good for $\tau$ in all $K^{\a}$. 
\end{corollary}
\begin{proof} All $K^{\a}$ that are algebraic over $K$ are isomorphic as $A$-extensions of $K$, so any tuple
that is good for one is good for all.  Now, let any $K^{\a}$ be given, and let $K_{\alg}$ be the algebraic closure of
$K$ in $K^{\a}$, with $K_{\alg}$ as $\L_{\preccurlyeq, D}^{A,K}$-substructure of $K^{\a}$. Following the steps and recursive constructions
in the proof of Proposition~\ref{prtau} yields a tuple that is good for $\tau$ in $K_{\alg}$, as well as in $K^{\a}$: 
to see this,
use (\ref{QE}), Lemmas~\ref{isolem}, ~\ref{powerb}, and Corollary~\ref{vsu} to pass from $K_{\alg}$ to $K^{\a}$.   
\end{proof}

\subsection*{The quantifier-free type of an element over $K$}  For a valued field $E$ and an element $z$ in a valued field extension $L$ of $E$, the {\em quantifier-free $
\L_{\preccurlyeq}$-type of $z$ over $E$\/} is the set $\text{qftp}_{\preccurlyeq}(z|E)$ of all quantifier-free
$\L^E_{\preccurlyeq}$-formulas $\theta(Z)$ such that $L\models \theta(z)$. Likewise, for an element $z$ in an $A$-extension $L$ of $K$, the {\em quantifier-free $\L^A_{\preccurlyeq,D}$-type of $z$ over $K$\/} is the set $\text{qftp}^A_{\preccurlyeq,D}(z|K)$ of all quantifier-free $\L^{A,K}_{\preccurlyeq,D}$-formulas $\theta(Z)$ such that $L\models \theta(z)$. 


\begin{proposition}\label{id1}
Let $K_1$ and $K_2$ be A-extensions of  K, and suppose
$z_1\in K_1$ and $z_2\in K_2$ are such that $\qftp_{\preccurlyeq}(z_1|K)=\qftp_{\preccurlyeq}(z_2|K)$.
Then $$\qftp^A_{\preccurlyeq,D}(z_1|K)\ =\  \qftp^A_{\preccurlyeq,D}(z_2|K).$$
\end{proposition}

\begin{proof}
Our assumption gives an $\L_{\preccurlyeq}$-isomorphism  $i:K(z_1)\to K(z_2)$ over $K$ that sends $z_1$ to $z_2$. If $z_1$ is algebraic over $K$, then so is $z_2$ and $K(z_1)$ and $K(z_2)$ underly the
$\L_{\preccurlyeq,D}^{A}$-substructures $K_{z_1}$  and $K_{z_2}$ of $K_1$ and $K_2$, respectively, by Corollary~\ref{algtype}, so  $i$ is an  
$\L_{\preccurlyeq,D}^{A}$-isomorphism over $K$ by Corollary~\ref{Aint5}. Hence $z_1$ and $z_2$ have the same quantifier-free $\L_{\preccurlyeq,D}^{A}$-type over $K$.

 For the rest of the proof, assume that $z_1$ and $z_2$ are both transcendental over $K$. Replacing $z_1, z_2$ by their reciprocals if necessary, we arrange $z_1, z_2\preccurlyeq 1$. We claim that for every $\L_{\preccurlyeq,D}^{A,K}$-term $\tau(Z)$, $$\tau(z_1)=0\ \Longleftrightarrow\ \tau(z_2)=0.$$ For $c,d$ in any $A$-extension of $K$, $c\preccurlyeq d$ if and only if $c=0$ or $D(c,d)\neq0$. Hence, in light of Corollary~\ref{type}, our claim yields an $\L_{\preccurlyeq,D}^{A}$-isomorphism 
$K_{z_1} \to K_{z_2}$ over $K$ given by $\tau(z_1)\mapsto \tau(z_2)$, where $\tau(Z)$ ranges over $\L_{\preccurlyeq,D}^{A,K}$-terms. Thus a proof of the claim will complete the proof of the proposition.

By passing to algebraic closures we arrange that $K_1$ and $K_2$ are algebraically closed, with respective $A$-valuation rings
$R_1$ and $R_2$.
 Let $\tau(Z)$ be an $\L_{\preccurlyeq,D}^{A,K}$-term such that $\tau(z_1)=0$. Take a tuple $$ \big(\phi_j, \Phi_j\big)_{j=1}^n$$ as in Corollary~\ref{id}. This gives $j\in \{1,\dots, n\}$ with $z_1\in \phi_j(R_1)$ and
 $z_2\in \phi_j(R_2)$. Set $F_{j1}:=\Phi_j(K_1)\subseteq R_1$, and let $f_{j1}\in \O(F_{j1})$ be given by
 $f_{j1}(z)=\tau_1(z)$, and define $F_{j2}\subseteq R_2$ and $f_{j2}\in \O(F_{j2})$ in the same way.  
 
 Then $\tau(z_1)= f_{j1}(z_1)=0$, so $f_{j1}=0$ by Corollary~\ref{aff3+}  and $z_1$ being transcendental over $K$. By descending to
 the algebraic closures of $K$ in $K_1$ and $K_2$ and using that these algebraic closures are isomorphic $A$-extensions of $K$ we obtain $f_{j2}=0$ by Lemma~\ref{isolem}. 
 Hence $\tau(z_2)= f_{j2}(z_2)=0$. This proves the forward direction of our claim. The backward direction follows in the same way. 
\end{proof}

\section{Analytic AKE-type Equivalence and Induced Structure}\label{A-AKE}

\noindent
We begin with some terminology and conventions. 
A valued field will be construed as an $\L_{\preccurlyeq}$-structure in the usual way.

 Let $K$ be a valued field. We denote its valuation ring by $R$ (by $R_F$ if we are dealing with a valued field $F$ instead). Let $\smallo(R)$ be the maximal ideal of $R$ and $\k:=R/\smallo(R)$ the residue field of $K$. We also let $v:K^\times \to \Gamma$ with $\Gamma=v(K^\times)$ be a valuation on the field $K$ such that $R=\{z\in K: v(z)\geqslant 0\}$ (and if we are dealing instead with a valued field $F$, we have likewise the residue field $\k_F$ and a valuation $v_F: F^\times\to \Gamma_F$). 

A {\em coefficient field of $K$\/} is a lift of $\k$, that is, a subfield $C$ of $K$ such that $C\subseteq R$ and
$C$ maps bijectively onto $\k$ under the residue map $R\to \k$, equivalently, a subfield $C$ of $K$ such that
$R=C+\smallo(R)$. Likewise, a {\em monomial group\/} of $K$ is a lift of $\Gamma$, that is, a subgroup $G$ of $K^\times$ that is mapped bijectively onto $\Gamma$ by $v: K^\times \to \Gamma$. 
If $K$ is henselian (by which we mean that the local ring $R$ is henselian)  and $\k$ has characteristic $0$, then $K$ has a coefficient field; see for example \cite[Lemma 2.9]{Lou}. If $K$ is algebraically closed or $\aleph_1$-saturated, then $K$ has a monomial group; see for example \cite[Lemmas 3.3.32, 3.3.39]{ADH}.

\bigskip\noindent
{\em Throughout $A$ is as in Section~\ref{vrna}: $A$ is noetherian with an ideal $\smallo(A)\ne A$, such that $\bigcap_e \smallo(A)^e=\{0\}$ and $A$ is $\smallo(A)$-adically complete.} 

By an {\em $\Acg$-field} we mean an expansion
$\F=(F, C,G)$ of an $A$-field $F$ where $C$ is (the underlying set of) a coefficient field of $F$ and $G$ is (the underlying set of) a monomial group $G$ of $F$. Let $\L^{\Acg}_{\preccurlyeq,D}$ be the language $\L^{A}_{\preccurlyeq,D}$ augmented by unary predicate symbols $C$ and $G$. We construe an $\Acg$-field as an $\L^{\Acg}_{\preccurlyeq,D}$-structure in the obvious way. 

\medskip\noindent
{\em Example to keep in mind}:  $\F=\big(F, C, t^{\Z}\big)$, where $C$ is any field, $F$ is the Laurent series field $C(\!(t)\!)$ with valuation ring $C[[t]]$, and $A=C[[t]]$, $\smallo(A)=tC[[t]]$, with the natural $A$-analytic structure on $C[[t]]$. To simplify
notation we denote this $\Acg$-field $\F$ by $\big(C(\!(t)\!), C, t^{\Z}\big)$.

\medskip\noindent
{\em In the rest of this section $\K=(K, C_{\K}, G_{\K})$ is an $\Acg$-field such that the valuation $A$-ring $R$ of $K$ is viable; $\k$ is the residue field of $K$ and $\Gamma:= v(K^\times)$ its value group.}


\subsection*{Good substructures and good maps} Our aim is to establish an analogue of the Equivalence Theorem \cite[5.21]{Lou} in our {\em analytic} setting {\em with coefficient field and monomial group}, and we follow the general setup and proof strategy there.

\medskip\noindent
A {\em good substructure} of $\K$ is an $\L^{\Acg}_{\preccurlyeq,D}$-substructure $\E=(E,C_{\E}, G_{\E})$ of $\K$ which is also an $\Acg$-field.
Note that then $E$ is an $\L^{A}_{\preccurlyeq,D}$-substructure of $K$, and
$$C_{\E}\ =\ C_{\K}\cap E, \qquad G_{\E}\ =\ G_{\K}\cap E.$$  Below, $\E=(E ,C_{\E}, G_{\E})$ is a good substructure of $\K$.  By Lemma~\ref{subinher} the valuation $A$-ring $R_E$ is viable. For $a\in K$, set $\E_a:=\big(E_a, C_{\K}\cap E_a, G_{\K}\cap E_a\big)\subseteq \K$. 

\begin{lemma}\label{good}  We consider four cases for an element $a\in C_{\K}\cup G_{\K}$:  \begin{enumerate} 
\item[(i)]  {\em $a\in C_{\K}$ is algebraic over $E$}.  Then $a$ is algebraic over $C_{\E}$, $E[a]$ is the underlying field of $E_a$, $C_{\K}\cap E_a=C_{\E}[a]$, and $G_{\K}\cap E_a=G_{\E}$;
\item[(ii)]  {\em $a\in C_{\K}$ is transcendental over $E$}. Then 
$C_{\K}\cap E_a=C_{\E}(a)$, $G_{\K}\cap E_a=G_{\E}$;
\item[(iii)] {\em $a\in G_{\K}$ is algebraic over $E$}.  Then $a^d\in G_{\E}$ for some $d\geqslant 1$, $E[a]$ is the underlying field of $E_a$, $C_{\K}\cap E_a=C_{\E}$, and $G_{\K}\cap E_a=G_{\E}\cdot a^{\Z}$;
\item[(iv)] {\em $a\in G_{\K}$ is transcendental over $E$}. Then  
$C_{\K}\cap E_a=C_{\E}$, $G_{\K}\cap E_a=G_{\E}\cdot a^{\Z}$. 
\end{enumerate}
In each of these four cases, $\E_a$ is a good substructure of $\K$. 
\end{lemma} 
\begin{proof} If $a\in C_{\K}\cup G_{\K}$ is algebraic over $E$, this follows from Corollary~\ref{algtype}.
In the transcendental case, use Corollary~\ref{corcard}. 
\end{proof}

\noindent
In this subsection, $\K'=(K', C_{\K'},G_{\K'} )$ is an $\Acg$-field like $\K$: its valuation $A$-ring $R'$ is viable. We also let 
$\E'=(E', C_{\E'}, G_{\E'})$  be a good substructure of $\K'$, and for $b\in K'$  we set $\E'_{b}:=\big(E'_{b}, C_{\K'}\cap E'_{b}, G_{\K'}\cap E'_{b}\big)$.

\bigskip\noindent
Let $\L_{\mathrm{r}}:=\{0,1,+,-,\cdot\}$ be the language of rings and $\L_{\mathrm{v}}:=\{1,\cdot,\preccurlyeq\}$ the language of (multiplicative) ordered abelian groups, taken as sublanguages of $\L^{A}_{\preccurlyeq,D}$; we construe $C_{\K}, C_{\K'}$  as $\L_{\mathrm{r}}$-structures and $G_{\K}, G_{\K'}$ as $\L_{\mathrm{v}}$-structures accordingly. 

A  \textit{good map} $f: \E\to \E'$ is an $\L^{\Acg}_{\preccurlyeq,D}$-isomorphism $\E\to \E'$ such that: \begin{enumerate}
\item [(r)] the $\L_{\mathrm{r}}$-isomorphism $f|_{C_{\E}}: C_{\E}\to C_{\E'}$ is a partial elementary map from the field $C_{\K}$ to the field $C_{\K'}$; 
\item [(v)] the $\L_{\mathrm{v}}$-isomorphism $f|_{G_{\E}}: G_{\E}\to G_{\E'}$ is a partial elementary map from the ordered group $G_{\K}$ to the ordered group $G_{\K'}$.
\end{enumerate} 


\begin{theorem}\label{AKE} Suppose $\operatorname{char} \k=0$. Let
$f:\E\to \E'$ be a good map. Then $f$ is a partial elementary map from $\K$ to $\K'$.
\end{theorem}

\noindent
We need $\operatorname{char} \k=0$ only towards the end of the proof below to guarantee that a certain pc-sequence
$(a_{\rho})$ introduced there is of transcendental type.

\begin{proof} By passing to suitable elementary extensions we arrange that the underlying valued fields of $\K$ and $\K'$ are $\kappa$-saturated, where $\kappa$ is an uncountable cardinal greater than the cardinalities of $C_{\E}$ and  $G_{\E}$.
A good substructure $$\E_1\ =\ (E_1,C_{\E_1},G_{\E_1})$$ of $\K$ is termed {\em small} if $\kappa$ is greater than the cardinalities of $C_{\E_1}$ and $G_{\E_1}$.  We shall prove that for any $a\in \K$ we can extend $f$ to a good map with small domain $\mathcal{F}\supseteq \E$  such that $a\in \F$. By the properties of ``back-and-forth" this  suffices. In addition to Corollary \ref{Aimmu}, we will need the extension procedures in (1)--(4) below.  

 In (1) and (2) we assume $a\in C_{\K}$ and extend $\E$ and $\E'$ to small good substructures $\F$ of $\K$
 and $\F'$ of $\K'$ and our good map $f$ to a good map $g: \F\to \F'$ such that $a\in C_{\F}$ and $G_{\E}=G_{\F}$.  In (3) and (4) we assume that $a\in G_{\K}$, and extend $\E$ and $\E'$ to small good substructures $\F$ of $\K$
 and $\F'$ of $\K'$ and our good map $f$ to a good map $g: \F\to \F'$ such that $a\in G_{\F}$ and $C_{\E}=C_{\F}$. 

\medskip
\noindent
(1) {\em The case that $a\in C_{\K}$ is algebraic over $E$}.  Then $\kappa$-saturation of $\K'$  gives $b\in C_{\K'}$ and an $\L_{\mathrm{r}}$-isomorphism $g_{\mathrm{r}}: C_{\E}[a]\to C_{\E'}[b]$ extending $f|_{C_{\E}}$ and sending $a$ to $b$ such that $g_{\mathrm{r}}$ is a partial elementary map from $C_{\K}$ to $C_{\K'}$. 

Now \cite[Lemma 3.21]{Lou} gives an $\L_{\preccurlyeq}$-isomorphism $g:E[a]\to E'[b]$ extending both $f$ and $g_{\mathrm{r}}$. Then
$g$ is an $\L_{\preccurlyeq,D}^{A}$-isomorphism $E_a \to E'_{b}$ by Corollary~\ref{algtype} and Proposition~\ref{id1}. By Lemma~\ref{good} (i),  $\E_a$ and $\E'_{b}$ are good substructures of $\K$ and $\K'$ respectively, and $g$ is a good map.

\medskip\noindent
(2){ \em The case that $a\in C_{\K}$ is transcendental over $E$}.  As in (1) we have $b\in C_{\K'}$ and an $\L_{\mathrm{r}}$-isomorphism 
$g_{\mathrm{r}}: C_{\E}(a)\to C_{\E'}(b)$ extending $f|_{C_{\E}}$ and sending $a$ to $b$ such that $g_{\mathrm{r}}$ is a partial elementary map from $C_{\K}$ to $C_{\K'}$.

Then \cite[Lemma 3.22]{Lou} gives an $\L_{\preccurlyeq}$-isomorphism $E(a)\to E'(b)$ extending both $f$ and $g_{\mathrm{r}}$.  This $\L_{\preccurlyeq}$-isomorphism extends to an $\L_{\preccurlyeq,D}^{A}$-isomorphism $g:\ E_a \to E'_{b}$ by Proposition~\ref{id1}. Lemma~\ref{good}(ii) gives that $\E_a$ and $\E'_{b}$ are good substructures of $\K$ and $\K'$ respectively, and that $g$ is a good map.

\medskip\noindent
(3) {\em The case that $a\in G_{\K}\setminus G_{\E}$ and $a^p \in G_{\E}$, where $p$ is a prime number}.  As before we get $b\in G_{\K'}$ and an $\L_{\mathrm{v}}$-isomorphism $g_{\mathrm{v}}: G_{\E}\cdot a^{\Z}\to G_{\E'}\cdot {b}^{\Z}$ extending $f|_{G_{\E}}$ and sending $a$ to $b$ such that $g_{\mathrm{v}}$ is a partial elementary map from $G_{\K}$ to $G_{\K'}$.
Now \cite[Lemma 5.6]{Lou} gives an $L_{\preccurlyeq}$-isomorphism $g:E(a)\to E'(b)$ extending both $f$ and $g_{\mathrm{v}}$. 
Then $g$ is an $\L_{\preccurlyeq,D}^{A}$-isomorphism $E_a \to E'_{b}$ by Corollary~\ref{algtype} and Proposition~\ref{id1}. By Lemma~\ref{good}(iii),  $\E_a$ and $\E'_{b}$ are good substructures of $\K$ and $\K'$ respectively, and $g$ is a good map.

\medskip\noindent
(4) {\em The case that $a\in G_{\K}$ and $a^d \notin G_{\E}$ for all $d\geqslant 1$}. 
As before we get $b\in G_{\K'}$ and an $\L_{\mathrm{v}}$-isomorphism 
$$g_{\mathrm{v}}: G_{\E}\cdot a^{\Z}\to G_{\E'}\cdot {b}^{\Z}$$ extending $f|_{G_{\E}}$ and sending $a$ to $b$ such that $g_{\mathrm{v}}$ is a partial elementary map from $G_{\K}$ to $G_{\K'}$. Note that $a$ is transcendental over $E$  by \cite[Proposition 3.19]{Lou}; likewise, $b$ is transcendental over $E'$.  

Now \cite[Lemma 3.23]{Lou} gives an $\L_{\preccurlyeq}$-isomorphism $E(a)\to E'(b)$ extending both $f$ and $g_{\mathrm{v}}$. 
This $\L_{\preccurlyeq}$-isomorphism extends to an $\L_{\preccurlyeq,D}^{A}$-isomorphism $g:\ E_a \to E'_{b}$ by Proposition~\ref{id1}. By Lemma~\ref{good}(iv), $\E_a$ and $\E'_{b}$ are good substructures of $\K$ and $\K'$ respectively, and $g$ is a good map.


\medskip
\noindent
Let now any $a\in K$ be given. Let $C_1$ be the subfield of $C_{\K}$ such that $\res C_1=\res E_a$, and let
$G_1$ be the subgroup of $G_{\K}$ such that $v(G_1)=v(E_a^{\times})$. We do not guarantee that  $C_1\subseteq E_a$ or 
$G_1\subseteq E_a^\times$, but $C_{\E}$ and $C_1$ have the same cardinality, and so do
$G_{\E}$ and $G_1$, by Corollary~\ref{corcard}.  Thus by iterating (1)--(4), we extend $\E$ and $\E'$ to small good substructures 
$\E_1=(E_1,C_1,G_1)$ of $\K$ and $\E'_1=(E'_1, C'_1, G'_1)$ of $\K'$, and extend $f$ to a good map
$f_1:\E_1\to \E'_1$. Next, let $C_2$ be the subfield of $C_{\K}$ such that $\res C_2=\res E_{1,a}$, and let
$G_2$ be the subgroup of $G_{\K}$ such that $v(G_2)=v(E_{1,a}^\times)$, and obtain likewise $\E_2=(E_2, C_2, G_2)$ with $\E_1\subseteq \E_2\subseteq \K$, and an extension of $f_1$ to a good map
$f_2: \E_2\to \E'_2$, with  $\E'_1\subseteq \E'_2\subseteq \K'$. Continuing this way we obtain for each $n$  small good substructures 
$$\E_n\ =\ (E_n, C_n, G_n)\  \subseteq\  \E_{n+1}\  =\ (E_{n+1}, C_{n+1}, G_{n+1})$$ of $\K$ such that
 $\res C_{n+1}= \res E_{n,a}$ and $v(G_{n+1})=v(E_{n,a}^\times)$,
and small good substructures $\E'_n\subseteq \E'_{n+1}$ of $\K'$, and good maps $$f_n\ :\  \E_n\to \E'_n, \qquad f_{n+1}\ :\ \E_{n+1}\to \E'_{n+1}$$ such that $f_{n+1}$ extends $f_n$; here $\E_0:= \E$, $\E'_0:= \E'$ and $f_0:= f$.
Then $$\E_{\infty}\ :=\  \bigcup_n \E_n\ =\ (E_{\infty},C_{\infty},G_{\infty})$$ is a small good substructure of $\K$, and
$\E'_{\infty}:= \bigcup_n \E'_n=(E'_{\infty},C'_{\infty},G'_{\infty})$ is a small good substructure of $\K'$, and we have a good map
$f_{\infty}:\E_{\infty} \to \E'_{\infty}$ extending each $f_n$.  Using $E_{\infty,a}=\bigcup_n E_{n,a}$ we see that
 $E_{\infty,a}$ is an immediate extension of $E_{\infty}$. 
 
 If $a\in E_{\infty}$ we have achieved our goal of extending $f$ to a good map with small domain containing $a$, so assume $a\notin E_{\infty}$. Replacing $a$ by $a^{-1}$ if necessary we arrange $a\preccurlyeq 1$. Take a divergent pc-sequence
$(a_{\rho})$ in $E_{\infty}$ such that all $a_{\rho}\preccurlyeq 1$ and $a_{\rho}\leadsto a$. Then $(b_{\rho}):=\big(f_{\infty}(a_{\rho})\big)$ is a divergent pc-sequence in $E'_{\infty}$. 
Since the underlying valued field of $\K'$ is $\kappa$-saturated and the cardinality of the value group of $E'_{\infty}$ is less than $\kappa$ we have $b\in K'$ such that $b_{\rho}\leadsto b$.  Note that $(a_{\rho})$ is of transcendental type over $E_{\infty}$, by \cite[4.22, 4.16]{Lou}. Hence
$(b_{\rho})$ is of transcendental type over $E'_{\infty}$, and so $E'_{\infty,b}$ is an immediate extension of $E'_{\infty}$ by Proposition~\ref{imm1}.   This yields the (small) good substructures
$\E_{\infty,a}:=\big(E_{\infty,a}, C_{\infty}, G_{\infty}\big)$ of $\K$ and $\E'_{\infty,b}:=\big(E'_{\infty,b}, C'_{\infty}, G'_{\infty}\big)$ of $\K'$. Moreover,  $f_{\infty}$ extends by Corollary~\ref{Aimmu} to a good map
$\E_{\infty,a}\to \E'_{\infty,b}$, and we have achieved our goal. 
\end{proof}

\begin{corollary} Suppose $\operatorname{char} \k =0$, $C_{\E}\preccurlyeq C$ as $\L_{\mathrm{r}}$-structures, and $G_{\E} \preccurlyeq G$ as $\L_{\mathrm{v}}$-structures. Then $\E\preccurlyeq\K$.
\end{corollary}
\begin{proof} Note that
$\E$ is a good substructure of both $\K$ and $\K':=\E$, and the identity on $\E$ is a good map. Now apply Theorem \ref{AKE}. 
\end{proof}

\subsection*{Induced structure on coefficient field and monomial group}  In this subsection we assume for our
$\Acg$-field $\K=(K, C_{\K}, G_{\K})$  that $\operatorname{char} \k=0$. Our aim here is Corollary~\ref{indstr} on the structure that $\K$ induces on $C_{\K}$ and $G_{\K}$ combined.  It will be derived in a familiar way from Theorem~\ref{AKE} and a fact implicit in its proof.
To state that fact we let $\E=(E, C_{\E}, G_{\E})$ and
 $\F=(F,C_{\F},G_{\F})$ be $\Acg$-fields and $\L^{\Acg}_{\preccurlyeq,D}$-extensions of $\K$. For $a\in E^n$, let $\tp(a|K)$ be the $\L^{\Acg}_{\preccurlyeq,D}$-type of $a$ over $K$, that is, the set of $\L^{\Acg, K}_{\preccurlyeq,D}$-formulas
 $\phi(Y_1,\dots, Y_n)$ such that $\E\models \phi(a)$.  Likewise, for $c\in C^n_\E$, let
$\tp(c|C_{\K})$ be the $\L_{\mathrm{r}}$-type of $c$ over $C_{\K}$, and for $g\in G_\E^{n}$, let $\tp(g|G_{\K})$ be the $\L_{\mathrm{v}}$-type of $g$ over $G_{\K}$.

\begin{lemma}\label{separation} Suppose $\E$ and $\F$ are elementary extensions of $\K$.  Let $c_{\E}\in C^m_\E$, $g_{\E}\in G^n_\F$ and $c_{\F}\in C^m_F$, $g_{\F} \in G^n_\F$ be such that $$\tp(c_{\E}|C_{\K})\ =\ \tp(c_{\F}|C_{\K}), \qquad \tp(g_{\E}|G_{\K})\ =\ \tp(g_{\F}|G_{\K}).$$
 Then for the points $(c_{\E},g_{\E})\in E^{m+n}$ and  $(c_{\F},g_{\F})\in F^{m+n}$ we have $$\tp\big((c_{\E},g_{\E})|K\big)\ =\  \tp\big((c_{\F},g_{\F})|K\big).$$
\end{lemma}
\begin{proof}
By our assumptions $\K$ is a good substructure of both $\E$ and $\F$, and the identity on $\K$ is a good map.
Using  $\tp(c_{\E}|C)= \tp(c_{\F}|C)$ and the extension procedures (1) and (2) in the proof of Theorem~\ref{AKE} in conjunction with Lemma~\ref{good}(i),(ii) we obtain a good map  whose  domain contains the elements of $K$ and the components of $c_{\E}$ and that is the identity on $\K$ and sends $c_{\E}$ to $c_{\F}$, such that the monomial group of its domain is still $G_{\E}$. Next we use likewise the extension procedures from (3) and (4) in that proof to extend this good map further so that its domain now contains the components of $g_{\E}$ as well, and sends sends $g_{\E}$ to $g_{\F}$. It remains to use Theorem~\ref{AKE}.
\end{proof}

\begin{corollary}\label{indstr}
Each subset of $C_{\K}^m\times G_{\K}^n\subseteq K^{m+n}$ which is definable in $\K$ is a finite union of ``rectangles'' $P\times Q$ with $P\subseteq C_{\K}^m$ definable in the $\L_{\mathrm{r}}$-structure $C_{\K}$ and $Q\subseteq G_{\K}^n$ definable in the $\L_{\mathrm{v}}$-structure $G_{\K}$.
\end{corollary}
\begin{proof}
Apply Lemma~\ref{separation} in conjunction with  \cite[Lemmas 5.13, 5.14]{Lou}. 
\end{proof}

\begin{corollary}\label{corind} If $P\subseteq K^n$ is definable in $\K$, then $P\cap C_{\K}^n$ is
definable in the $\L_{\mathrm{r}}$-structure $C_{\K}$, and $P\cap G_{\K}^n$ is definable in 
the $\L_{\mathrm{v}}$-structure $G_{\K}$.
\end{corollary}

\noindent
In particular, the sets $C_{\K}, G_{\K}\subseteq K$ are  stably embedded and orthogonal in $\K$. Next an application
of Corollary~\ref{corind}.

\subsection*{Recovering the Binyamini-Cluckers-Novikov result} We construe $\C(\!(t)\!)$ below as an
$A$-field in the usual way, with $A=\C[[t]], \smallo(A)=tA$. Proposition 2 in  
\cite{BCN} concerns the $3$-sorted structure $\mathcal{M}$ consisting of the following: 
$$ \text{ the $A$-field }\C(\!(t)\!),\qquad  \text{ the field }\C, \qquad \text{  the ordered abelian group }\Z,$$
(each a $1$-sorted structure) 
and two functions relating the three sorts: the obvious $t$-adic valuation $v: \C(\!(t)\!)^\times \to \Z$, and the ``reduced angular component map'' $\overline{ac}: \C(\!(t)\!)\to \C$ that assigns to each nonzero Laurent series $f=\sum_{k\in \Z} c_kt^k$ (all $c_k\in \C$) its leading coefficient $c_{v(f)}$, with $\overline{ac}(0):=0$ by convention. 

This $3$-sorted $\mathcal{M}$ should not be confused with the $1$-sorted
$\big(\C(\!(t)\!), \C, t^{\Z}\big)$ that is among the 
$\Acg$-fields $\K$ considered in Section~\ref{A-AKE}. We have a natural interpretation of
$\mathcal{M}$ in $\big(\C(\!(t)\!), \C, t^{\Z}\big)$, which shows that if a set $P\subseteq \C(\!(t)\!)^n$ is definable in $\mathcal{M}$, then it is definable in $\big(\C(\!(t)\!), \C, t^{\Z}\big)$. The converse fails: the subsets $\C$ and  $t^{\Z}$ of $\C(\!(t)\!)$ are definable in the latter, but not in the former by \cite[Theorem 3.9]{vdD}; thus $\big(\C(\!(t)\!), \C, t^{\Z}\big)$ is ``richer''  than $\mathcal{M}$. 

For $d\geqslant 1$ we let $\C[t]_{<d}$ be the set of polynomials in $\C[t]$ of degree $<d$. Then 
$\C[t]_{<d}$ is a subset of $\C[[t]]$, and thus of $\C(\!(t)\!)$. We identify  $\C[t]_{<d}$ with $\C^d$ via the bijection $c_0+ c_1t+\cdots + c_{d-1}t^{d-1} \mapsto (c_0,\dots, c_{d-1})$ for $c_0,\dots, c_{d-1}\in \C$. For $P\subseteq \C(\!(t)\!)^n$ we set 
$P(d):=P\cap (\C[t]_{<d})^n$, which under the identification above becomes a subset of $\C^{dn}$. Now Proposition 2 in
\cite{BCN} says:

\medskip\noindent
{\it if $P\subseteq \C(\!(t)\!)^n$ is definable in $\mathcal{M}$, then for each $d\geqslant 1$ the set $P(d)\subseteq \C^{dn}$ is a constructible subset of
the space $\C^{dn}$ with its Zariski topology}. 

\medskip\noindent
By ``Chevalley-Tarski'' a subset of $\C^m$ is constructible iff it is definable in the field $\C$, so this proposition is for $\K=\big(\C(\!(t)\!), \C, t^{\Z}\big)$ a special case of Corollary~\ref{corind}.

\subsection*{NIP} The model-theoretic condition NIP forbids certain combinatorial configurations;  there is a lot of information about it in  \cite{Simon}.  Below a structure $\cM$ is said to have NIP if its theory $\text{Th}(\cM)$ has NIP. By Delon \cite{Delon} and Gurevich \& Schmitt \cite{GS}, a henselian valued field of equicharacteristic $0$ has NIP  iff its residue field (as a ring) has NIP.  Jahnke \& Simon~\cite{JS}  extend this to a criterion that
also applies to {\em expansions\/} of such valued fields. We apply their criterion now to our analytic setting:

\begin{corollary} Assume $\operatorname{char} \k=0$. Then:
$$\text{ the $\L_{\preccurlyeq,D}^{\Acg}$-structure 
$\K$ has $\mathrm{NIP}$}\ \Longleftrightarrow\  \text{ the ring $\k$ has $\mathrm{NIP}$}.$$
\end{corollary}
\begin{proof} The direction $\Rightarrow$ is obvious, and for $\Leftarrow$, assume that $\k$ has NIP as a ring.
We refer to \cite[Section 2]{JS} for definitions and notations used in this proof. Using \cite[Theorem 2.3]{JS}, it suffices to show that $\mathrm{Th}(\K)$ satisfies the conditions (SE) and (Im). Condition (SE) requires that the residue field and the value group be stably embedded in $\K$. This condition is satisfied in our setting by Corollary~\ref{corind}. 

In order that $\mathrm{Th}(\K)$ satisfies condition (Im) it suffices to show the following: Let $\E=(E, \dots)$ be an elementary extension of
$\K$ and $a\in E$ such that $K(a)$ is an immediate valued field extension of $K$; then 
the $\L_{\preccurlyeq,D}^{\Acg}$-type of $a$ over $K$ (in $\E$) is implied by instances of NIP formulas in this type. By the Delon-Gurevich-Schmitt result the valued field reduct of $\E$  has NIP. So let $\F=(F,\dots)$ also be an elementary extension of $\K$ and let $b\in F$ have the same $\L_{\preccurlyeq}$-type over $K$ as $a$; it is enough to show that then $a$ and $b$ also have the same $\L_{\preccurlyeq,D}^{\Acg}$-type over $K$. Now $\K$ is a good substructure of both $\E$ and $\F$ and the identity on $\K$ is a good map. By Corollary~\ref{Aimmu}  this gives a good map $\K_a\to \K_b$ that is the identity on $\K$ and sends $a$ to $b$. Now apply Theorem~\ref{AKE}.
\end{proof}

\subsection*{Elementary Equivalence} The AKE-results are often summarized suggestively as follows: For any
henselian valued fields $E$ and $F$ of equicharacteristic $0$,
$$E \equiv F\ \Longleftrightarrow\ \res E \equiv \res F \text{ as fields,  and } \Gamma_E \equiv \Gamma_F \text{ as ordered groups}.$$ 
For special $A$ we now derive a similar result in our analytic setting. Let $A=C[[t]]$ where $C$ is a field and take $\smallo(A)=tA$.
Recall that we construe $A$ itself as an $A$-ring. In the beginning of this section we introduced the $\Acg$-field $\big(C(\!(t)\!), C, t^{\Z}\big)$. 
We try to embed it into our $\Acg$-field $\K=(K, C_{\K}, G_{\K})$ (with viable valuation $A$-ring $R$).  Lemma~\ref{initial} yields the $A$-ring morphism $\iota_0: A\to R$, which is injective: 
for $a\in A^{\ne}$ we have $a=ct^e(1+ b)$ with $c\in C^\times$, $e\in \N$, and $b\in \smallo(A)$, and so its image in $R$ is nonzero, since viability of $R$ gives $\iota_0(t)\ne 0$. Thus $\iota_0$ extends to an embedding $C(\!(t)\!)\to K$ of $A$-fields which we denote by $\iota_K$ to indicate its dependence on $K$. It is routine to verify the following:  

\begin{lemma}\label{obv} The map $\iota_{K}: C(\!(t)\!)\to K$ is an  embedding $\big(C(\!(t)\!), C, t^{\Z}\big)\to \K$  of $\Acg$-fields if and only if $\iota_{K}(C)\subseteq C_{\K}$ and $\iota_{K}(t)\in G_{\K}$. 
\end{lemma} 

\noindent
The conditions $\iota_K(C)\subseteq C_{\K}$ and $\iota_K(t)\in G_{\K}$ are  satisfied for $\K=\big(C(\!(t)\!), C, t^{\Z}\big)$. These conditions are of a first-order nature, since for any $a\in A$ the constant symbol $a$ of $\L^{A}$ names the element
$\iota_K(a)\in K$. 
If these conditions are satisfied, let $C_{\K,\star}$ be the expansion $\big(C_{\K}, (\iota_{K}(c))_{c\in C}\big)$
of the field $C_{\K}$, and let $G_{\K, \star}$ be the expansion $\big(G_{\K}, \iota_{K}(t)\big)$ of the ordered group $G_{\K}$.

\begin{corollary}\label{AAKEtr} Assume $\operatorname{char} C=0$. Suppose $\iota_K(C)\subseteq C_{\K}$ and $\iota_K(t)\in G_{\K}$, and likewise for $\K'$. Then we have the following equivalence: 
$$\K \equiv \K'\ \Longleftrightarrow\  C_{\K,\star} \equiv C_{\K',\star} \text{ and }\ G_{\K,\star}\equiv G_{\K',\star}.$$
\end{corollary} 
\begin{proof} The direction $\Rightarrow$ is clear. For $\Leftarrow$, assume the right hand side. Then by Lemma~\ref{obv} we have $\Acg$-field embeddings
$$\iota_{K}\ :\ \big(C(\!(t)\!), C, t^{\Z}\big)\to \K, \qquad \iota_{K'}\ :\ \big(C(\!(t)\!), C, t^{\Z}\big)\to \K'. $$ 
Identifying 
$\big(C(\!(t)\!), C, t^{\Z}\big)$ with its image in $\K$ and $\K'$ via these embeddings yields good substructures
of $\K$ and $\K'$, with the identity on $\big(C(\!(t)\!), C, t^{\Z}\big)$ as a good map. Now use Theorem~\ref{AKE}.
\end{proof}

\section{Separating Variables}\label{req}

\noindent
For $\Acg$-fields of equicharacteristic $0$ with viable valuation $A$-ring we establish here a uniform reduction of any formula to a boolean combination of formulas whose quantifiers range only over $C$ and formulas whose
quantifiers range only over  $G$.  Towards this goal we extend the language by symbols for an absolute value map and for  a 
coefficient map. We begin by introducing these maps. Throughout $k$ and $l$ range over $\N$ (as do $d$, $m$, $n$). 

\medskip\noindent
Let $K$ be just a valued field and $G$ a monomial group of $K$. Then we
define the map $|\cdot|_G: K \to K$  as follows:  $$|0|_G\ :=\ 0,\qquad     |a|_G\ :=\ g\  \text{ if } a\asymp g\in G.$$ 
This map is $0$-definable in the expansion $(K,G)$ of the valued field $K$, takes values in $G\cup\{0\}$, and is the identity on
$G\cup \{0\}$.  Moreover, $a\mapsto |a|_G: K^\times \to G$ is a group morphism, and for all $a,b\in K$: $|a|_G\preccurlyeq |b|_G\Leftrightarrow a\preccurlyeq b$.

Let in addition $C$ be a coefficient field of $K$. Then we introduce the coefficient map  $\co=\co_{C,G}: K \to K$ as follows: 
 $$\textrm{co}(0)\ :=\ 0,\qquad
 \textrm{co}(a)\ :=\ c\ \text{ if } a\ne 0 \text{ and }a/|a|_G\sim c\in C^\times.$$
This map is $0$-definable in the expansion $(K,C,G)$ of the valued field $K$, takes values in $C$, is
the identity on $C$, and $\textrm{co}(ab)=\textrm{co}(a)\textrm{co}(b)$ for all $a,b\in K$. 

\medskip\noindent
 {\em In the rest of this section $A$ is noetherian with an ideal $\smallo(A)\ne A$ such that $\bigcap_e \smallo(A)^e=\{0\}$ and $A$ is $\smallo(A)$-adically complete}. 
We extend the language $\L^{A}_{\preccurlyeq,D}$ to $\L^{A, +}_{\preccurlyeq,D}$  by adding unary function symbols $\text{co}$ and $|\cdot|$, and likewise
we extend $\L^{\Acg}_{\preccurlyeq,D}$ to $\L^{\Acg, +}_{\preccurlyeq,D}$.  

 Below $\K=(K, C, G)$ ranges over $\Acg$-fields whose valuation $A$-ring is viable. We expand $\K$ to an  $\L^{\Acg, +}_{\preccurlyeq,D}$-structure $\K^+$ by interpreting  $|\cdot|$  as the function $|\cdot|_G$ and $\text{co}$ as the
corresponding coefficient map $\text{co}_{C,G}$.
Note that the good substructures of $\K$ are exactly the $\L^{\Acg}_{\preccurlyeq,D}$-reducts of $\L^{\Acg, +}_{\preccurlyeq,D}$-substructures of $\K^+$ whose underlying ring is a field. Now
$$  \{\tau^{\K}:\  \tau \text{ is a variable-free }\L^{A,+}_{\preccurlyeq, D}\text{-term}\}$$
underlies an $\L_{\preccurlyeq, D}^A$-substructure of $K$, so by  Lemma~\ref{lem10.0}, $$R_0\ :=\ \{\tau^{\K}:\ \tau \text{ is a variable-free }\L^{A, +}_{\preccurlyeq,D}\text{-term and } \tau^{\K}\preccurlyeq 1\}$$ is an $A$-subring of $R$ and a valuation ring dominated by $R$. Hence $$K_0\ :=\ \mathrm{Frac}(R_0)$$ underlies the smallest $\L^{\Acg, +}_{\preccurlyeq,D}$-substructure of $\K^+$ whose underlying ring is a field. Thus $\K_0:=(K_0, C\cap K_0, G\cap K_0)$ is a good substructure of $\K$. 

More generally, let $a=(a_1,\dots, a_k)\in K^k$, and take a tuple $u=(u_1,\dots, u_k)$ of distinct variables $u_1,\dots, u_k$. Then
$ \{\tau^{\K}(a):\  \tau(u) \text{ is an }\L^{A,+}_{\preccurlyeq, D}\text{-term}\}$
underlies an $\L_{\preccurlyeq, D}^A$-substructure of $K$, so by Lemma~\ref{lem10.0}, $$R_{0|a}\ :=\ \{\tau^{\K}(a):\ \tau(u) \text{ is an }\L^{A, +}_{\preccurlyeq,D}\text{-term and } \tau^{\K}(a)\preccurlyeq 1\}$$ is an $A$-subring of $R$ and a valuation ring dominated by $R$. Hence $$K_{0|a}\ :=\ \mathrm{Frac}(R_{0|a})$$ underlies the smallest $\L^{\Acg, +}_{\preccurlyeq,D}$-substructure of $\K^+$ that contains $a_1,\dots, a_k$ and whose underlying ring is a field. Thus 
$\K_{0|a} := (K_{0|a}, C\cap K_{0|a}, G\cap K_{0|a})$ is a good substructure of $\K$.

\medskip
\noindent
Towards separating variables, let $\Lrc$ be the language $\Lr$ augmented by the unary relation symbol $C$, and let $\Lvg$ be the language $\Lv$ augmented by the unary relation symbol $G$, so $\Lrc$ and $\Lvg$ are sublanguages of 
$\L_{\preccurlyeq, D}^{\Acg}$.  We define the {\it $\mathrm{c}$-relative formulas\/} to be the $\Lrc$-formulas obtained by applying the following recursive rules:
\begin{itemize}
    \item quantifier-free $\Lr$-formulas are $\mathrm{c}$-relative formulas;  
    \item if $\phi$  and $\psi$ are $\mathrm{c}$-relative formulas, then so are  $\phi \wedge \psi$, $\phi \vee \psi$, $\neg \phi$;
    \item if $\phi$ is a $\mathrm{c}$-relative formula and $u$ is a variable, then $\exists u\big(C(u)\wedge  \phi\big)$ and  $\forall u\big(C(u)\to  \phi\big)$ are $\mathrm{c}$-relative formulas.
\end{itemize}
So ``$\mathrm{c}$-relative'' indicates that all quantifiers are relativized to $C$. 
Likewise, the \textit{$\mathrm{g}$-relative formulas}  are the $\Lvg$-formulas obtained by the following recursive rules: 
\begin{itemize}
    \item quantifier-free $\Lv$-formulas are $\mathrm{g}$-relative formulas;  
    \item if $\phi$  and $\psi$ are $\mathrm{g}$-relative formulas, then so are $\phi \wedge \psi$, $\phi \vee \psi$, $\neg \phi$;
    \item if $\phi$ is a $\mathrm{g}$-relative formula and $u$ is a variable, then $\exists u\big( G(u)\wedge \phi\big)$ and  $\forall u\big(G(u)\to \phi\big)$ are  $\mathrm{g}$-relative formulas.
\end{itemize}

\noindent
Let $x_1,\ldots,x_m$ be distinct variables and $x=(x_1,\dots, x_m)$.
A \textit{$\mathrm{c}$-formula} in $x$ is by definition an $\L^{\Acg, +}_{\preccurlyeq,D}$-formula
$$\psi(x)\ :=\  \psi' \big(\text{co}(\tau_1(x)), \ldots, \text{co}(\tau_k(x))\big)$$
with $\psi'(u_1,\ldots,u_k)$ a $\mathrm{c}$-relative formula  and  $\L^{A, +}_{\preccurlyeq,D}$-terms $\tau_1(x),\ldots, \tau_k(x)$. Likewise, a \textit{$\mathrm{g}$-formula} in $x$ is an $\L^{\Acg, +}_{\preccurlyeq,D}$-formula
$$\theta(x)\ :=\ \theta' \big(|\tau_1(x)|,\ldots, |\tau_k(x)|\big) $$
with $\theta'(u_1,\ldots,u_k)$ a $\mathrm{g}$-relative formula  and $\L^{A, +}_{\preccurlyeq,D}$-terms
$\tau_1(x),\ldots, \tau_k(x)$.

\medskip\noindent
Let $\k$ be the residue field of the $\Acg$-field $\K=(K, C,G)$, and $\Gamma$ its value group.
Let $a\in K^m$, let $\tp^{\K}(a)$ denote the $\L^{\Acg}_{\preccurlyeq,D}$-type realized by $a$ in
$\K$, that is, the set of $\L^{\Acg}_{\preccurlyeq,D}$-formulas  $\phi(x)$ such that $\K\models \phi(a)$.  Also,
$$\tp^{\K}_{\mathrm{c}}(a)\ :=\ \{\psi(x):\ \K^+\models \psi(a)\}, \qquad \tp^{\K}_{\mathrm{g}}(a)\ :=\ \{\theta(x):\ \K^+\models \theta(a)\},$$
where $\psi(x)$ ranges over $\mathrm{c}$-formulas in $x$, and $\theta(x)$ over $\mathrm{g}$-formulas in $x$. 

 Next, let $\K'=(K', C', G')$ also be an $\Acg$-field whose valuation $A$-ring is viable,  and let $a'\in K'^m$.

\begin{lemma}\label{sepvar}
Suppose $\operatorname{char} \k=0$. If $\tp^{\K}_{\mathrm{c}}(a)=\tp^{\K'}_{\mathrm{c}}(a')$ and $\tp^{\K}_{\mathrm{g}}(a)=\tp^{\K'}_{\mathrm{g}}(a')$, then $\tp^{\K}(a)= \tp^{\K'}(a')$.
\end{lemma}

\begin{proof}  
 As at the beginning of this subsection we have the
 good substructures $\K_{0|a}$ of $\K$, and $\K'_{0|a'}$ of $\K'$. 
 Assume  $\tp^{\K}_{\mathrm{c}}(a)= \tp^{\K'}_{\mathrm{c}}(a')$ and $\tp^{\K}_{\mathrm{g}}(a)=\tp^{\K'}_{\mathrm{g}}(a')$. 
 Then for any $\L_{\preccurlyeq,D}^{A, +}$-term $\tau(x)$, $\tau^{\K}(a)=0\Leftrightarrow \tau^{\K'}(a')=0$, since it follows from $\tp^{\K}_{\mathrm{c}}(a)= \tp^{\K'}_{\mathrm{c}}(a')$ that
$$\tau^{\K}(a)=0\ \Leftrightarrow\ \co(\tau)^{\K}(a)=0\ \Leftrightarrow\ \co(\tau)^{\K'}(a')=0\ \Leftrightarrow\ \tau^{\K'}(a')=0.$$ We also have for $f,h\in K$ the equivalences
$$f\preccurlyeq h\ \Leftrightarrow\ f=hD(f,h),\quad f\in C\ \Leftrightarrow\ \co(f)=f,\quad f\in G\ \Leftrightarrow\ f\ne 0 \text{ and }|f|=f,$$ 
so Lemma~\ref{lem10.0}(ii) gives
a unique $\L^{\Acg+}_{\preccurlyeq,D}$-isomorphism $\sigma: \K_{0|a}^+\to \K'^+_{0|a'}$ such that $\sigma(a_i)=a_i'$ for $i=1,\dots, m$; in fact, 
 $\sigma\big(\tau^\K(a)\big)=\tau^{\K'}(a')$ for every $\L_{\preccurlyeq,D}^{A, +}$-term $\tau(x)$. 
  From $\tp^{\K}_{\mathrm{c}}(a)=\tp^{\K'}_{\mathrm{c}}(a')$ and $\tp^{\K}_{\mathrm{g}}(a)=\tp^{\K'}_{\mathrm{g}}(a)$ it now follows  that $\sigma: \K_{0|a}\to \K'_{0|a'}$ is a good map.  
 Hence $\tp^{\K}(a)= \tp^{\K'}(a')$ by Theorem~\ref{AKE}.
\end{proof}

\medskip\noindent
Let $T_{A}$  be the $\L^{\Acg}_{\preccurlyeq,D}$-theory whose models are the $\Acg$-fields of equicharacteristic 0 whose valuation $A$-ring is viable. (Recall that the viability of the valuation $A$-ring $R$ of $\K$ means that $\smallo(R)=tR$ for some $t\in R^{\ne}$ with $t\in \smallo(A)R$;
as $\smallo(A)$ is finitely generated, this is a first-order condition.) 
Let $T^+_{A}$ be the extension by definitions of $T_{A}$ whose models are the expansions
$\K^+$ of models $\K$ of $T_{A}$. 

\begin{corollary}
Every $\L^{\Acg}_{\preccurlyeq,D}$-formula $\phi(x)$ is $T^+_{A}$-equivalent to 
$$ \big( \psi_1(x)\wedge \theta_1(x)\big) \vee \cdots \vee \big( \psi_N(x)\wedge \theta_N(x)\big)$$
for some $N\in \N$, $\mathrm{c}$-formulas $\psi_1(x), \ldots, \psi_N(x)$, and $\mathrm{g}$-formulas $\theta_1(x), \ldots, \theta_N(x)$.
\end{corollary}
\begin{proof}
Apply Lemma~\ref{sepvar} in conjunction with  \cite[Lemmas 5.13, 5.14]{Lou}.
\end{proof}

\bibliographystyle{amsplain}

\bibliography{AAKEpaper}

\end{document}